\documentclass[aos,preprint]{imsart}
\usepackage[utf8]{inputenc}
\usepackage{amsmath,amssymb,amsthm,bbm,hyperref}
\usepackage{color}
\usepackage{appendix}
\usepackage{algorithm}
\usepackage{algorithmic}
\usepackage{float}
\usepackage{comment}
\usepackage{booktabs}
\usepackage{graphicx}
\usepackage{comment}
\newtheorem{thm}{Theorem}
\newtheorem{lem}{Lemma}
\newtheorem{cor}{Corollary}
\newtheorem{defi}{Definition}
\newtheorem{rem}{Remark}
\newtheorem{exa}{Example}
\newtheorem{ass}{Assumption}

\newcommand{\bSigma}{\boldsymbol{\Sigma}}

\newcommand{\bx}{\boldsymbol{x}}

\def\wh{\widehat}

\DeclareMathOperator*{\argmax}{arg\,max}
\DeclareMathOperator*{\argmin}{arg\,min}
\newcommand{\bX}{\boldsymbol{X}}
\newcommand{\bY}{\boldsymbol{Y}}
\newcommand{\bZ}{\boldsymbol{Z}}
\newcommand{\bN}{\boldsymbol{N}}
\newcommand{\bH}{\boldsymbol{H}}
\newcommand{\bL}{\boldsymbol{L}}
\newcommand{\bC}{\boldsymbol{C}}
\newcommand{\bM}{\boldsymbol{M}}
\newcommand{\oX}{\overline{\boldsymbol{X}}}
\newcommand{\oH}{\overline{\boldsymbol{H}}}
\newcommand{\oW}{\overline{\boldsymbol{W}}}
\newcommand{\oSigma}{\overline{\boldsymbol{\Sigma}}}
\newcommand{\bW}{\boldsymbol{W}}
\newcommand{\bTheta}{\boldsymbol{\Theta}}
\newcommand{\bOmega}{\boldsymbol{\Omega}}
\newcommand{\bomega}{\boldsymbol{\omega}}
\newcommand{\bPi}{\boldsymbol{\Pi}}
\newcommand{\bP}{\boldsymbol{P}}
\newcommand{\bDelta}{\boldsymbol{\Delta}}
\newcommand{\bu}{\boldsymbol{u}}
\newcommand{\bv}{\boldsymbol{v}}
\newcommand{\oU}{\overline{\boldsymbol{U}}}
\newcommand{\oV}{\overline{\boldsymbol{V}}}
\newcommand{\bR}{\boldsymbol{R}}
\newcommand{\btheta}{\boldsymbol{\theta}}
\newcommand{\bI}{\boldsymbol{I}}

\newcommand{\oLambda}{\overline{\boldsymbol{\Lambda}}}
\newcommand{\minsigma}{\sigma_{\textbf{min}}}
\newcommand{\maxsigma}{\sigma_{\textbf{max}}}
\newcommand{\bD}{\boldsymbol{D}}
\newcommand{\br}{\boldsymbol{r}}
\newcommand{\bw}{\boldsymbol{w}}
\newcommand{\bpi}{\boldsymbol{\pi}}
\newcommand{\ba}{\boldsymbol{a}}
\newcommand{\bb}{\boldsymbol{b}}
\newcommand{\bc}{\boldsymbol{c}}
\newcommand{\VV}{\mathbb{V}}
\newcommand{\bB}{\boldsymbol{B}}

\allowdisplaybreaks

\begin{document} 
\begin{frontmatter}

\title{Inferences on mixing probabilities and ranking in mixed-membership models\thanks{The research is supported in part by ONR grant N00014-22-1-2340,  NSF grants  DMS-2052926, DMS-2053832 and DMS-2210833.}}

\runtitle{Inferences on Mixing Probabilities and Ranking}
\begin{aug}
\author[A]{\fnms{Sohom}~\snm{Bhattacharya}{}},
\author[B]{\fnms{Jianqing}~\snm{Fan}{}}
\and
\author[B]{\fnms{Jikai}~\snm{Hou}{}}
\address[A]{Department of Statistics,
University of Florida\printead[presep={,\ }]{}}
\address[B]{Department of Operations Research and Financial Engineering,
Princeton University\printead[presep={.\ }]{}}
\end{aug}
\date{}

\begin{abstract}
Network data is prevalent in numerous big data applications including economics and health networks where it is of prime importance to understand the latent structure of network. In this paper, we model the network using the Degree-Corrected Mixed Membership
(DCMM) model. In DCMM model, for each node $i$, there exists a membership vector
$\bpi_ i = (\bpi_i(1), \bpi_i(2),\ldots, \bpi_i(K))$, where $\bpi_i(k)$ denotes the weight that node $i$ puts in community $k$. We derive novel finite-sample expansion for the $\bpi_i(k)$s which allows us to obtain asymptotic distributions and confidence interval of the membership mixing probabilities and other related population quantities.  This fills an important gap on uncertainty quantification on the membership profile.
We further develop a ranking scheme of the vertices based on the membership mixing probabilities on certain communities and perform relevant statistical inferences.  A multiplier bootstrap method is proposed for ranking inference of individual member's profile with respect to a given community.
The validity of our theoretical results is further demonstrated by via numerical experiments in both real and synthetic data examples.
\end{abstract}

\end{frontmatter}

\maketitle

\section{Introduction} 
In various fields of study, such as citation networks, protein interactions, health, finance, trade, and social networks, we often come across large amounts of data that describe the relationships between objects. There are numerous approaches to understand and analyze such network data. Algorithmic methods are commonly used to optimize specific criteria as shown in \cite{newman2013community,newman2013spectral} and \cite{zhang2014scalable}. Alternatively, model-based methods rely on probabilistic models with specific structures, which are reviewed by \cite{goldenberg2010survey}. One of the earliest models where the nodes (or vertices) of the network belong  to some latent community is Stochastic Block Model (SBM)~\cite{holland1983stochastic,wang1987stochastic,abbe2017community}. Several improvements have been proposed over this model to overcome its limitations, two of them are relevant in our paper. First, \cite{karrer2011stochastic} introduced degree-corrected SBM, where a degree parameter is used for each vertices to make the expected degrees match the observed ones. Second, \cite{airoldi2008mixed, airoldi2014handbook} study mixed membership model where each individual can belong to several communities with a mixing probability profile. In this paper, we study membership profiles in Degree-Corrected Mixed Membership (DCMM) model, which combines both the above benefits. In DCMM, every node $i$ is assumed to have a community membership probability vector $\bpi_i \in \mathbb{R}^K$, where $K$ is the number of communities and the $k$-th entry of $\bpi_i$ specifies the mixture proportion of node $i$ in community $k$ (see \cite{fan2022asymptotic,jin2017estimating}. For example, a newspaper can be $40\%$ conservative and $60\%$ liberal. In addition, each node is allowed to have its own degree. \par
Given such a network, estimation and inference of membership profiles has drawn some attention recently. For example, \cite{jin2017estimating} provides an algorithm to estimate the $\bpi_i$'s and \cite{fan2022simple,fan2022simple-rc} considers the hypothesis testing problem that two nodes have same membership profiles. However, 
they avoid the problems of uncertainty quantification of $\bpi_i(k)$.   In addition, to our best knowledge, none of the prior works concern with the problem of ranking nodes in a particular profile. As an example, one might consider asking the question: Is newspaper A more liberal than newspaper B? Or how many newspapers should I pick to ensure the top-K conservative newspapers are selected? Such a ranking question has applications in finances where one might be interested in knowing whether a particular stock is in top $K$ technology-stocks before investing in it. In our work, we device a framework to perform ranking inference based on $\bpi_i$s.\par 

Our work lies in the intersection of three research directions which we delineate here. \begin{enumerate}
    \item Community detection: Our estimation and inference procedure crucially rely on spectral clustering, which is one of the oldest methods of community detection (cf. \cite{von2007tutorial} for a tutorial). In the last decade, \cite{rohe2011spectral,jin2015score,lei2015clustering} have developed both the theory and methods of spectral clustering. Other line of research related to community detection involves showing optimality of detection boundary \cite{abbe2017community} or link-prediction problem \cite{noweu2007link,li2023link}. Recently, \cite{jin2017estimating} has developed an algorithm to estimate $\pi_i$s in $l_2$ norm, however it lacks any inferential guarantees and asymptotic distributions. Furthermore, there has been significant number of works about hypothesis testing in network data. \cite{castro2014dense,verzelen2015sparse} formulated community detection  as a hypothesis detection problem. \cite{fan2022simple, fan2022simple-rc} studies the testing problem of whether two vertices have same membership profiles in DCMM. Under stochastic block models, detection of the number of blocks has been studied by \cite{bickel2016hypothesis,lei2016goodness,wang2017likelihood} among others.
    
    \item Ranking inference: Most of the literature about ranking problems  deals with pairwise comparisons or multiple partial ranking models, like Bradley-Terry-Luce and other assortative network models. Prominent examples of ranking involves individual choices in economics~\cite{luce2012individual,mcfadden1973conditional}, websites~\cite{dwork2001rank}, ranking of journals~\cite{ji2021meta,stigler1994citation}, alleles in genetics~\cite{sham1995extended}.Hence, the ranking problem has been extensively studied in statistics, machine learning, and operations research, see, for example, \cite{hunter2004mm,chen2015spectral,chen2019spectral,chen2021spectral,gao2023uncertainty,fan2022ranking} for more details. However, none of the above work is concerned with ranking in DCMM model and hence, our work is significantly different from the aforementioned papers.
    
    \item $l_\infty$ and $l_{2,\infty}$ perturbation theory: Often, for uncertainty quantification of unknown parameters, it is not enough to obtain an $l_2$ error bound on the estimators, rather one needs a leave-one-out style analysis to get more refined  $l_\infty$ and $l_{2,\infty}$ error bounds. See \cite[Chapter 4]{chen2021spectral} for an introduction. Such analysis has been used in matrix completion\cite{chen2019inference}, principal component analysis \cite{yan2021inference}, ranking analysis\cite{fan2022ranking,fan2022uncertainty,gao2023uncertainty}. We develop novel subspace perturbation theory to obtain novel finite sample expansions of individual $\bpi_i(k)$s and use it to obtain asymptotic distributions.
\end{enumerate}
To perform inference about ranks and hypothesis testing, we employ an inference framework for ranking items through maximum pairwise difference statistic whose distribution is approximated by a newly proposed multiplier bootstrap method. A similar framework is recently introduced by \cite{fan2022ranking} in the context of Bradley-Terry-Luce model. \par
The rest of the paper is structured as follows. Section \ref{sec:problem} formulates the problem and describes the estimation procedure. Section \ref{sec:vertex_hunting} and Section \ref{sec:membership_reconstruction} delineates the vertex hunting and membership reconstruction steps of our estimation procedure respectively. Using the results we established, we develop some distribution theory and answer inference questions in Section \ref{distributionaltheory}. We complement our theoretical findings with numerical experiments in Section \ref{sec:simulation} where we perform simulations on both synthetic and real datasets.  Section \ref{sec:proof_outline} provides brief outline of the major proofs of our paper. Finally, Section \ref{sec:appendix} contains all the proofs.

\section{Problem Formulation}\label{sec:problem}
\subsection{Model setting}

Consider an undirected graph $G=(V,E)$ on $n$ nodes (or vertices), where $V=[n]:=\{1,2,\ldots,n\}$ is the set of nodes and $E \subseteq [n]\times[n]$ denotes the set of edges or links between the nodes. Given such a graph $G$, consider its symmetric adjacency matrix $X= \mathbb{R}^{n \times n}$ which captures the connectivity structure of $X$, namely $x_{ij}=1$ if there exists a link or edge between the nodes $i$ and $j$, i.e., $(i,j) \in E$ and $x_{ij}=0$ otherwise. We allow the presence of potential self-loops in the graph $G$. If there exists no self-loops, we let $x_{ii}=0$ for $i \in [n]$. Under a probabilistic model, we will assume that $x_{ij}$ is an independent realization from a Bernoulli random variable for all upper triangular entries of random matrix $X$.

 Given a network on $n$ nodes, we assume that there is an underlying latent community structure which contains $K$ communities. Each node $i$ is associated with a community membership probability vector $\bpi_i = (\bpi_i(1),\bpi_i(2),\dots,\bpi_i(K))\in \mathbb{R}^K$ such that 
\begin{align*}
    \mathbb{P}(\text{node } i \text{ belongs to community } k) = \bpi_i(k), \quad k = 1,2,\dots,K.
\end{align*}
A node is $i \in [n]$ is called a \textit{pure node} if there exists $k \in [K]$ such that $\bpi_i(k)=1$. The degree-corrected mixed membership (DCMM) model assumes (cf.\cite{jin2017estimating}) that the probability that an edge exists between node $i$ and node $j$ is given by
\begin{align}\label{eq:dcmm_define}
    \mathbb{P}(\text{edge exists between node }i\text{ and node }j) = \theta_i\theta_j\sum_{k=1}^K\sum_{l=1}^K \bpi_i(k)\bpi_j(l)p_{kl}.
\end{align}
Here, $\theta_i>0$ captures the degree heterogeneity for node $i\in [n]$, and $p_{kl}>0$ can be viewed as the probability of a typical member in community $k$  ($\theta_i=1$, say) connects with a typical member in community $l$ ($\theta_j=1$, say). The mixture probabilities can be written in the following matrix form as follows. Let $\bTheta = \textbf{diag}(\theta_1,\theta_2,\dots,\theta_n)$ be a diagonal matrix that captures the degree heterogeneity, $\bPi = (\bpi_1,\bpi_2,\dots,\bpi_n)^\top\in \mathbb{R}^{n\times K}$ be the matrix of community membership probability vectors, and $\bP = (p_{kl})\in \mathbb{R}^{K\times K}$ be a nonsingular matrix with $p_{kl}\in [0,1]$. Then, the mixing probability matrix can be expressed as
\begin{align}\label{eq:h_define}
    \bH = \bTheta\bPi\bP\bPi^\top\bTheta.
\end{align}

Let $\bX \in \mathbb{R}^{n\times n}$ be the symmetric adjacency matrix of these $n$ nodes, i.e., $X_{ij} = 1$ if there is a link connecting nodes $i$ and $j$, and $X_{ij} = 0$ otherwise.   Note that $\mathbb{E}(X_{ij})=H_{ij}$  for $i \not = j$.  We set $X_{ii}=0$ for the case without self-loop.  We also allow the case with loop.  In that case, we assume that $P(X_{ii}=1) = H_{ii}$.  In both cases, we write

\begin{equation}\label{eq:define_h}
   \bX=\bH+\bW,
\end{equation}
where $\bW$ is a symmetric random matrix with mean zero and independent entries above  the diagonal and on the diagonal for the case with self-loop.
In the following, since our theory applies to both cases, we will not distinguish between them. Therefore, the notation \eqref{eq:define_h} is used throughout the article.

Our goal is to study the community membership probability matrix $\bPi$ and conduct inference on its entries. Based on their uncertainty quantifications, we provide a framework for ranking inference and perform some hypothesis testing problems on the rank of each node's mixing probability on a community. To this end, we need to first estimate $\bPi$ using Mixed-SCORE algorithm \cite{jin2017estimating}. We invoke a slight modification of the algorithm along with the $l_2-l_\infty$ perturbation theory to get desired entry-wise expansion. 

We impose the following identifiability condition for the DCMM model \eqref{eq:dcmm_define}.
 
\begin{ass}\label{assn:pure_node}
    Each community $\mathcal{C}_k$, $1\leq k\leq K$ has at least one pure node, namely, there exists a vertex $i \in [n]$ such that $\bpi_i(k) = 1$.
\end{ass}

\subsection{Estimation procedure}
We describe the version of the Mixed-SCORE algorithm (cf. \cite{jin2017estimating}) which will be used to estimate $\bPi$. The algorithm consists of three key steps. First, we map each node to a $(K-1)$-dimension space using observed $\bX$. Ideally, in absence of noise, i.e., $\bW = 0$, these $n$ points in the $(K-1)$-dimension space would form a simplex with $K$ vertices, with the pure nodes defined by Assumption \ref{assn:pure_node} becoming the vertices. Presence of such a simplex structure has been discussed by Lemma 2.1 of \cite {jin2017estimating}. 

In the presence of noise, as long as the noise level is mild, we can still hope they are approximately a simplex. So, we apply a vertex hunting algorithm to estimate these $K$ vertices. After that, we estimate the membership vector $\bpi_i$ for each node based on the estimated vertices. Below, we describe the procedure mathematically.
\begin{itemize}
    \item \textbf{SCORE step}: Let $\wh \lambda_1,\wh\lambda_2,\dots,\wh \lambda_K$ be the largest $K$ eigenvalues (in magnitude) of $\bX$, sorted in descending order. Let $\wh\bu_1,\wh\bu_2,\dots, \wh\bu_K$ be the corresponding eigenvectors. Calculate the following $n$ vectors
    \begin{align}\label{eq:define_hat_r}
        \wh\br_i := \left[\frac{(\wh\bu_2)_i}{(\wh\bu_1)_i},\frac{(\wh\bu_3)_i}{(\wh\bu_1)_i},\dots, \frac{(\wh\bu_K)_i}{(\wh\bu_1)_i}\right]^\top\in \mathbb{R}^{K-1},\quad\forall i\in [n].
    \end{align}
    \item \textbf{Vertex Hunting step}: Apply vertex hunting algorithm (see Secion \ref{sec:vertex_hunting} for details) to $\wh\br_1,\wh\br_2,\dots,\wh\br_n$ and get estimated vertices $\wh\bb_1,\wh\bb_2,\dots,\wh\bb_K \in \mathbb{R}^{K-1}$.
    
    \item \textbf{Membership Reconstruction step}: 
    For each $i\in [n]$, let $\wh \ba_i= (\wh a_i(1), \ldots, \wh a_i(K))\in \mathbb{R}^K$ be the unique solution of the linear equations: 
    \begin{align}\label{eq:define_ai}
        \wh\br_i = \sum_{k=1}^K\wh a_i(k)\wh\bb_k, \qquad \sum_{k=1}^K\wh a_i(k)=1.
    \end{align}
    Define a vector $\wh \bpi'_i \in \mathbb{R}^K$ with $\bpi'_i(k) =  \wh a_i(k)/\left[\wh\bc\right]_k $, where
    \begin{align}\label{eq:define_hat_c}
        \left[\wh\bc\right]_k =\left[\wh\lambda_1+\wh\bb_k^\top\textbf{diag}(\wh\lambda_2,\wh\lambda_3,\dots,\wh\lambda_K)\wh\bb_k\right]^{-1/2}, \quad 1\leq k \leq K. 
    \end{align}
    Estimate $\bpi_i$ as $\wh\bpi_i = \bpi'_i / \sum_{k=1}^K\bpi_i'(k)$.
\end{itemize}
To analyze the algorithm, we begin by analyzing the SCORE step, i.e., investigate the $\wh\br_i$'s.

We introduce some notations first. Since $\bH$ and $\bX$ are symmetric, we write their eigen-decomposition as
\begin{align*}
    \bH &= \sum_{i=1}^n \lambda_i^*\bu_i^*\bu_i^{*\top} = \sum_{i=1}^n |\lambda_i^*|\bu_i^*\left(\textbf{sgn}\left(\lambda_i^*\right)\bu_i^*\right)^{\top}, \\
    \bX &= \sum_{i=1}^n\wh\lambda_i\wh\bu_i\wh\bu_i^{\top} = \sum_{i=1}^n |\wh\lambda_i|\wh\bu_i\left(\textbf{sgn}\left(\wh\lambda_i\right)\wh\bu_i\right)^{\top}.
\end{align*}
The eigenvalues are sorted in the following way. First, $\lambda_1^*,\lambda_2^*,\dots, \lambda_K^*$ are the largest $K$ eigenvalues of $\bH$ in magnitude, and $\wh\lambda_1,\wh\lambda_2,\dots, \wh\lambda_K$ are the largest $K$ eigenvalues of $\bX$ in magnitude. Second, $\lambda_1^*>\lambda_2^*>\cdots> \lambda_K^*$, $\wh\lambda_1>\wh\lambda_2\cdots>\wh\lambda_K$ are sorted descendingly, while the other eigenvalues can be sorted in any order.
By~Lemma C.3 of \cite{jin2017estimating}, we can choose the sign of $\bu_1^*$ to make sure $(\bu_1^*)_i >0$, $1\leq i\leq n$. The direction of $\wh\bu_1$ are chosen to make sure $\wh\bu_1^\top\bu_1^*\geq 0$. Define
\begin{align}\label{eq:define_vi}
    \sigma_i^* := \left|\lambda_i^*\right|, \qquad \bv_i^* := \textbf{sgn}\left(\lambda_i^*\right)\bu_i^*, \qquad \wh\sigma_i := |\wh\lambda_i|, \qquad \wh\bv_i := \textbf{sgn}(\wh\lambda_i)\wh\bu_i. 
\end{align}
We further define,
\begin{align}\label{eq:define_ubar}
    \oU^* = [\bu_2^*,\bu_3^*,\dots,\bu_K^*]\in \mathbb{R}^{n\times (K-1)} \quad\text{and} \quad \oU = [\wh\bu_2,\wh\bu_3,\dots,\wh\bu_K]\in \mathbb{R}^{n\times (K-1)}.
\end{align} We also denote by 
\begin{align}
    \br_i^* := \left[\frac{(\bu_2^*)_i}{(\bu_1^*)_i},\frac{(\bu_3^*)_i}{(\bu_1^*)_i},\dots, \frac{(\bu_K^*)_i}{(\bu_1^*)_i}\right]^\top\in \mathbb{R}^{K-1},\quad\forall i\in [n].\label{eq:define_r_star}
\end{align}
Note from the expression of $\br_i^*$ and $\wh\br_i$, in order to analyze $\wh\br_i$, we need to study the difference between $\bu_1^*$ and $\wh\bu_1$, and the difference between $\oU^*$ and $\oU$. To this end, we need the following four assumptions for the theoretical analysis, introduced in Section 3 of \cite{jin2017estimating}, which are necessary for the Mixed-SCORE algorithm to work.\par 
Denote by $\boldsymbol{G}: = K \left\|\btheta\right\|_2^{-2}(\bPi^\top\bTheta^2\bPi)\in \mathbb{R}^{K\times K}$ and $\theta_{\max} := \max_{i\in [n]}\theta_i$. Note that, the eigenvalues of $\bP\boldsymbol{G}$ are real, since $\boldsymbol{G}$ is positive-definite.

\begin{ass}\label{assn:theta_order}
There exist constants $C, C', C''>0$ such that 
\begin{align*}
   \theta_{\text{max}}\leq C\min_{i\in [n]}\theta_i\quad \text{and}\quad  C'\sqrt{\frac{\log n}{n}}\leq \theta_{\text{max}}\leq C''.
\end{align*}
\end{ass}
\begin{ass}\label{assn:p_max}
     There exists a constant $C>0$ such that 
    \begin{align*}
        \left\|\bP\right\|_{\text{max}}\leq C,\quad \left\|\boldsymbol{G}\right\|\leq C,\quad \left\|\boldsymbol{G}^{-1}\right\|\leq C.
    \end{align*}
\end{ass}

\begin{ass}\label{assn:lambda_pg}
     There exists $c_1>0$ such that $|\lambda_2(\bP\boldsymbol{G})|\leq (1-c_1) \lambda_1(\bP\boldsymbol{G})$ and $c_1\beta_n \leq |\lambda_K(\bP\boldsymbol{G})|\leq |\lambda_2(\bP\boldsymbol{G})|\leq c_1^{-1}\beta_n$, where $\{\beta_n\}_{n=1}^\infty$ is a sequence of positive real number such that $\beta_n\leq 1$, and $\lambda_i(\bP\boldsymbol{G})$ is the $i$-th largest right eigenvalue of $\bP\boldsymbol{G}$. 
\end{ass}

\begin{ass}\label{assn:eta_ratio}
    There exists a constant $C>0$ such that $\min_{1\leq k\leq K}\boldsymbol{\eta}_1(k) >0$ and 
    \begin{align*}
        \frac{\max_{1\leq k\leq K}\boldsymbol{\eta}_1(k)}{\min_{1\leq k\leq K}\boldsymbol{\eta}_1(k)}\leq C.
    \end{align*}
    Here $\boldsymbol{\eta}_1$ is the right eigenvector corresponding to $\lambda_1(\bP\boldsymbol{G})$.
\end{ass}

 Here, Assumption \ref{assn:theta_order} and Assumption \ref{assn:p_max} ensure that the underlying model is not spiky. In other words, the signal spread across all nodes. Assumption \ref{assn:lambda_pg} is the eigen-gap assumption, and as we will see next, this assumption ensures that there is a sufficient gap between the first eigenvalue and the remaining eigenvalues of $\bH$, and the remaining non-zeros eigenvalues of $\bH$ are of the same order. Assumption \ref{assn:eta_ratio} seems to be less straightforward, but it actually satisfied by a very wide range of models (See \cite{jin2017estimating} for examples). 
We will make Assumption \ref{assn:theta_order}-\ref{assn:eta_ratio} for the remainder of the paper. 

With these assumptions in hand, the following lemma lists some important properties of the eigenvalues of $\bH$ and $\bX$.

\begin{lem}\label{eigenvaluelemma}
(Lemma C.2 of \cite{jin2017estimating}) Let $\btheta = [\theta_1, \ldots, \theta_n]^\top$ and $\beta_n$s are defined by  Assumption \ref{assn:lambda_pg}. 
Under Assumption \ref{assn:theta_order}-\ref{assn:eta_ratio}, we have the following statements
\begin{itemize}
    \item $C_1^{-1}K^{-1}\left\|\btheta\right\|_2^2\leq\lambda_1^*\leq C_1\left\|\btheta\right\|_2^2$. If $\beta_n = o(1)$, then $\lambda_1^*\asymp\left\|\btheta\right\|_2^2$.
    \item $\lambda_1^*-\left|\lambda_i^*\right|\asymp \lambda_1^*$, for $2\leq i\leq K$.
    \item $|\lambda_i^*|\asymp \beta_n K^{-1}\left\|\btheta\right\|_2^2$, for $2\leq i\leq K$.
\end{itemize}
\end{lem}
 Combine this lemma with the fact that $\left\|\bW\right\|\lesssim \sqrt{n}\theta_{\max}$ (will be shown later in Lemma \ref{Wspectral}) with high probability, we know that the eigenvalues of $\bH$ and $\bX$ can be divided into three group as long as $\sqrt{n}\theta_{\max}\ll \beta_n K^{-1}\left\|\btheta\right\|_2^2$.


\begin{align*}
    &\{\lambda_1^*\}\gtrsim \{|\lambda_2^*|,|\lambda_3^*|,\dots, |\lambda_K^*|\}\gg \{\lambda_{K+1}^*,\lambda_{K+2}^*,\dots, \lambda_n^*\} ,\\
    &\{\wh\lambda_1\}\gtrsim \{|\wh\lambda_2|,|\wh\lambda_3|,\dots, |\wh\lambda_K|\}\gg \{\wh\lambda_{K+1},\wh\lambda_{K+2},\dots, \wh\lambda_n\}.
\end{align*}
This also motivates us to derive two results separately. First, we obtain the expansion of $\wh\bu_1-\bu_1^*$. Second, we analyze the difference between $\oU^*$ and $\oU$ by writing the expansion of $\oU\bR - \oU^*$, where 
\begin{align}
    \bR := \argmin_{\boldsymbol{O}\in \mathcal{O}^{(K-1)\times (K-1)}}\left\|\oU\boldsymbol{O}-\oU^*\right\|_F  \label{Rdefinition}
\end{align} is an orthogonal matrix which best ``matches'' $\oU$ and $\oU^*$. Here $\mathcal{O}^{(K-1)\times (K-1)}$ stands for the set of all the $(K-1)\times (K-1)$ orthogonal matrices. \par
The analysis of $\oU\bR - \oU^*$ is similar to a matrix denoising problem, so we define some quantities which are commonly used in matrix denoising literature \cite{yan2021inference}. We define
\begin{align}\label{eq:define_kappa}
    &\maxsigma^* = \max_{2\leq i\leq K}|\lambda_i^*|, \quad \minsigma^* = \min_{2\leq i\leq K}|\lambda_i^*|, \quad \kappa^* = \maxsigma^*/\minsigma^*, \nonumber \\ &\wh\sigma_{\textbf{max}} := \max_{2\leq i\leq K}|\wh\lambda_i|, \quad \wh\sigma_{\textbf{min}} = \min_{2\leq i\leq K}|\wh\lambda_i|.
\end{align}
Our results regarding uncertainty quantification contributes to the literature of estimation of subspace spanned by partial eigenvectors, cf. \cite{abbe2020entrywise,abbe2022}. \par
To state our results, we state an incoherence condition on the matrix $\bH$, which is a natural modification of the standard incoherence assumption \cite[Assumption 2]{yan2021inference} to adapt to our setting by separating the first eigenspace from the second to $K$-th eigenspace.
\begin{defi}\label{incorherenceassumption}
(\textbf{Incoherence}) The symmetric matrix $\bH$ is said to be $\mu^*$-incoherent if
\begin{align}\label{eq:incoherence}
    \left\|\bu_1^*\right\|_{\infty}
      \leq  \sqrt{\frac{\mu^*}{n}}, \quad\mbox{and} \quad \left\|\oU^*\right\|_{2, \infty} \leq  \sqrt{\frac{\mu^*(K-1)}{n}}.
\end{align}
\end{defi}
Note that, unlike \cite{yan2021inference} we do not indeed to assume an incoherence assumption in this article. This is because a combination of Assumption \ref{assn:theta_order} and Assumption \ref{assn:p_max} shows that $\bH$ is $\mu^\star$ incoherent with $\mu^*\asymp 1$, see Remark \ref{rem1} and Section \ref{IncoherenceExplanation} for more details. Also Lemma \ref{eigenvaluelemma} implies $\kappa^*\asymp 1$. However, the quantities $\mu^*$ and $\kappa^*$ are two key parameters in the literature of matrix denoising, and we keep them in our results.

\begin{rem}\label{rem1}
By Lemma C.3 of \cite{jin2017estimating}, we obtain $(\bu_1^*)_i\asymp \theta_i/\|\btheta\|_2$ for $1\leq i\leq [n]$. Therefore, an incoherence condition on $\bu_1^*$ actually can be interpreted as the incoherence condition on $\btheta$, which is implied by Assumption \ref{assn:theta_order}.
\end{rem}

Now, we are ready to state our results. We first state the following bound on $\wh\bu_1-\bu_1^*$ whose proof is deferred. 
\begin{thm}\label{mainthmu1}
If $\max\{\sqrt{n}\theta_{\text{max}}, \log n \}\ll \lambda_1^*$, with probability at least $1-O(n^{-10})$,
\begin{align*}
    \left\|\wh\bu_1-\bu_1^*\right\|_\infty\lesssim \frac{K\log^{0.5} n+K^{1.5}\sqrt{\mu^*}}{n\theta_{\text{max}}}+\frac{K\sqrt{\mu^*}\log n}{n^{1.5}\theta_{\text{max}}^2}+\frac{K^3}{n^{1.5}\theta_{\text{max}}^3},
\end{align*}
where $\mu^\star$ is given by Definition \ref{incorherenceassumption}. Moreover, we have the following expansion
\begin{align*}
    \wh\bu_1-\bu_1^* = \sum_{i=2}^n \frac{\bu_i^{*\top}\bW\bu_1^*}{\lambda_1^*-\lambda_i^*}\bu_i^*+\boldsymbol{\delta},
\end{align*}
where with probability at at least $1-O(n^{-10})$ we have $\|\boldsymbol{\delta}\|_2\lesssim K^2 / n\theta_{\text{max}}^2$ and 
\begin{align*}
\left\|\boldsymbol{\delta}\right\|_\infty\lesssim & \frac{K^{2.5}\sqrt{\mu^*}+K^2\log^{0.5}n}{n^{1.5}\theta_{\text{max}}^2}+\frac{K^3}{n^{1.5}\theta_{\text{max}}^3}+\frac{K^2\log^{1.5} n+K^{2.5}\sqrt{\mu^*}}{n^{2}\theta_{\text{max}}^3} +\frac{\log^2n\sqrt{\mu^*}}{n^{2.5}\theta_{\text{max}}^4}.
\end{align*}
Furthermore, if $K, \mu^*, \theta_{\max} \asymp 1$, we obtain  
\begin{align*}
    \left\|\wh\bu_1-\bu_1^*\right\|_\infty\lesssim \frac{\log^{0.5}n}{n}\quad \text{ and }\quad \left\|\boldsymbol{\delta}\right\|_\infty \lesssim \frac{\log^2 n}{n^{1.5}}.
\end{align*}

\end{thm}
\begin{proof}
Combine Lemma \ref{eigenvaluelemma} with Theorem \ref{u1expansion}, Theorem \ref{deltainfty} and Lemma \ref{u1differenceinfty}.
\end{proof}

Theorem \ref{mainthmu1} provides an error bound for $\hat{u}_1$ through the quantity $\boldsymbol{\delta}$. Both $l_2$ and $l_\infty$ bounds on $\boldsymbol{\delta}$ are provided. The following result states the expansion for $\oU\bR-\oU^*$. Define
\begin{align}\label{eq:define_eps_zero}
    \varepsilon_0 := &\left(\frac{\kappa^*K^{2.5}\sqrt{\mu^*}}{n^{1.5}\beta_n^2\theta_{\text{max}}^2}+\frac{K^{2.5}\log^{0.5}n}{n^{1.5}\beta_n^2\theta_{\text{max}}^3}+\frac{K^{3.5}}{n^{1.5}\beta_n^2\theta_{\text{max}}^3} +\frac{K^{1.5}\log^{0.5}n}{n\beta_n\theta_{\text{max}}}+\frac{K^2}{n\beta_n\theta_{\text{max}}^2} +\frac{K^{0.5}\sqrt{\mu^*}}{n^{0.5}}\right)  \nonumber\\
    &\cdot \left(\frac{\kappa^*K^2}{n\beta_n^2\theta_{\text{max}}^2}+\frac{K^{1.5}\log^{0.5}n}{n\beta_n\theta_{\text{max}}}\right) +\frac{K^{2.5}\log^{0.5}n}{n^{1.5}\beta_n^2\theta_{\text{max}}^4}+ \frac{K^{3.5}}{n^{1.5}\beta_n^2\theta_{\text{max}}^3}.
\end{align}
Define $\oLambda^* := \textbf{diag}(\lambda_2^*,\lambda_3^*,\dots,\lambda_K^*)$.
\begin{thm}\label{mainthmmatrixdenoising}
Assume that $\max\{\sqrt{n}\theta_{\text{max}},\log n\}\ll \minsigma^*$ and $\sqrt{(K-1)\log n}\theta_{\text{max}}/\minsigma^*+\kappa^* n\theta_{\text{max}}^2/\minsigma^{*2}\ll 1$, where $\kappa^\star$, $\mu^\star$ are given by \eqref{eq:define_kappa} and Definition \ref{incorherenceassumption} respectively. Write
\begin{align*}
    \oU\bR-\oU^* = \left[\bW-\bu_1^*\bu_1^{*\top}\bW\bN\right]\oU^*\left(\oLambda^*\right)^{-1} +\boldsymbol{\Psi}_{\oU},
\end{align*}
where $\bR$ is defined by \eqref{Rdefinition} and
\begin{equation}\label{eq:define_n}
    \bN = \sum_{j=2}^n \lambda_1^*\bu_j^*\bu_j^{*\top}/(\lambda_1^*-\lambda_j^*).
\end{equation} If $n\gtrsim \max\left\{\mu^*\log^2n, K\log n\right\}$, then with probability at least $1-O(n^{-10})$, $\boldsymbol{\Psi}_{\oU}$ satisfies
\begin{align*}
\left\|\boldsymbol{\Psi}_{\oU}\right\|_{2,\infty}\lesssim\varepsilon_0,
\end{align*}
where $\varepsilon_0$ is defined by \eqref{eq:define_eps_zero}. 
If, in addition $K, \mu^*, \kappa^*, \beta_n, \theta_{\max} \asymp 1$ as in 
Theorem \ref{mainthmu1}, we obtain $\left\|\boldsymbol{\Psi}_{\oU}\right\|_{2,\infty}\lesssim \frac{\log^{0.5}n}{n^{1.5}}$.
\end{thm}
\begin{proof}
See Section \ref{mainthmmatrixdenoisingproof}.
\end{proof}

 We are now in a position to state our main result about $\hat{r}_i$s defined by \eqref{eq:define_hat_r}. This is our final result regarding finite sample analysis of SCORE-step. The proof involves both Theorem \ref{mainthmu1} and Theorem \ref{mainthmmatrixdenoising}, and is deferred.
\begin{thm}\label{mainthmrexpansion}
Assume the conditions in Theorem \ref{mainthmmatrixdenoising} hold. In addition, we assume 
\begin{align*}
   \frac{K\log^{0.5} n+K^{1.5}\sqrt{\mu^*}}{n^{0.5}\theta_{\text{max}}}+\frac{K\sqrt{\mu^*}\log n}{n\theta_{\text{max}}^2}+\frac{K^3}{n\theta_{\text{max}}^3}\ll 1.
\end{align*}
Recall the orthogonal matrix $\bR$ defined by~\eqref{Rdefinition} and
\begin{align*}
    \bw_i = \left\{\left[\bW-\bu_1^*\bu_1^{*\top}\bW\bN\right]_{i,\cdot}\oU^*\left(\oLambda^*\right)^{-1}\right\}^\top,\quad \forall i \in [n],
\end{align*}
where the matrix $N$ was defined by \eqref{eq:define_n}.  Then we have the following decomposition
\begin{align*}
    \bR^\top\wh\br_i-\br_i^* = \frac{1+\gamma_i}{(\bu_1^*)_i}\left(\bw_i-\frac{1}{\lambda_1^*}\left[\bN\bW\bu_1^*\right]_i\br_i^*\right)+[\boldsymbol{\Psi}_{\br}]_{i,\cdot}^\top,\quad \forall i\in [n],
\end{align*}
such that with probability at least $1-O(n^{-10})$, for all $i\in [n]$ we have
\begin{align*}
    &|\gamma_i|\lesssim \frac{K\log^{0.5} n+K^{1.5}\sqrt{\mu^*}}{n^{0.5}\theta_{\text{max}}}+\frac{K\sqrt{\mu^*}\log n}{n\theta_{\text{max}}^2}+\frac{K^3}{n\theta_{\text{max}}^3}, \\
    &\left\|[\boldsymbol{\Psi}_{\br}]_{i,\cdot}\right\|_2\lesssim \sqrt{n}\left(\left\|\boldsymbol{\Psi}_{\oU}\right\|_{2,\infty}+\left\|\br_i^*\right\|_2\left\|\boldsymbol{\delta}\right\|_{\infty}\right),
\end{align*}
where $\left\|\boldsymbol{\delta}\right\|_{\infty}$and $\left\|\boldsymbol{\Psi}_{\oU}\right\|_{2,\infty}$ are controlled by Theorem \ref{mainthmu1} and Theorem \ref{mainthmmatrixdenoising}.
Specifically, if $K, \mu^*, \kappa^*, \beta_n, \theta_{\max} \asymp 1$, for all $i\in [n]$ we have  
\begin{align*}
    |\gamma_i|\lesssim \frac{\log^{0.5} n}{n^{0.5}}\quad \text{and}\quad \left\|[\boldsymbol{\Psi}_{\br}]_{i,\cdot}\right\|_2\lesssim \frac{\log^2 n}{n}.
\end{align*}
\end{thm}
\begin{proof}
See Section \ref{mainthmrexpansionproof}.
\end{proof}
\begin{rem}
   
    In terms of the number of communities are allowed, if we assume $\mu^*, \kappa^*, \beta_n, \theta_{\max} \asymp 1$, then the assumptions in Theorem \ref{mainthmu1}-\ref{mainthmrexpansion} require $K\ll n^{1/3}$. When it comes to $\theta_{\max}$, if $K, \mu^*, \kappa^*, \beta_n \asymp 1$, the assumptions in Theorem \ref{mainthmu1}-\ref{mainthmrexpansion} require $\theta_{\text{max}} \gg n^{-1/3}$.

\end{rem}

We finish the section with our result regarding estimation error bound on the $\wh\br_i$s. To this end, we need to define the following quantities:
\begin{align}
    \varepsilon_1 :=& \frac{K\sqrt{\mu^*}}{n^{0.5}\theta_{\text{max}}}+\frac{K^{0.5}\log n\sqrt{\mu^*}}{n^{0.5}\theta_{\text{max}}^2}+\frac{K^{1.5}\log^{0.5}n}{n^{0.5}\beta_n \theta_{\text{max}}}+\frac{K^{1.5}\log n\sqrt{\mu^*}}{n\beta_n\theta_{\text{max}}^2}; \nonumber  \\
    \varepsilon_2 :=& \left(\frac{K\log^{0.5} n+K^{1.5}\sqrt{\mu^*}}{n^{0.5}\theta_{\text{max}}}+\frac{K\sqrt{\mu^*}\log n}{n\theta_{\text{max}}^2}+\frac{K^3}{n\theta_{\text{max}}^3}\right)\varepsilon_1 + \frac{K^3\sqrt{\mu^*}+K^{2.5}\log^{0.5}n}{n\theta_{\text{max}}^2} \nonumber\\
    &+\frac{K^{2.5}\log^{1.5} n+K^{3}\sqrt{\mu^*}}{n^{1.5}\theta_{\text{max}}^3} +\frac{K^{0.5}\log^2n\sqrt{\mu^*}}{n^2\theta_{\text{max}}^4}+\varepsilon_0,\label{eps1eps2}
\end{align}
where $\varepsilon_0$ is defined by \eqref{eq:define_eps_zero}. \par
Despite the complicated expressions, these two quantities are easily interpretable. $\varepsilon_1$ controls the estimation error $\|\bR^\top\wh\br_i-\br_i^*\|_2$ according to the Theorem \ref{estimationerr} below, while $\varepsilon_2$ controls the expansion error $\left\|\boldsymbol{\Psi}_{\oU}\right\|_{2,\infty}$ according to Theorem \ref{mainthmrexpansion}. If we assume $K, \mu^*, \kappa^*, \beta_n, \theta_{\max} \asymp 1$, for all $i\in [n]$, then one have
\begin{align*}
    \varepsilon_1 \asymp \frac{\log n}{n^{0.5}}\quad \text{and}\quad\varepsilon_2\asymp\frac{\log^2 n}{n}.
\end{align*}
That is, the expansion error decays faster than the estimation error by $\sqrt{n}$ up to logarithmic factors. This validates our theoretical results. As we have mentioned before, the following theorem shows that $\varepsilon_1$ controls the estimation error.

\begin{thm}\label{estimationerr}
    Assume the conditions in Theorem \ref{mainthmrexpansion} hold. Assume 
\begin{align}
    n\gtrsim \max\Bigg\{&\frac{K^2}{\beta_n^2\theta_{\text{max}}^6}, \frac{K^4}{\beta_n^2\theta_{\text{max}}^4\log n}, \frac{\kappa^{*2}K^2\mu^*}{\beta_n^2\theta_{\text{max}}^2\log n}, K\mu^*, \frac{\kappa^*K^{2.5}}{\beta_n^2\theta_{\text{max}}^3\log^{0.5} n}, \frac{\kappa^*K^2}{\beta_n^2\theta_{\text{max}}^2}, \nonumber\\
    &\frac{K^{1.5}\log^{0.5}n}{\beta_n\theta_{\text{max}}}, \frac{K^4}{\theta_{\text{max}}^2},\frac{K^3}{\theta_{\text{max}}^4\mu^*}\Bigg\}\label{estimationerrthmeq1}
\end{align}

and
\begin{align}
    n^{1.5}\gtrsim \max\Bigg\{\frac{\kappa^*K^4}{\beta_n^3\theta_{\text{max}}^4\log^{0.5}n}, \frac{\kappa^*K^3}{\beta_n^3\theta_{\text{max}}^4}, \frac{K^{2.5}\log^{0.5}n}{\beta_n^2\theta_{\text{max}}^3}, \frac{\kappa^{*2}K^3\sqrt{\mu^*}}{\beta_n^3\theta_{\text{max}}^3 \log^{0.5}n}, \frac{\kappa^*K^{2.5}\sqrt{\mu^*}}{\beta_n^2\theta_{\text{max}}^2}\Bigg\}.\label{estimationerrthmeq2}
\end{align}

    Then with probability at least $1-O(n^{-10})$ we have 
    \begin{align}
        \max_{1\leq i\leq n}\left\|\bR^\top\wh\br_i-\br_i^*\right\|_2\lesssim \varepsilon_1.\label{estimationerrbound2}
    \end{align}
\end{thm}
\begin{proof}
    See Section \ref{estimationerrproof}.
\end{proof}

The above result obtains finite sample error bound for the $\hat{r}_i$s defined by \eqref{eq:define_hat_r}. By Theorem \ref{mainthmrexpansion}, we know that
\begin{align}
    \max_{1\leq i\leq n}\left\|\bR^\top\wh\br_i-\br_i^*-\Delta\br_i\right\|_2&\lesssim \max_{1\leq i\leq n}|\gamma_i|_2\cdot   \max_{1\leq i\leq n}\left\|\bR^\top\wh\br_i-\br_i^*\right\|_2 + \max_{1\leq i\leq n}\left\|[\boldsymbol{\Psi}_{\br}]_{i,\cdot}\right\|_2 \nonumber \\
    &\lesssim \max_{1\leq i\leq n}|\gamma_i|_2\cdot   \varepsilon_1 + \max_{1\leq i\leq n}\left\|[\boldsymbol{\Psi}_{\br}]_{i,\cdot}\right\|_2\lesssim \varepsilon_2 \label{eq:approxerror}
\end{align}
with probability at least $1-O(n^{-10})$, where we define
\begin{align}\label{eq:define_delta_r_i}
    \Delta \br_i := (\bw_i-\left[\bN\bW\bu_1^*\right]_i\br_i^*/\lambda_1^*)/(\bu_1)_i, \quad i\in [n].
\end{align}

\begin{rem}
   We remark here that \eqref{estimationerrthmeq1} and \eqref{estimationerrthmeq2} are two sufficient conditions to ensure \eqref{estimationerrbound2} holds, but not necessary. In fact, these two conditions ensure the upper bound of the first order term $\left\|\Delta \br_i\right\|_2$ dominates the upper bound of the expansion error, which is given by Theorem \ref{mainthmrexpansion}, in order to simplify the upper bound of the estimation error $\max_{1\leq i\leq n}\|\bR^\top\wh\br_i-\br_i^*\|_2$. In other words, the results in the rest of the paper hold without \eqref{estimationerrthmeq1} and \eqref{estimationerrthmeq2}, but a more complicated expression of $\varepsilon_1$ is required if we don't have \eqref{estimationerrthmeq1} and \eqref{estimationerrthmeq2}.
\end{rem}

\section{Vertex Hunting}\label{sec:vertex_hunting}
In this section, we describe how to estimate the $K$ underlying vertices of the simplex based on the dataset. To this end, define disjoint subsets $\VV_1, \VV_2, \dots, \VV_K \subset [n]$ such that $\VV_k$ is the collection of all the pure nodes of the $k$-th community, and let $\bb_k^*$ be the vertex of the corresponding community, i.e.,
\begin{align}\label{eq:define_b_star}
    \VV_k = \{i\in [n]:\bpi_{i}(k) = 1\} \quad \text{and}\quad \bb_k^* = \left(\sum_{i\in \VV_k}\br_i^*\right) / |\VV_k| \quad \forall k\in [K].
\end{align}
The following quantity 
\begin{align}\label{eq:min_dist_from_vert}
    \Delta_{\br} = \min_{k\in [K]}\min_{i\in [n]\backslash \VV_k}\left\|\br_i^*-\bb_k^*\right\|_2
\end{align}
measures the gap between any vertex and the other points. We will use the following successive projection algorithm given by Algorithm \ref{alg1} for the vertex hunting step.
\begin{algorithm}[ht]
\caption{Successive projection}
	\begin{algorithmic}[1]		
        \STATE \textbf{Input} $\wh\br_1,\wh\br_2,\dots,\wh\br_n$, radius $\phi$
		\STATE \textbf{Initialize} $\bZ_i = (1, \wh\br_i^\top)^\top$, for $i\in [n]$
        \FOR{$k = 1,2,\dots, K$} 
        \STATE Let $i_k = \argmax_{1\leq i\leq n}\left\|\bZ_i\right\|_2$ and $\wh\bb_k' = \wh \br_{i_k}$  
        \STATE Update $\bZ_i \leftarrow \bZ_i - \wh \br_{i_k}\wh \br_{i_k}^\top\bZ_i /\left\|\wh \br_{i_k}\right\|_2^2$, for $i\in [n]$
        \ENDFOR
        \STATE Let $$\wh\VV_k = \left\{i\in [n]:\left\|\wh\br_i - \wh\bb_k'\right\|_2\leq \phi\right\} \text{ and } \wh\bb_k = \frac{1}{|\wh\VV_k|}\sum_{i\in \wh\VV_k}\wh\br_i$$
        for $k\in [K]$
        \RETURN $\wh\VV_1, \wh\VV_2, \dots, \wh\VV_K$ and $\wh\bb_1,\wh\bb_2,\dots, \wh\bb_K$
	\end{algorithmic}\label{alg1}
\end{algorithm}
Successive projection algorithm (SPA) is a forward variable selection method introduced by~\cite{araujo2001successive}. Our version of SPA borrows from\cite{gillis2013fast}, who derived its  theoretical guarantee. One might consider alternate vertex hunting algorithms similar in spirit to ~\cite{jin2017estimating}. We leave this for future research directions.

Our goal for the remainder of the section is to understand how SPA is effective in selecting the underlying vertices. If $\Delta_r$, defined by \eqref{eq:min_dist_from_vert}, is not too small, we expect the vertex hunting algorithm to retrieve all the vertices of the simplex, namely, selection consistency. This is the main result of this section. The proof follows by combining Theorem \ref{estimationerr} with Theorem 3 of \cite{gillis2013fast}.
\begin{thm}\label{SPmainthm}
Assume the conditions in Theorem \ref{estimationerr} hold and the estimation error bound on the right hand side of \eqref{estimationerrbound2} is at most of a constant order. If we further have 
\begin{align*}
    \Delta_{\br}> 2\phi> C_{\text{SP}}\cdot \varepsilon_1 
\end{align*}
for some constant $C_{\text{SP}}>0$, then with probability at least $1-O(n^{-10})$, there exists a permutation $\rho$ of $[K]$, such that $\wh\VV_{\rho(k)} = \VV_{k}$ for all $k\in[K]$.
\end{thm}
\begin{proof}
    See Section \ref{SPmainthmproof}.
\end{proof}

 Theorem \ref{SPmainthm} along with Theorem \ref{mainthmrexpansion} yields the following result regarding $\wh\bb_k$.
\begin{cor}\label{vertexexpansion}
Assume the conditions in Theorem \ref{SPmainthm} hold. Then with probability at least $1-O(n^{-10})$, there exists a permutation $\rho$ of $[K]$, such that 
\begin{align}\label{eq:expand_hat_b}
    \bR^\top\wh\bb_{\rho(k)}-\bb_k^* = \frac{1}{\left|\VV_k\right|}\sum_{i\in \VV_k}\Delta\br_i + [\boldsymbol{\Psi}_{\bb}]_{k,\cdot}^\top,\quad \forall k\in [K],
\end{align}
where $\Delta\br_i$ is defined by \eqref{eq:define_delta_r_i}. Furthermore,
$\left\|\boldsymbol{\Psi}_{\bb}\right\|_{2,\infty}\lesssim \varepsilon_2$.
\end{cor}
\begin{proof}
    See Section \ref{vertexexpansionproof}.
\end{proof}

The leading term of RHS of \eqref{eq:expand_hat_b} will be denoted by
\begin{align}\label{eq:define_delta_b_k}
    \Delta\bb_k := \left(\sum_{i\in \VV_k}\Delta\br_i\right)/|\VV_k |, \quad k\in [K].
\end{align} 

\section{Membership Reconstruction}\label{sec:membership_reconstruction}
In this section we characterize the behavior of $\wh\bpi_i$. To this end, we will require the expansion of the following three terms
\begin{align*}
    \wh\lambda_1, \quad \wh\bb_k^\top\textbf{diag}(\wh\lambda_2,\wh\lambda_2,\dots,\wh\lambda_K)\wh\bb_k,k\in [K] \quad \text{and} \quad \wh\ba_i, i\in [n].
\end{align*}
The expansion of $\wh\lambda_1$ can be directly given by Theorem \ref{firstsubspaceexpansion}, and we defer the precise statement to Section \ref{sec:proof_outline}. We turn to derive the expansion of $\wh\bb_k^\top\textbf{diag}(\wh\lambda_2,\wh\lambda_2,\dots,\wh\lambda_K)\wh\bb_k = \wh\bb_k^\top\oLambda\wh\bb_k$. We will use the following notations: given any vector $\bv\in \mathbb{R}^K$ and permutation $\rho(\cdot)$ of $[K]$, set
\begin{align*}
    \rho(\bv) = \left[v_{\rho(1)}, v_{\rho(2)},\dots, v_{\rho(K)}\right]^\top.
\end{align*}
And, given a matrix $\boldsymbol{A}=[\ba_1,\ba_2,\dots,\ba_K]$ with $K$ columns, define
\begin{align*}
    \rho(\boldsymbol{A}) = \left[\ba_{\rho(1)}, \ba_{\rho(2)},\dots, \ba_{\rho(K)}\right].
\end{align*}
Now, we are in a position to state our result regarding $\wh\bb_k^\top\oLambda\wh\bb_k$.

\begin{lem}\label{membershipreconstructionlem2}
    Assume the conditions in Theorem \ref{SPmainthm} hold and $\varepsilon_2\lesssim\varepsilon_1\lesssim \sqrt{K-1}$, where $\varepsilon_1$ and $\varepsilon_2$ are defined via \eqref{eps1eps2}. Then with probability at least $1-O(n^{-10})$, for all $k\in [K]$, we have the following expansion:
    \begin{align*}
        \wh\bb_{\rho(k)}^\top\oLambda\wh\bb_{\rho(k)} - \bb_k^{*\top}\oLambda^*\bb_k^* =2\bb_k^{*\top}\oLambda^*\Delta\bb_k + \left[\boldsymbol{\psi}\right]_k, 
    \end{align*}
    where $\Delta\bb_k$ is defined by \eqref{eq:define_delta_b_k}, $b^\star_k$ is defined by \eqref{eq:define_b_star}, $\oLambda^\star := \textbf{diag}(\lambda^\star_2, \ldots, \lambda^\star_K)$ and the error $\boldsymbol{\psi}$ satisfies
    \begin{align*}
        \left\|\boldsymbol{\psi}\right\|_\infty\lesssim \frac{\kappa^*K^2}{\beta_n}+K^{1.5} \theta_{\text{max}}\log^{0.5}n + \left(K^{0.5} \varepsilon_2+\varepsilon_1^2\right)\maxsigma^*.
    \end{align*}
    For the estimation error, if 
        $\frac{\kappa^*K^{1.5}}{\beta_n}+K \theta_{\text{max}}\log^{0.5}n\lesssim \varepsilon_1\maxsigma^*$, we have 
    \begin{align*} \left|\wh\bb_{\rho(k)}^\top\oLambda\wh\bb_{\rho(k)} - \bb_k^{*\top}\oLambda^*\bb_k^*\right|\lesssim K^{0.5}\varepsilon_1\maxsigma^*.
    \end{align*}
\end{lem}
\begin{proof}
    See Section \ref{membershipreconstructionlem2proof}.
\end{proof}

The following Lemma gives an expansion of $\wh\ba_i,i\in [n]$.

\begin{lem}\label{membershipreconstructionlem3}
    Assume the conditions in Theorem \ref{SPmainthm} hold and $\varepsilon_2\lesssim \varepsilon_1\leq C$ for some constant $C>0$. Let $\rho(\cdot)$ be the permutation in Theorem \ref{SPmainthm}. Then with probability at least $1-O(n^{-10})$, for all $i\in [n]$, $\wh\ba_i$ can be expanded as 
    \begin{align*}
        \rho(\wh\ba_i)-\ba^*_i = (\bB^*)^{-1}\begin{bmatrix}
    \Delta\br_i\\
    0 
    \end{bmatrix}-(\bB^*)^{-1}\Delta\bB\ba_i^* + \left[\boldsymbol{\Psi}_{\ba}\right]_{i, \cdot}^\top,
    \end{align*}
    where 
    \begin{align}\label{eq:define_a_star}
    \boldsymbol{B}^* = \begin{bmatrix}
\bb_1^* & \bb_2^* & \dots & \bb_K^*\\
1 & 1 & \dots & 1
\end{bmatrix},
\quad \ba_i^* = (\bB^*)^{-1}\begin{bmatrix}
\br_i^* \\
1
\end{bmatrix}, 
\quad \Delta\bB = \begin{bmatrix}
\Delta\bb_1 & \Delta\bb_2 & \dots & \Delta\bb_K\\
0 & 0 & \dots & 0
\end{bmatrix},
\end{align}
and 
\begin{align*}
    \|\left[\boldsymbol{\Psi}_{\ba}\right]_{i, \cdot}\|_2 \lesssim \varepsilon_2+\varepsilon_1\left\|\rho(\wh\ba_i)-\ba^*_i\right\|_2.
\end{align*}
Furthermore, the estimation error can be controlled as 
\begin{align*}
    \left\|\rho(\wh\ba_i)-\ba^*_i\right\|_2\lesssim \varepsilon_1.
\end{align*}
\end{lem}
\begin{proof}
    See Section \ref{membershipreconstructionlem3proof}.
\end{proof}
For simplicity, we define, for $i\in [n]$
\begin{align}
    \Delta \ba_i = (\bB^*)^{-1}\begin{bmatrix}
    \Delta\br_i\\
    0 
    \end{bmatrix}-(\bB^*)^{-1}\Delta\bB\ba_i^*\label{eq:define_delta_a_i},
\end{align}
where $B^\star$ and $a^\star$ are defined by \eqref{eq:define_a_star}. As a counterpart of $\wh\bc$ in the membership construction step \eqref{eq:define_hat_c}, we define $c_k^* = [\lambda_1^*+\bb_k^{*\top}\oLambda^*\bb_k^*]^{-1/2}$ for $k\in [K]$, where $\bb_k^{*\top}$ is given by \eqref{eq:define_b_star}. The expansion of $\wh\bc$ can be seen from Corollary \ref{membershipreconstructionlem1} and Lemma \ref{membershipreconstructionlem2}. Combine them with Lemma \ref{membershipreconstructionlem3}, we obtain the following expansion of $\wh\bpi_i$ that is linear in $\bW$.

\begin{thm}\label{Piexpansion}
Assume the conditions in Theorem \ref{SPmainthm} hold and $\varepsilon_2\lesssim \varepsilon_1\leq C$ for some constant $C>0$. Let $\rho(\cdot)$ be the permutation in Theorem \ref{SPmainthm}. Then with probability at least $1-O(n^{-10})$, for all $i\in [n]$ and $k\in [K]$, we have 
\begin{align*}
    \wh\bpi_i(\rho(k))-\bpi_i(k) = (1+\eta_i)\Delta \bpi_i(k) +\left[\boldsymbol{\Psi}_{\bPi}\right]_{i, k},
\end{align*}
where 
\begin{align}
    \Delta \bpi_i(k) =& \frac{1}{\left(\sum_{l=1}^K (\ba^*_i)_{l} /c^*_{l}\right)^2}\Bigg\{\sum_{l\neq k, l\in [K]} \textbf{Tr}\left[\bW\bu_1^*\bu_1^{*\top}+ 2\bN\bW\bu_1^*\bu_1^{*\top}\right]\left(\frac{c_k^*}{2 c_l^*} - \frac{c_l^*}{2 c_k^*}\right)(\ba_i^*)_k(\ba_i^*)_l  \nonumber\\
    &\quad \quad + \frac{(\Delta\ba_i)_k(\ba_i^*)_l - (\Delta\ba_i)_l(\ba_i^*)_k}{c_k^*c_l^*}+\left(\frac{\bb_k^{*\top}\oLambda^*\Delta\bb_k c_k^*}{c_l^*} - \frac{\bb_l^{*\top}\oLambda^*\Delta\bb_l c_l^*}{c_k^*}\right)(\ba_i^*)_k(\ba_i^*)_l \Bigg\}, \label{eq:define_delta_pi_ik}
\end{align}
$|\eta_i|  \lesssim K\varepsilon_1$ and
\begin{align*}  \left|\left[\boldsymbol{\Psi}_{\bPi}\right]_{i, k}\right|  \lesssim \frac{K}{n\theta_{\text{max}}^2}\left(\frac{\kappa^*K^2}{\beta_n}+K^{1.5} \theta_{\text{max}}\log^{0.5}n + \left(K^{0.5} \varepsilon_2+\varepsilon_1^2\right)\maxsigma^*\right) + K(\varepsilon_2+\varepsilon_1^2).
\end{align*}
\end{thm}
\begin{proof}
See Section \ref{Piexpansionproof}.
\end{proof}

Given $\varepsilon_1$ and $\varepsilon_2$ in \eqref{eps1eps2}, we define 
\begin{align*}
    \varepsilon_3 := \frac{K}{n\theta_{\text{max}}^2}\left(\frac{\kappa^*K^2}{\beta_n}+K^{1.5} \theta_{\text{max}}\log^{0.5}n + \left(K^{0.5} \varepsilon_2+\varepsilon_1^2\right)\maxsigma^*\right) + K(\varepsilon_2+\varepsilon_1^2),
\end{align*}
where $\kappa^\star$ is defined by \eqref{eq:define_kappa}. Since $\maxsigma^*\asymp\beta_nK^{-1}\|\btheta\|_2^2\lesssim K^{-1}\|\btheta\|_2^2$, $\max_{k\in [K]}\|\Delta \bb_k\|_2\lesssim \varepsilon_1$ and $\max_{i\in [n]}\|\Delta\ba_i\|_2\lesssim \varepsilon_1$, 
from Theorem \ref{Piexpansion} one can show that 
\begin{align*}
    |\wh\bpi_i(\rho(k))-\bpi_i(k)  -\Delta \bpi_i(k)|\lesssim \varepsilon_3
\end{align*}
with probability at least $1-O(n^{-10})$. Specifically, if $K, \mu^*, \kappa^*, \beta_n, \theta_{\max} \asymp 1$, we have
\begin{align*}
    \varepsilon_3\asymp \frac{\log^2 n}{n}.
\end{align*}

\section{Distributional Theory and Rank Inference}\label{distributionaltheory}
In this section, we tackle specific inference problems based on the uncertainty quantification results we stated before. First, we establish distributional guarantee using the first order expansion we derived in Theorem \ref{Piexpansion}. Second, we apply the distributional results to related inference problem, especially rank inference application. To state the distributional results of $\wh\bpi_i$, we need the following notations.
They are non-random matrices of dimension $n\times n$
\begin{align}
    &\boldsymbol{C}^{\br}_{i,k} :=\frac{\oLambda_{kk}^*}{(\bu_1)_i}\left[\boldsymbol{e}_i\left(\oU_{\cdot,k}^*\right)^\top\right] - \oLambda_{kk}^*\left[\bu_1^*\left(\oU_{\cdot,k}^*\right)^\top\bN\right] - \frac{(\br_i^*)_k}{(\bu_1^*)_i\lambda_1^*}\bu_1^*\bN_{i, \cdot} ,  \nonumber \\
    &\boldsymbol{C}^{\ba}_{i,k} := \sum_{l\in [K-1]}\left(\bB^*\right)_{k, l}^{-1}\boldsymbol{C}^{\br}_{i, l}-\sum_{l\in [K-1]}\sum_{t\in [K]}\left(\bB^*\right)_{k, l}^{-1}(\ba_i^*)_t\boldsymbol{C}^{\bb}_{t, l},  \quad \boldsymbol{C}^{\bb}_{i,k} := \frac{1}{\left|\VV_i\right|}\sum_{j\in \VV_i}\boldsymbol{C}^{\br}_{j,k}, \nonumber \\
    &\boldsymbol{C}^{\bpi}_{i,k} := \frac{1}{\left(\sum_{l=1}^K (\ba^*_i)_{l} /c^*_{l}\right)^2}\sum_{l\neq k, l\in [K]} \Bigg\{\left(\bu_1^*\bu_1^{*\top}+ 2\bu_1^*\bu_1^{*\top}\bN\right)\left(\frac{c_k^*}{2 c_l^*} - \frac{c_l^*}{2 c_k^*}\right)(\ba_i^*)_k(\ba_i^*)_l  \nonumber \\
    &\quad \quad + \frac{(\ba_i^*)_l \boldsymbol{C}^{\ba}_{i,k}- (\ba_i^*)_k \boldsymbol{C}^{\ba}_{i,l}}{c_k^*c_l^*}+\sum_{t\in [K-1]}\left(\frac{\bb_{kt}^{*\top}\oLambda^*_{tt} c_k^*\boldsymbol{C}^{\bb}_{k,t}}{c_l^*} - \frac{\bb_{lt}^{*\top}\oLambda^*_{tt} c_l^*\boldsymbol{C}^{\bb}_{l,t}}{c_k^*}\right)(\ba_i^*)_k(\ba_i^*)_l \Bigg\}. \label{Cpidefinition}
\end{align}
One can see that
\begin{align*}
    &(\Delta\br_i)_k = \textbf{Tr}\left[\boldsymbol{C}_{i, k}^{\br}\bW\right],\quad  (\Delta\bb_i)_k = \textbf{Tr}\left[\boldsymbol{C}_{i, k}^{\bb}\bW\right],\\ &(\Delta\ba_i)_k = \textbf{Tr}\left[\boldsymbol{C}_{i, k}^{\ba}\bW\right],\quad \Delta\bpi_i(k) = \textbf{Tr}\left[\boldsymbol{C}_{i, k}^{\bpi}\bW\right],
\end{align*}
where $\Delta\br_i$, $\Delta\bb_i$, $\Delta\ba_i$, $\Delta\bpi_i(k)$ are defined by  \eqref{eq:define_delta_r_i}, \eqref{eq:define_delta_b_k}, \eqref{eq:define_delta_a_i} and \eqref{eq:define_delta_pi_ik} respectively. For any matrices $\bM,\bM_1,\bM_2\in \mathbb{R}^{n\times n}$, we define the variance of $\textbf{Tr}[\bM\bW]$ as
\begin{align*}
    &V_{\bM} := \sum_{i\in [n]} M_{ii}^2H_{ii}(1-H_{ii})+\sum_{1\leq i < j \leq n}(M_{ij}+M_{ji})^2H_{ij}(1-H_{ij}),
\end{align*}
and the covariance of $\textbf{Tr}[\bM_1\bW]$ and $\textbf{Tr}[\bM_2\bW]$ as
\begin{align*}
    & V_{\bM_1,\bM_2} := \sum_{i\in [n]} M_{1, ii}M_{2, ii}H_{ii}(1-H_{ii})+\sum_{1\leq i < j \leq n}(M_{1,ij}+M_{1,ji})(M_{2,ij}+M_{2,ji})H_{ij}(1-H_{ij}).
\end{align*}
The asymptotic distribution of $\wh\bpi_i$ is given by the following Theorem: 
\begin{thm}\label{distributionthm}
Let $\rho(\cdot)$ be the permutation in Theorem \ref{SPmainthm}. Let $r$ be a fixed integer and
\begin{align*}
\mathcal{I} = \left\{(i_1,k_1), (i_2,k_2),\dots,(i_r,k_r)\right\}    
\end{align*}
be $r$ distinct fixed pairs from $[n]\times [K]$. We consider vectors
\begin{align*}
    \wh\bpi_\mathcal{I}:=\left(\wh\bpi_{i_1}(\rho(k_1)), \wh\bpi_{i_2}(\rho(k_2)),\dots, \wh\bpi_{i_r}(\rho(k_r))\right)^\top \text{ and  }
    \bpi_\mathcal{I}:=\left(\bpi_{i_1}(k_1), \bpi_{i_2}(k_2),\dots, \bpi_{i_r}(k_r)\right)^\top, 
\end{align*}
Denote by $\bSigma$ the covariance matrix whose $j^{th}$ diagonal entry is $V_{\bC^{\bpi}_{i_j,k_j}}, j\in [r]$ and $(j, k)$ off-diagonal entry is $V_{\bC^{\bpi}_{i_j,k_j},\bC^{\bpi}_{i_k,k_k}}, j\neq k\in [r]$. 
Then for any convex set $\mathcal{D}\subset \mathbb{R}^r$, we have
\begin{align*}
    \left|\mathbb{P}(\wh\bpi_\mathcal{I} - \bpi_\mathcal{I} \in \mathcal{D}) - \mathbb{P}(\mathcal{N}(\boldsymbol{0}_r,\bSigma) \in \mathcal{D})\right| = o(1),
\end{align*}
as long as
\begin{align}
    r^{5/4}\max_{1\leq i\leq j\leq n}\left\|\bSigma^{-1/2}\bomega_{ij}\right\|_2 = o(1) \text{ and } \lambda_{r}^{-1/2}(\bSigma)r^{3/4}\varepsilon_3 = o(1), \label{distributionthmcondition}
\end{align}
where for each pair $(i,j)\in [n]\times [n]$ such that $i\leq j$, 
\begin{align*}
    \bomega_{ij} = \begin{cases}
      \left((\bC^{\bpi}_{i_1,k_1})_{ij}+(\bC^{\bpi}_{i_1,k_1})_{ji},(\bC^{\bpi}_{i_2,k_2})_{ij}+(\bC^{\bpi}_{i_2,k_2})_{ji},\dots,(\bC^{\bpi}_{i_r,k_r})_{ij}+(\bC^{\bpi}_{i_r,k_r})_{ji}\right)^\top & i\neq j; \\
      \left((\bC^{\bpi}_{i_1,k_1})_{ii},(\bC^{\bpi}_{i_2,k_2})_{ii},\dots,(\bC^{\bpi}_{i_r,k_r})_{ii}\right)^\top & i=j.
    \end{cases}    
\end{align*} 
\end{thm}
\begin{proof}
See Section \ref{distributionthmproof}.
\end{proof}

Next, we will apply Theorem \ref{distributionthm} to answer some inference questions. 
In many practical applications, one is concerned with the characteristics of each community rather than the index (e.g., $1,2,\dots,K$) of the community. 
For this reason, we will assume that the permutation $\rho(\cdot)$ is the identical map, that is, $\rho(i) = i$ for all index $i$.
\begin{exa}\label{exa1}
Given a node $i$ in the network, it is natural to ask which community it is closest to. This amounts to find the largest component of $\{\pi_i(k)\}_{k=1}^K$ and can be formulated as $K$ hypothesis testing problems. For $k\in [K]$, we consider the following testing problem:
\begin{align*}
    H_{k0}:\bpi_i(k) \leq
    \max_{l\in [K]\backslash \{k\}} \bpi_i(l) \quad \text{ v.s. }\quad H_{ka}: \bpi_i(k)>
    \max_{l\in [K]\backslash \{k\}} \bpi_i(l).
\end{align*}
According to Theorem \ref{distributionthm} with $r = 2$, we consider the following 
Bonferroni-adjusted critical region at a significance level of $\alpha$ 
\begin{align*}
    \left\{\wh\bpi_i(k)> \wh\bpi_i(l)+\Phi^{-1}\left(\frac{\alpha}{K-1}\right)
    \sqrt{\wh V_{\bM}}
    ,\;  \forall l\in [K]\backslash\{k\}\right\},
\end{align*}
where with $\bM = \widehat{\boldsymbol{C}}^{\bpi}_{i,k} - \widehat{\boldsymbol{C}}^{\bpi}_{i,l}$ and $\wh\bH := \sum_{i=1}^K\wh\lambda_i\wh\bu_i\wh\bu_i^\top$, 
\begin{align}
   \wh V_{\bM} = \sum_{i\in [n]} M_{ii}^2\wh H_{ii}(1-\wh H_{ii})+\sum_{1\leq i < j \leq n}(M_{ij}+M_{ji})^2\wh H_{ij}(1-\wh H_{ij}).\label{hatvardefinition}
\end{align}
It is easy to see from the above critical region that at most one of $H_{10},H_{20},\dots, H_{K0}$ can be rejected. Once $H_{k0}$ is rejected, we can conclude that node $i$ is closest to community $k$ at a significance level of $\alpha$.
\end{exa}

\begin{exa}\label{exa2}
Moving beyond the question of closest community detection, it is often of interest to understand the rank of node $i$ with respect to community $k$ such as a ranking of conservativeness of a journal or a book.   Let $R_{i, k}$ be the rank  of $\bpi_i(k)$ among $\bpi_1(k), \bpi_2(k),\dots, \bpi_n(k)$. We adopt the method proposed by \cite{fan2022ranking} to construct rank confidence interval for $R_{i, k}$. Consider the following random variable $\mathcal{T}$ and its bootstrap counterpart $\mathcal{G}$:
\begin{align*}
    \mathcal{T} &= \max_{j:j\neq i}\left|\frac{(\wh\bpi_j(k) - \wh\bpi_i(k))-(\bpi_j(k) - \bpi_i(k))}{\sqrt{V_{\boldsymbol{C}^{\bpi}_{j,k} - \boldsymbol{C}^{\bpi}_{i,k}}}}\right|,\\
    \mathcal{G} &= \max_{j:j\neq i}\left|\frac{\textbf{Tr}\left[\left(\boldsymbol{C}^{\bpi}_{j,k} - \boldsymbol{C}^{\bpi}_{i,k}\right)\left(\bW\odot \boldsymbol{G} \right)\right]}{\sqrt{V_{\boldsymbol{C}^{\bpi}_{j,k} - \boldsymbol{C}^{\bpi}_{i,k}}}} - \frac{\sum_{1\leq a\leq b \leq n}G_{ab}}{(n^2+n)/2}\frac{\textbf{Tr}\left[\left(\boldsymbol{C}^{\bpi}_{j,k} - \boldsymbol{C}^{\bpi}_{i,k}\right)\bW\right]}{\sqrt{V_{\boldsymbol{C}^{\bpi}_{j,k} - \boldsymbol{C}^{\bpi}_{i,k}}}}\right|,
\end{align*}
where $\boldsymbol{G}\in \mathbb{R}^{n\times n}$ is a symmetric random matrix whose upper triangular (include diagonal) entries are i.i.d. standard Gaussian distribution and $\boldsymbol{A}\odot\boldsymbol{B}$ denotes the elementwise product of matrices $\boldsymbol{A}$ and $\boldsymbol{B}$. Given any $\alpha\in (0, 1)$, we define $c_{1-\alpha}$ as the $(1-\alpha)$th quantile of the conditional distribution of $\mathcal{G}$ given $\bW$. Then by \cite[Theorem 2.2]{chernozhuokov2022improved}, one can show that 
\begin{align}
    \left|\mathbb{P}(\mathcal{T}>c_{1-\alpha})-\alpha\right|\to 0 \label{exa2eq1}
\end{align}
under mild regularity condition (See Section \ref{exa2eq1proof} for more details). Using the plug-in estimators and estimated critical value $\wh c_{1-\alpha}$ from the bootstrap samples, we construct the following simultaneous confidence intervals for $\{\bpi_j(k) - \bpi_i(k)\}_{j\in [n]\backslash\{i\}}$ with a confidence level of $1-\alpha$ as
\begin{align*}
     \left[C_L(j), C_U(j)\right]:= \left[\wh\bpi_j(k)-\wh\bpi_i(k)\pm\wh c_{1-\alpha} \sqrt{\wh V_{\widehat{\boldsymbol{C}}^{\bpi}_{j,k} - \widehat{\boldsymbol{C}}^{\bpi}_{i,k}}}\right].
\end{align*}
In other words, for all $j\in [n]\backslash\{i\}$, $\bpi_j(k) - \bpi_i(k)\in [C_L(j), C_U(j)]$ with probability at least $1-\alpha$. Now $C_L(j)>0$ implies $\bpi_j(k)> \bpi_i(k)$ and counting the number of such $j's$ give the lower bound the rank of $\bpi_i(k)$. Similarly, 
$C_U(j)<0$ implies $\bpi_j (k)< \bpi_i (k)$ and this gives an upper bound on the rank of $\bpi_i(k)$.  As a result,  
\begin{align*}
    \left[1+\sum_{j\in [n]\backslash\{i\}}\mathbb{I}(C_L(j)>0), n-\sum_{j\in [n]\backslash\{i\}}\mathbb{I}(C_U(j)<0)\right]
\end{align*}
forms a $100(1-\alpha)\%$ confidence interval for $R_{i,k}$.
\end{exa}

\begin{exa}
\cite{fan2022simple} proposed the SIMPLE test to study the statistical inference on the membership profiles. Specifically, for each node pair $i\neq j\in [n]$, we are interested in the following testing problem:
\begin{align*}
    H_0: \bpi_i = \bpi_j \quad \text{ v.s. }\quad H_a:\bpi_i\neq \bpi_j.
\end{align*}
Theorem~\ref{distributionthm} allows us to recover their result.
To see this,  we take $r = 2(K-1)$, 
\begin{align*}
    \mathcal{I} = \{(i,1),(i,2),\dots,(i, K-1),(j, 1), (j, 2),\dots,(j, K-1)\}.
\end{align*}
We define matrix $\boldsymbol{T}\in \mathbb{R}^{(K-1)\times(2K-2)}$ by 
\begin{align*}
  T_{pq} =
    \begin{cases}
      1 & q = p \\
      -1 & q = p + K - 1 \\
      0 & \text{otherwise}
    \end{cases},\quad  \forall (p, q)\in [K-1]\times [2K-2].
\end{align*}
As a result, by Theorem \ref{distributionthm}, as long as condition \eqref{distributionthmcondition} holds, under null hypothesis $H_0: \bpi_i=\bpi_j$ we have 
\begin{align*}
\Bigl ((\wh\bpi_i)_{1:K-1} - (\wh\bpi_j)_{1:K-1} \Bigr )^T \left(\boldsymbol{T}\bSigma\boldsymbol{T}^\top\right)^{-1} \Bigl ( (\wh\bpi_i)_{1:K-1} - (\wh\bpi_j)_{1:K-1} \Bigr ) \to \chi^2_{K-1}.
\end{align*}
This Hotelling type of statistic can be used to test the null hypothesis for two individual nodes and recovers the result in \cite{fan2022simple}.

\end{exa}

\section{Numerical Studies}\label{sec:simulation}
In this section, we conduct numerical experiments on both synthetic data and real data to complement our theoretical results. We first validate our distributional results by simulations. Then, we apply our approach to stock dataset to study the simplex structure and do rank inference, as we have mentioned in previous examples.

\subsection{Synthetic Data Simulation}
Here, we conduct synthetic data experiments to verify our uncertainty quantification results in Theorems \ref{Piexpansion} and \ref{distributionthm}. Our simulation setup is as follows: set the number of nodes $n=2000$ and number of communities $K=2$. To generate the membership matrix $\bPi$, we first set the first two rows of the $2000\times 2$ matrix $\bPi$ as $[1, 0]$ and $[0, 1]$, as two pure nodes. The first entries of the remaining $1998$ rows are sampled independently from the uniform distribution over interval $[0.1,0.9]$, while the second entries are determined by the first entries since the row sum is $1$. Then we randomly shuffle the rows of $\bPi$. For the matrix $\bP\in \mathbb{R}^{2\times 2}$, we set its diagonals as $1$ and off-diagonals as $0.2$. In terms of the $\theta_i$, which partially represents the signal strength, we consider three settings: (i) $\theta_i = 0.6$ for all nodes. (ii) $\theta_i$'s are sampled independently from the uniform distribution over interval $[0.3, 0.9]$. (iii) $\theta_i=0.9$ for all nodes. In each setting, we generate the network and obtain the estimated mixed membership matrix estimator $\wh\bPi$ $500$ times. We then record the realizations of the following standardized random variable
\begin{align*}
    \frac{\wh\bpi_{1}(1) - \bpi_1(1)}{\sqrt{ \widehat{V}_{\widehat{\boldsymbol{C}}^{\bpi}_{1,1}}}},
\end{align*}
where $\widehat{V}$ {\color{black}is defined by \eqref{hatvardefinition}} and $\widehat{\boldsymbol{C}}^{\bpi}_{1,1}$ {\color{black} is the plug-in estimator of $\boldsymbol{C}^{\bpi}_{1,1}$ given by \eqref{Cpidefinition}}. Figure \ref{simulation} summarizes the results collected from the $500$ simulations. The three plots in the first row show the histograms of the results from each setting, and the orange curve is the density of standard normal distribution. The three Q-Q plots in the second row further examine the normality in the three settings. These results suggest that the random variable is nearly normally distributed and the estimated asymptotic variance is right.
They in turn support further our theoretical results Theorem \ref{Piexpansion} and Theorem \ref{distributionthm}.

\begin{figure}[ht]
\begin{center}
\includegraphics[width=\textwidth]{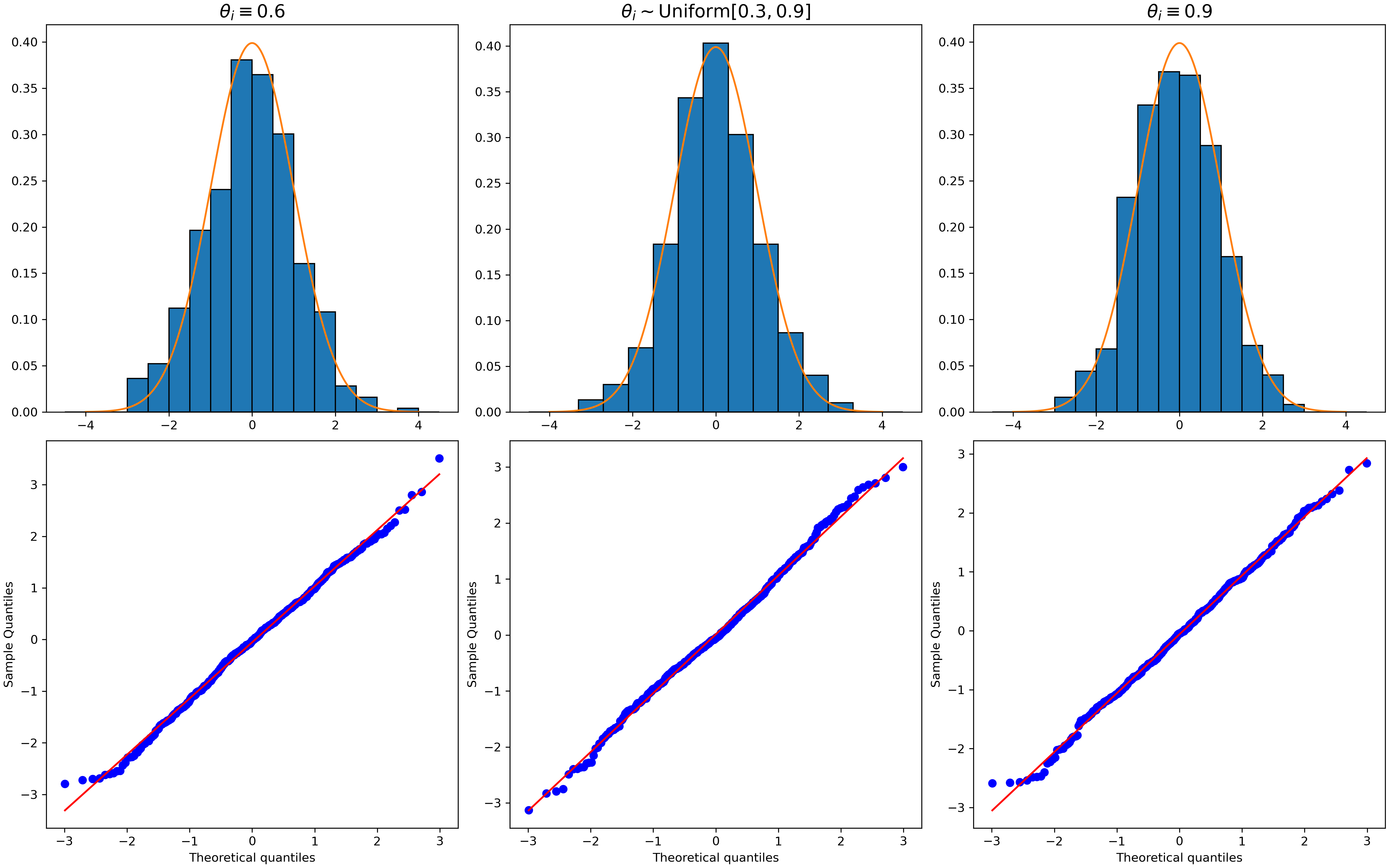} 
\end{center}
\caption{Histograms and Q-Q plots for validating the normality of $(\wh\bpi_{1}(1) - \bpi_1(1))/\sqrt{\widehat{\textbf{var}}(\widehat{\boldsymbol{C}}^{\bpi}_{1,1})}$. The orange curves in the first row of plots are the density of standard normal distribution. Three columns represent three choices of $\theta_i$, and in each setting the simulation is repeated for $500$ times.}
\label{simulation}
\end{figure}

\subsection{Real Data Experiments}
In this subsection, we apply our uncertainty quantification results 
to financial dataset. Our dataset consists of the daily close prices of the S\&P 500 stocks from January 1, 2010 to December 31, 2022 {\color{black}from Yahoo Finance}\footnote{https://pypi.org/project/yfinance/} and we calculated the log returns. We clean the data by removing the stocks with more than one missing values. For those with just one missing value, we let the corresponding log returns be zero. We would like to construct a network based on the correlation of the log returns. It is well known in finance that some common factors account for much of this correlation. Similar to \cite{fan2019simple}, we first fit a factor model with five factors to remove these common factors and then construct the network based on the covariance matrix of the idiosyncratic components. Letting $\boldsymbol{A}$ be the covariance matrix of the idiosyncratic components, we draw an edge between nodes $i$ and $j$ if and only if $A_{ij} > 0.1$. In this way, we obtain the adjacency matrix $\bX$. After the preprocessing steps, we have $n = 433$ stocks remained. 

We take $K=3$ and apply the SCORE normalization step to the leading eigenvectors of $\bX$. On the left side of Figure \ref{simplex} we display the scatter plot of $\wh\br_i$ for $i\in [n]$ and show the $2$-dimensional simplex structure. As we can see from the figure, it has a clear triangular structure. Looking closely at the nodes, we can find some characteristics of the three corners of this triangle. Many financial companies (PNC, NDAQ, TROW, RE, PSA) are very close to the vertex of the right corner, which can be viewed as the pure node of this community. And, many other financial companies (AJG, WRB, AXP, BRO, C) are also closer to this corner compared to the other two corners. In the top corner, we can find companies (REGN, EW, HOLX, ILMN) belonging to the healthcare industries. Moreover, some other healthcare companies PFE, HUM, LH, TMO, ABT can also be viewed as mixed members which are closer to this community. Similarly, companies related to the energy industry such as MPC, BA, TDY, EMR, AEP, DOV, XEL, FE make up a large part of the bottom corner, while other similar companies ED, LNT seem to be mixed members close to this corner. To validate these observations, we apply Theorem \ref{distributionthm} to conduct the following tests for each $i\in [n]$ as we have stated in Example \ref{exa1}.
\begin{align}
    &H_{01}: \bpi_i(1)\leq \max\{\bpi_i(2),  \bpi_i(3)\} \quad \text{ v.s. }\quad H_{a1}: \bpi_i(1)> \max\{\bpi_i(2),  \bpi_i(3)\};  \nonumber\\
    &H_{02}: \bpi_i(2)\leq \max\{\bpi_i(1),  \bpi_i(3)\} \quad \text{ v.s. }\quad H_{a2}: \bpi_i(2)> \max\{\bpi_i(1),  \bpi_i(3)\};  \label{eqtest}\\
    &H_{03}: \bpi_i(3)\leq \max\{\bpi_i(1),  \bpi_i(2)\} \quad \text{ v.s. }\quad H_{a3}: \bpi_i(3)> \max\{\bpi_i(1),  \bpi_i(2)\};  \nonumber
\end{align}
 At most one of $H_{01}, H_{02}$ and $H_{03}$ can be rejected. On the right side of Figure \ref{simplex}, we show the results of the aforementioned tests. We use red/blue/green to represent that $H_{01}$/$H_{02}$/$H_{03}$ is  rejected respectively. If none of these three null hypothesises is rejected, we let be point be grey. As we can see from Figure \ref{simplex}, for most ($293/433 = 67.7\%$) of the nodes, one of $H_{01}, H_{02}$ and $H_{03}$ is rejected. In other words, although many nodes have mixed membership, we can identify them to be closer to one community. Furthermore, the test results confirm our observation about the three corners we have mentioned before. It is worth mentioning that the information technology companies, which make up a large proportion of the S\&P 500 list, can be found in abundance in any part of the simplex. This also indicates that the prosperous development of information technology industry is a result of the nurturing of traditional industries.
\begin{figure}[ht]
\begin{center}
\includegraphics[width=\textwidth]{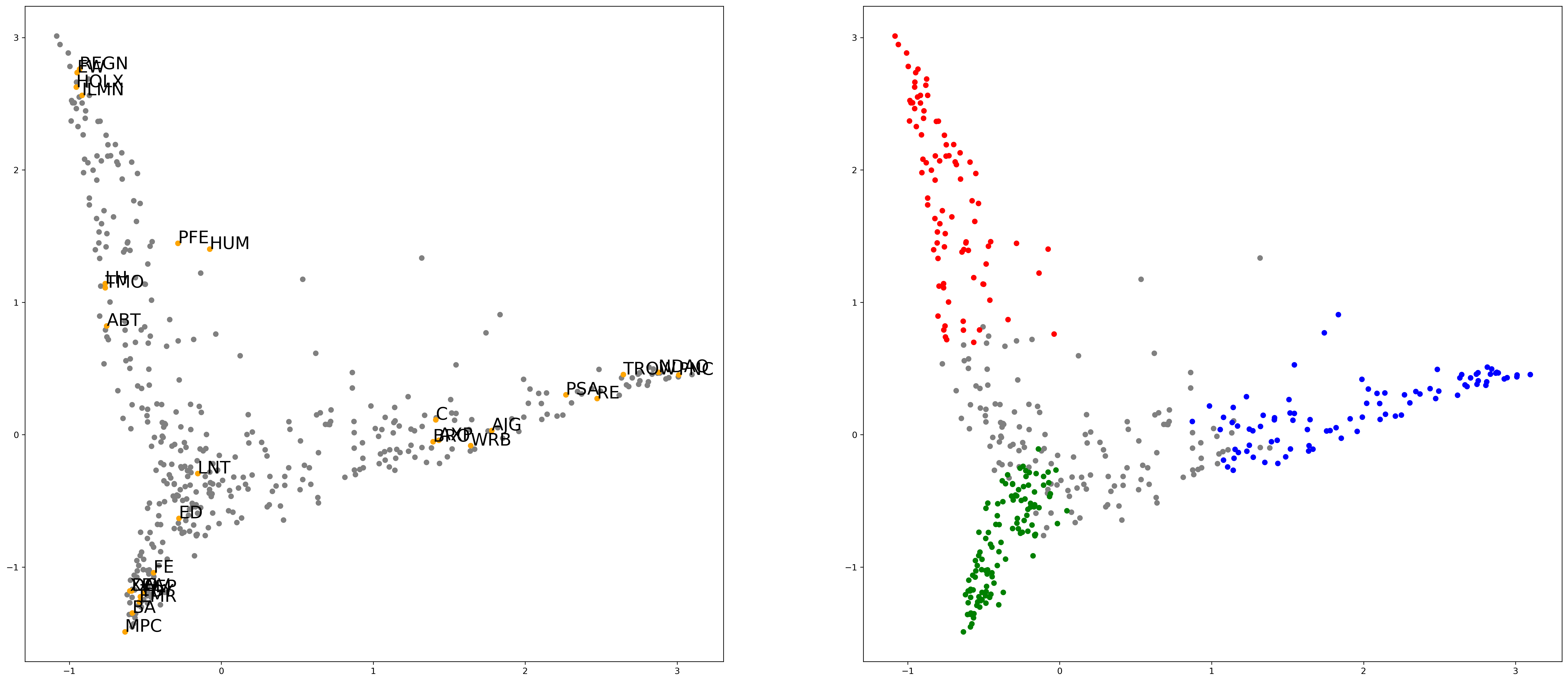} 
\end{center}
\caption{Left: Scatter plot of $\wh\br_i$ for $i\in [n]$. Several representative companies in each corners are highlighted. Right: The test results of $H_{01}, H_{02}$ and $H_{03}$. Red/blue/green means that $H_{01}$/$H_{02}$/$H_{03}$ is  rejected while grey means that none of these three null hypothesises is rejected.}
\label{simplex}
\end{figure}

Next, we apply our approach in Example \ref{exa2} to construct rank confidence interval for the nodes. We summarize our inference results in Table \ref{table1}. We select $11$ stocks as representatives for presentation. First, we include the three estimated vertices. Second, of the four categories (red/blue/green/grey) shown in Figure \ref{simplex}, we randomly select two from each category and we label them with R/B/G/C (C stands for center). In Table \ref{table1} we include the estimated mixed membership vectors $\wh\pi_i$ for each stock as well as the three $95\%$ rank confidence intervals for $\pi_i(1), \pi_i(2)$ and $\pi_i(3)$, and we denote them by RCI I, RCI II and RCI III. As we can see from Table \ref{table1}, our approach provides meaningful rank confidence intervals for the stocks and the categories which they are close to. 

\begin{table}
	\begin{center}
		\begin{tabular}[!htbp]{ p{2.5cm}|p{3.5cm}p{2cm}p{2cm}p{2cm}}
  \specialrule{.1em}{.05em}{.05em} 
			Symbol & $\wh\bpi_i$ &RCI I & RCI II & RCI III     \\
			\hline
			AAL ($V_I$) & $[1,0,0]$   & $[1,11]$  & $[206, 433]$ & $[382, 433]$  \\ 
			LLY ($V_{II}$) & $[0,1,0]$   & $[256, 433]$ & $[1,13]$& $[381, 433]$  \\ 
			MPC ($V_{III}$) & $[0,0,1]$   & $[137, 433]$ & $[148, 433]$& $[1, 7]$  \\  
            META (C)& $[0.457, 0.177, 0.366]$   & $[96, 127]$  & $[163, 341]$& $[179, 255]$  \\
            EBAY (C)& $[0.379, 0.220, 0.401]$ & $[111, 157]$  &$[150, 277]$ & $[157, 228]$  \\ 
            DHR (R) & $[0.959, 0.017, 0.025]$  & $[1, 19]$ & $[218, 433]$ & $[381, 433]$   \\ 
			HOLX (R)& $[0.933, 0.024, 0.044]$ & $[1, 23]$  & $[227, 433]$ & $[369, 433]$   \\ 
            PNC (B) & $[0.009 , 0.983 , 0.008]$ & $[253, 433]$ & $[1,17]$ & $[380, 433]$  \\ 
            MOS (B) & $[0.037, 0.947, 0.016]$ & $[243, 433]$  & $[1, 23]$ & $[378, 433]$  \\ 
            AEP (G)& $[0.080, 0.063, 0.857]$   & $[190, 433]$ &$[178, 421]$ & $[10, 37]$  \\ 
            LYB (G) & $[0.159, 0.071, 0.771]$  &$[166, 400]$  & $[221, 433]$ & $[35, 51]$   \\ 
   \specialrule{.1em}{.05em}{.05em} 
		\end{tabular}
	\end{center}
 \caption{Estimated mixed membership profile vectors $\wh\pi_i$ and $95\%$ confidence intervals for the ranks of some representative stocks's membership profiles with respect to the three estimated vertices, denoted respectively RC I, RC II, and RC III. AAL, LLY, MPC are the estimated vertices of the three categories. R/B/G/C stands for red/blue/green/grey(center), the color group that has been shown in Figure \ref{simplex}.}
\label{table1}
\end{table}

Next, we investigate whether there has been a change of simplex structure before the COVID-$19$ and after the pandemic began. We use stock data from January 1, 2017 to January 1, 2020 as before COVID-$19$ data, and stock data from May 1, 2020 to May 1, 2023 as after COVID-$19$ data. We follow the same data preprocessing procedure as before, while the threshold for $A_{ij}$ is replaced by $0.12$ when dealing with the after COVID-19 data. And, we also conduct the aforementioned tests \eqref{eqtest} to these two dataset. In Figure \ref{covid} we include the experiment results and the results of tests are also shown by the colors. 

Recall that we have identified three categories with finance, healthcare, and energy as their respective representatives in Figure \ref{simplex}. From the top two plots of Figure \ref{covid} we can see that the `finance corner' (red) and the `energy corner' (blue) are consistent with each other, and these structures are similar to what we have observed in Figure \ref{simplex} except for some companies (e.g., AXP, BRO, ED, LNT, WRP). However, a structural change in another corner brings us some interesting observations. First, the remaining corner (green) in the top left plot (before COVID-$19$) of Figure \ref{covid} is a mix of companies from many different industries, and there is no industry can be considered representative of this corner. Second, the healthcare companies from the `healthcare corner' in Figure \ref{simplex}, are now in the center cluster of the top left plot. The bottom left plot of Figure \ref{covid} is the zoomed plot of the center cluster of the top left plot, and all the companies (EW, HUM, HOLX, ABT, REGN, TMO, LH, ILMN, PFE) which have been identified as the representatives of the `healthcare corner' in Figure \ref{simplex} are there. However, when we look at the after COVID-$19$ plots (top right and bottom right of Figure \ref{covid}), we find that these healthcare companies move from the center cluster to the green corner, and consequently make the green corner a `healthcare corner'. The bottom right plot of Figure \ref{covid} is the zoomed plot of the green corner of the top right plot, and again we can see that all the nine aforementioned healthcare companies are here. To sum up, the `healthcare corner' we have observed in Figure \ref{simplex} does not exist before COVID-$19$. At that time, the healthcare companies are more in the center of the simplex, and this indicates them to be a mix of the different industries. After the pandemic began, the healthcare industry have grown dramatically, and now they become the representative of the third corner along with `finance corner' and `energy corner'.

\begin{figure}[ht]
\begin{center}
\includegraphics[width=\textwidth]{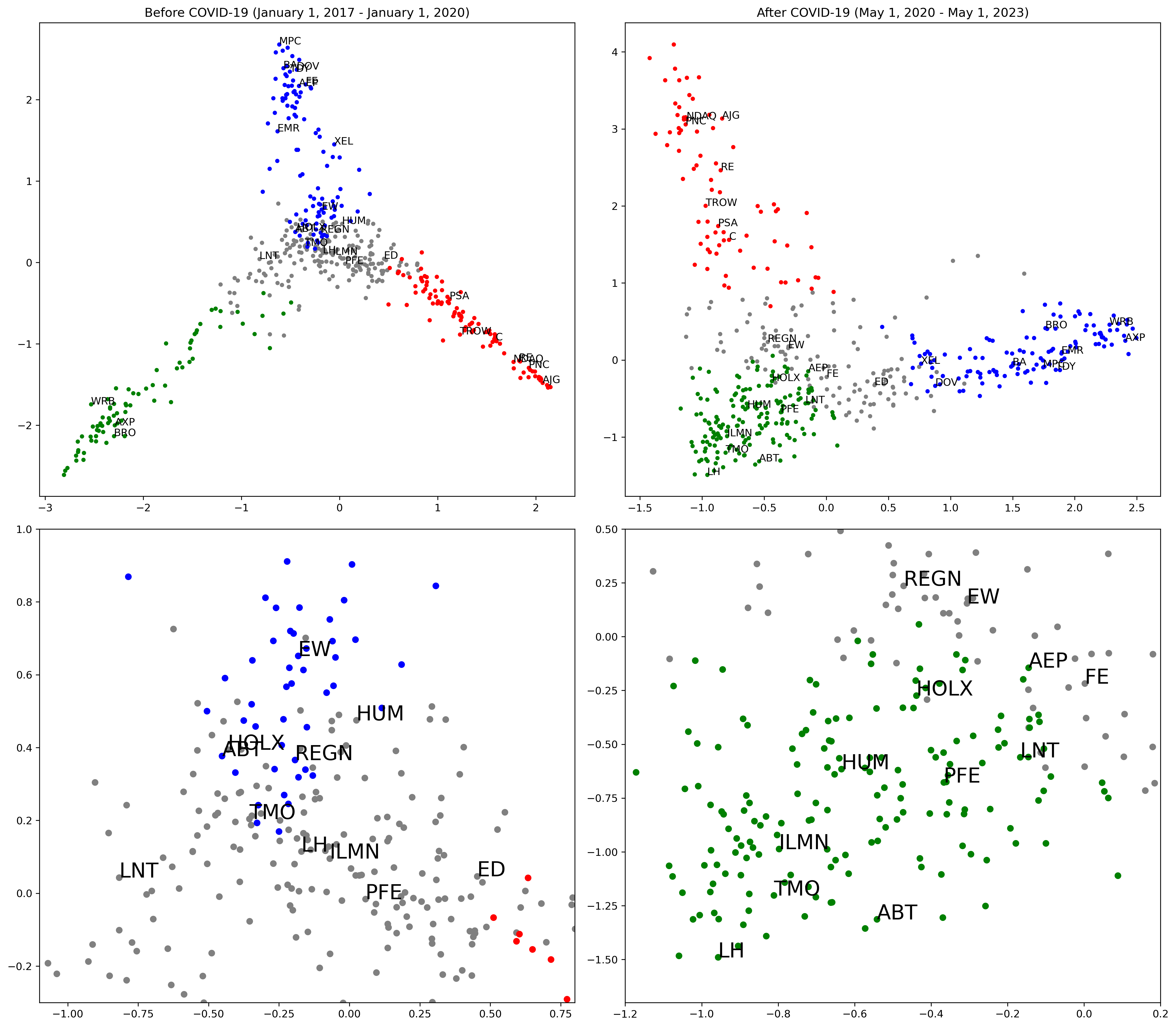} 
\end{center}
\caption{Top left: the simplex structure before COVID-$19$. Top right: the simplex structure after the pandemic began. Red/blue/green means the corresponding $H_{01}$/$H_{02}$/$H_{03}$ is rejected while grey means none of $H_{01}$/$H_{02}$/$H_{03}$ is rejected. Bottom left: the zoomed plot of the center cluster of the top left plot. Bottom right: the zoomed plot of the green corner of the top right plot.}
\label{covid}
\end{figure}

\section{Proof Outline}\label{sec:proof_outline}
In this section, we sketch the main steps of the proof. Details of the proof will be provided in Section \ref{sec:appendix}. We begin by showing an ``eigen-gap" property, which is fundamental to the matrix analysis. To this end, we provide the high-probability bound on the spectral norm of the noise matrix $\bW$.
\begin{lem}\label{Wspectral}
    The following event $\mathcal{A}_1$ happens with probability at least $1-O(n^{-10})$:
    \begin{align}\label{eq:define_a1}
        \mathcal{A}_1 = \left\{\left\|\bW\right\|\leq C_2\theta_{\text{max}}\sqrt{n}\right\}.
    \end{align}
\end{lem}
\begin{proof}
By definition we know that $|W_{ij}|\leq 1$. Assumption \ref{assn:p_max} implies
\begin{align*}
    \max_{i,j}\mathbb{E}\left[W_{ij}^2\right] &= \max_{i,j}H_{ij}\left(1-H_{ij}\right)\leq \max_{i,j}H_{ij}\leq \max_{i,j}\theta_i\theta_j \bpi_i^\top\bP\bpi_j \lesssim \theta_{\text{max}}^2 .
\end{align*}
As a result, by \cite[Theorem 3.4]{chen2021spectral} we know that
\begin{align*}
    \left\|\bW\right\|\lesssim \theta_{\text{max}}\sqrt{n} +\sqrt{\log n}\lesssim \theta_{\text{max}}\sqrt{n}
\end{align*}
with probability at least $1-O(n^{-10})$.
\end{proof}

We are going to prove our results under this favorable set $\mathcal{A}_1$. We begin by listing the key results leading to Theorem \ref{mainthmu1} and Theorem \ref{mainthmmatrixdenoising}. A key assumption in all these results will be either $\theta_{\text{max}}\sqrt{n}\ll \lambda_1^*$ (for results regarding $\wh\bu_1-\bu_1^*$) or $\theta_{\text{max}}\sqrt{n}\ll \minsigma^*$ (for results regarding $\oU\bR-\oU^*$) to guarantee the eigen-gap. From Lemma \ref{eigenvaluelemma}, we already know that these two assumptions are mild. 

We now state the results for $\wh\bu_1-\bu_1^*$. The proofs rely on contour integrals, similar in spirit to \cite{fan2022asymptotic}. 
\begin{thm} \label{u1expansion}
Assume that $\theta_{\text{max}}\sqrt{n}\ll \lambda_1^*$. Under the event $\mathcal{A}_1$ defined in \eqref{eq:define_a1}, we can write the following expansion 
\begin{align}\label{eq:delta_define}
    \wh\bu_1-\bu_1^* = \sum_{i=2}^n \frac{\bu_i^{*\top}\bW\bu_1^*}{\lambda_1^*-\lambda_i^*}\bu_i^* + \boldsymbol{\delta},
\end{align}
where $\|\boldsymbol{\delta}\|_2\lesssim \theta_{\text{max}}^2 n/\lambda_1^{*2}$. 
\end{thm}
\begin{proof}
See Section \ref{u1expansionproof}.
\end{proof}
Theorem \ref{u1expansion} expresses the quantity $\wh\bu_1-\bu_1^*$ as the sum of a leading term (first summand of the RHS of \eqref{eq:delta_define}) and a error-term $\boldsymbol{\delta}$, along with its $l_2$ bound. We also need an $l_\infty$ bound of $\boldsymbol{\delta}$ which is provided by the following Theorem.
\begin{thm}\label{deltainfty}
    Assume that $\max\{\sqrt{n}\theta_{\text{max}}, \log n \}\ll \lambda_1^*$. Then with probability at least $1-O(n^{-10})$, we have
    \begin{align*}
        \left\|\boldsymbol{\delta}\right\|_\infty\lesssim & \frac{\theta_{\text{max}}^3n^{1.5}}{\lambda_1^{*3}}+\frac{\sqrt{n((K-1)\mu^*+\log n)}\theta_{\text{max}}^2}{\lambda_1^{*2}} \nonumber \\
        &+ \frac{\log n \left(\sqrt{\log n}+\sqrt{(K-1)\mu^*}\right)\theta_{\text{max}} + \log^2 n\sqrt{\mu^* / n}}{\lambda_1^{*2}},
    \end{align*}
    where $\delta$ is defined via \eqref{eq:delta_define}.
\end{thm}

\begin{proof}
See Section \ref{deltainftyproof}.
\end{proof}
 Combining Theorem \ref{u1expansion} and Theorem \ref{deltainfty} with triangle inequality provides the next lemma. 
\begin{lem}\label{u1differenceinfty}
Assume that $\max\{\sqrt{n}\theta_{\text{max}}, \log n \}\ll \lambda_1^*$. Then with probability at least $1-O(n^{-10})$, we have
\begin{align*}
    \left\|\wh\bu_1-\bu_1^*\right\|_\infty
    \lesssim & \frac{\left(\sqrt{\log n}+\sqrt{(K-1)\mu^*}\right)\theta_{\text{max}}+\log n\sqrt{\mu^* / n}}{\lambda_1^*} + \frac{\theta_{\text{max}}^3n^{1.5}}{\lambda_1^{*3}}.
\end{align*}
\end{lem}
\begin{proof}
See Section \ref{u1differenceinftyproof}.
\end{proof}
 The combination of Theorem \ref{u1expansion}, Theorem \ref{deltainfty} and Lemma \ref{u1differenceinfty} is exactly Theorem \ref{mainthmu1}.

The next goal is to analyze $\oU\bR-\oU^*$. This is a matrix denoising problem with ground truth $\oH$ and noisy matrix $\oX$, where 
\begin{align*}
    \oH = \sum_{i=2}^K\lambda_i^*\bu_i^*\bu_i^{*\top},\quad \oX = \sum_{i=2}^n\wh\lambda_i\wh\bu_i\wh\bu_i^{\top}.
\end{align*}
Define the noise matrix
\begin{equation}\label{eq:wbar_define}
    \oW := \oX-\oH = \bW -[\wh\lambda_1\wh\bu_1\wh\bu_1^\top - \lambda_1^*\bu_1^*\bu_1^{*\top}] 
\end{equation}
Unlike $\bW$, the matrix $\oW$ does not have a close form expression. The following expansion for $\oW$ will be useful for our results.

\begin{lem}\label{oWexpansion}
Assume that $\sqrt{n}\ll \lambda_1^*$. Then we have
\begin{align*}
   \oW = \bW - \bW\bu_1^*\bu_1^{*\top}\bW\bu_1^*\bu_1^{*\top}- \bN\bW\bu_1^*\bu_1^{*\top}-\left(\bW\bu_1^*\bu_1^{*\top}\right)^\top\bN - \bDelta,
\end{align*}
where $N$ is defined by \eqref{eq:define_n}.

Further, with probability at least $1-O(n^{-10})$ we have
\begin{align*}
    \|\bDelta\|\lesssim n/\lambda_1^* \quad \text{and} \quad \left\|\bDelta\oU^*\right\|_{2,\infty} \lesssim\sqrt{\frac{(K-1)\mu^*n}{\lambda_1^{*2}}}+\frac{n^{1.5}}{\lambda_1^{*2}}.
\end{align*}
\end{lem}
\begin{proof}
See Theorem \ref{firstsubspaceexpansion} and Lemma \ref{lemma5}.
\end{proof}

In order to study the expansion of $\oW$, defined via \eqref{eq:wbar_define}, it is enough to expand $\wh\lambda_1\wh\bu_1\wh\bu_1^\top - \lambda_1^*\bu_1^*\bu_1^{*\top}$. 
We state our result here. The proof, which uses contour integrals, is deferred to Section \ref{firstsubspaceexpansionproof}.
\begin{thm}\label{firstsubspaceexpansion}
Assume that $\sqrt{n}\theta_{\text{max}}\ll \lambda_1^*$. Then under event $\mathcal{A}_1$ defined by \eqref{eq:define_a1}, we have the following expansion:
\begin{align}
    \wh\lambda_1\wh\bu_1\wh\bu_1^\top - \lambda_1^*\bu_1^*\bu_1^{*\top} = \bu_1^*\bu_1^{*\top}\bW\bu_1^*\bu_1^{*\top}+ \bN\bW\bu_1^*\bu_1^{*\top}+\left(\bW\bu_1^*\bu_1^{*\top}\right)^\top\bN + \bDelta,\label{eqfirstsubspaceexpansion}
\end{align}
where $\bN \overset{\Delta}{=} \sum_{i=2}^n\frac{\lambda_1^*}{\lambda_1^*-\lambda_i^*}\bu_i^*\bu_i^{*\top}$ is a symmetric matrix and $\|\bDelta\|\lesssim n\theta_{\text{max}}^2 /\lambda_1^*$.
\end{thm}
 Theorem \ref{firstsubspaceexpansion} shows the first part of Lemma \ref{oWexpansion}, while the second part, the bound for $\|\bDelta\oU^*\|_{2,\infty}$, will be shown later. From Theorem \ref{firstsubspaceexpansion} we can directly deduce the following corollary.
\begin{cor}\label{cor1}
Assume that $\sqrt{n}\theta_{\text{max}}\ll \lambda_1^*$. Then under event $\mathcal{A}_1$, we have
\begin{align*}
     \left\|\wh\lambda_1\wh\bu_1\wh\bu_1^\top - \lambda_1^*\bu_1^*\bu_1^{*\top}\right\|\lesssim \sqrt{n}\theta_{\text{max}}.
\end{align*}
\end{cor}
\begin{proof}
See Section \ref{cor1proof}.
\end{proof}
 Corollary \ref{cor1}, coupled with Lemma \ref{Wspectral}, shows that under event $\mathcal{A}_1$,
\begin{align}\label{eq:wbar_op}
    \|\oW\| \le \|\bW\|+ \left\|\wh\lambda_1\wh\bu_1\wh\bu_1^\top - \lambda_1^*\bu_1^*\bu_1^{*\top}\right\|\lesssim \sqrt{n}\theta_{\text{max}}.
\end{align}
 Equipped with Theorem \ref{u1differenceinfty}, we are ready to prove results regarding matrix denoising. We show the following five results whose combination proves Theorem \ref{mainthmmatrixdenoising}. Note that the first four results here are similar to \cite[Lemma 1-4]{yan2021inference}, while the fifth results is exactly the second part of Lemma \ref{oWexpansion}. To state these results, we need to introduce some notations first. Define
\begin{align}\label{eq:define_L}
    \bL := \oU^\top\oU^*, \text{ } \oV := [\wh\bv_2,\dots,\wh\bv_K],\text{ } \oV^* := [\bv_2^*,\dots,\bv_K^*], \text{ }\oLambda := \textbf{diag}(\wh\lambda_2,\dots,\wh\lambda_K),
\end{align}
where $\wh \bv_i$'s are defined as \eqref{eq:define_vi}. Recall the definition of $R$ from \eqref{Rdefinition}. Here we state the five lemmas whose proofs are deferred.
\begin{lem}\label{lemma1}
Assume that $\sqrt{n}\theta_{\text{max}}\ll \minsigma^*$. Then under event $\mathcal{A}_1$ we have
\begin{align*}
    \left\| \oU\bR -\oU^*\right\|\lesssim \frac{\sqrt{n}\theta_{\text{max}}}{\minsigma^*},\quad \left\|\bL-\bR\right\|\lesssim\frac{n\theta_{\text{max}}^2}{\minsigma^{*2}}\quad \text{and} \quad \frac{1}{2}\leq\sigma_i(\bL)\leq 2, \quad 1\leq i\leq K-1.
\end{align*}
Furthermore, under event $\mathcal{A}_1$ we have
$\oV = \oU\bD$, $\oV^* = \oU^*\bD$, where 
\begin{align}
    \bD = \textbf{diag}(\textbf{sgn}(\lambda_2^*),\dots,\textbf{sgn}(\lambda_K^*))\in\mathcal{O}^{(K-1)\times(K-1)}.\label{Ddefinition}
\end{align}
This immediately implies  $\oV^\top\oV^* = \bD\bL\bD$, $$\bD\bR\bD = \min_{\boldsymbol{O}\in \mathcal{O}^{(K-1)\times (K-1)}}\left\|\oV\boldsymbol{O}-\oV^*\right\|_F,$$
and the same results for $\oV$ and $\oV^*$ are also true. In fact, we have
\begin{align*}
    \left\| \oV\bD \bR\bD-\oV^*\right\| = \left\| \oU\bR -\oU^*\right\|.
\end{align*}
\end{lem}
\begin{proof}
See Section \ref{lemma1proof}.
\end{proof}

\begin{lem}\label{lemma2}
Assume that $\sqrt{n}\theta_{\text{max}}\ll \minsigma^*$ and $n^2\theta_{\text{max}}^2\gtrsim (K-1)\mu^{*2}\log n$. Then with probability exceeding $1-O(n^{-10})$ we have
\begin{align*}
    &\left\|\bR^\top\oLambda\bR - \oLambda^*\right\|\lesssim \frac{\kappa^*n\theta_{\text{max}}^2}{\minsigma^{*}} +\sqrt{(K-1)\log n}\theta_{\text{max}},  \\
    &\left\|\bL^\top\oLambda\bL - \oLambda^*\right\|\lesssim \frac{\theta_{\text{max}}^3 n^{1.5}}{\minsigma^{*2}} +\sqrt{(K-1)\log n}\theta_{\text{max}}.
\end{align*}
\end{lem}
\begin{proof}
See Section \ref{lemma2proof}.
\end{proof}

\begin{lem}\label{lemma3}
Assume that $\sqrt{(K-1)\log n}\theta_{\text{max}}/\minsigma^*+\kappa^* n\theta_{\text{max}}^2/\minsigma^{*2}\ll 1$ and $\max\{\sqrt{n}\theta_{\text{max}}, \log n\}\ll \minsigma^*$. Then with probability at least $1-O(n^{-10})$ we have
\begin{align*}
    \left\|\oU\oLambda\bL-\oX\oU^*\right\|_{2,\infty} 
    \lesssim&  \left(\frac{n \theta_{\text{max}}^2}{\minsigma^*}+\log n\right)\left\|\oU\bL-\oU^*\right\|_{2,\infty}  +\frac{\sqrt{(K-1)n\log n}}{\minsigma^*}\theta_{\text{max}} \\
    &+ \frac{(\sqrt{(K-1)n}\theta_{\text{max}}+\log n) n\theta_{\text{max}}^2}{\lambda_1^*\minsigma^*} + \frac{(\sqrt{n}\theta_{\text{max}}+\log n)^2\sqrt{(K-1)\mu^*/n}}{\minsigma^*} \\
    & + \frac{\kappa^*}{\minsigma^{*}}\sqrt{(K-1)\mu^*n}\theta_{\text{max}}^2.
\end{align*}
\end{lem}
\begin{proof}
See Section \ref{lemma3proof}.
\end{proof}

\begin{lem}\label{lemma4}
Assume that $\sqrt{(K-1)\log n}\theta_{\text{max}}/\minsigma^*+\kappa^* n\theta_{\text{max}}^2/\minsigma^{*2}\ll 1$ and $n\gtrsim\mu^*\max\{\log^2n, K-1\}$. Then with probability at least $1-O(n^{-10})$ we have
\begin{align*}
   \left\|\oU\bL-\oU^*\right\|_{2,\infty}\lesssim &\frac{\kappa^*}{\minsigma^{*2}}\sqrt{(K-1)\mu^*n}\theta_{\text{max}}^2+\frac{\sqrt{(K-1)n\log n}}{\minsigma^{*2}}\theta_{\text{max}} + \frac{\sqrt{K-1} n^{1.5}\theta_{\text{max}}^3}{\lambda_1^*\minsigma^{*2}}   \\
    &+\frac{\sqrt{(K-1)\log n}\theta_{\text{max}} + \log n\sqrt{(K-1)\mu^* / n} + \sqrt{\mu^*}\theta_{\text{max}}}{\minsigma^*}+\frac{n\theta_{\text{max}}^2}{\lambda_1^*\minsigma^*}.
\end{align*}
\end{lem}
\begin{proof}
See Section \ref{lemma4proof}.
\end{proof}

\begin{lem}\label{lemma5}
Assume that $\max\{\sqrt{n}\theta_{\text{max}}, \log n\}\ll \lambda_1^*$. Then with probability at least $1-O(n^{-10})$ we have
\begin{align*}
    \left\|\bDelta\oU^*\right\|_{2,\infty} \lesssim\frac{\sqrt{(K-1)\mu^*n}\theta_{\text{max}}^2}{\lambda_1^*}+\frac{\sqrt{K-1}\mu^*\log ^2n}{n\lambda_1^*}+\frac{n^{1.5}\theta_{\text{max}}^3}{\lambda_1^{*2}}.
\end{align*}
\end{lem}
\begin{proof}
See Section \ref{lemma5proof}.
\end{proof}

 Similar to \cite{yan2021inference}, we prove Theorem \ref{mainthmmatrixdenoising} by combining these five lemmas. The proof of Theorem \ref{mainthmmatrixdenoising} is included in Section \ref{mainthmmatrixdenoisingproof}. Next, we combine Theorem \ref{mainthmu1} and Theorem \ref{mainthmmatrixdenoising} to yield Theorem \ref{mainthmrexpansion}. See Section \ref{mainthmrexpansionproof} for details.

Finally, to obtain the membership reconstruction results in Section \ref{sec:membership_reconstruction}, we need the following result regarding $\wh\lambda_1$, which is a direct corollary of Theorem \ref{firstsubspaceexpansion}.
\begin{cor} \label{membershipreconstructionlem1}
Assume that $\max\{\sqrt{n}\theta_{\text{max}}, \log n\}\ll \lambda_1^*$. Then under event $\mathcal{A}_1$ defined by \eqref{eq:define_a1}, we have the following expansion:
\begin{align*}
    \wh\lambda_1 - \lambda_1^* = \textbf{Tr}\left[\bW\bu_1^*\bu_1^{*\top}+ 2\bN\bW\bu_1^*\bu_1^{*\top}\right]+ \textbf{Tr}\left[\bDelta\right],
\end{align*}
where $|\textbf{Tr}\left[\bDelta\right]|\lesssim n\theta_{\text{max}}^2/\lambda_1^*$. In terms of the estimation error, we have $|\wh\lambda_1-\lambda_1^*|\lesssim \sqrt{n}\theta_{\text{max}}$.
\end{cor}
\begin{proof}
    See Section \ref{membershipreconstructionlem1proof}.
\end{proof}

\bibliographystyle{alpha}
\bibliography{dynamic}

\section{Proofs}\label{sec:appendix}
\subsection{Proof of Theorem \ref{u1expansion}} \label{u1expansionproof}
\begin{proof}
Define $r := \frac{\lambda_1^*-|\lambda_2^*|}{2}$ and let $\mathcal{C}_1$ be the circular contour around $\lambda_1^*$ with radius $r$. Then $\lambda_1^*$ is the only eigenvalue of $\bH$ that is inside $\mathcal{C}_1$. Under event $\mathcal{A}_1$ defined by \eqref{eq:define_a1}, by Weyl's theorem, we know that 
\begin{align*}
    &\wh\lambda_1 = \wh\sigma_1\geq \sigma_1 ^*-C_2\sqrt{n}  = \lambda_1 ^*-C_2\theta_{\text{max}}\sqrt{n},   \\
    &\wh\lambda_i\leq\wh\sigma_{\text{max}}\leq \sigma_{\text{max}} ^*+C_2\theta_{\text{max}}\sqrt{n}  = \max_{2\leq j\leq n}|\lambda_j ^*|+C_2\theta_{\text{max}}\sqrt{n},\text{ for } 2\leq i \leq K. 
\end{align*}

As a result, $\wh\lambda_1$ is the only eigenvalue of $\bX$ that is inside $\mathcal{C}_1$. 
For $\lambda \in \mathbb{C}$, we have
\begin{align}
    \left(\lambda\bI-\bH \right)^{-1} &= \sum_{i=1}^n \frac{1}{\lambda-\lambda_i^*}\bu_i^*\bu_i^{*\top} \label{eq:IHinverse}\\
    \left(\lambda\bI-\bX \right)^{-1} &= \left(\lambda\bI-\bH-\bW \right)^{-1} = \sum_{i=1}^n \frac{1}{\lambda-\wh\lambda_i}\wh\bu_i\wh\bu_i^\top.\label{eq:IXinverse}
\end{align}
As a result, we know that
\begin{align}\label{eq:inverse_spectral_bound}
    \left\|(\lambda \bI-\bH)^{-1}\right\|&= \max_{i\in [n]}\frac{1}{|\lambda-\lambda_i^*|} = \frac{1}{r} \asymp \frac{1}{\lambda_1^*}, \nonumber \\
    \left\|(\lambda \bI-\bX)^{-1}\right\|&= \max_{i\in [n]}\frac{1}{|\lambda-\wh\lambda_i|} \leq \max_{i\in [n]}\frac{1}{|\lambda-\lambda_i^*| - |\wh\lambda_i -\lambda_i^*|}\leq\frac{1}{r - C_2\theta_{\text{max}}\sqrt{n}}\asymp \frac{1}{\lambda_1^*}
\end{align}
under event $\mathcal{A}_1$. Using \eqref{eq:IHinverse} and \eqref{eq:IXinverse} we know that 
\begin{align}\label{eq:define_p1}
    \frac{1}{2\pi i}\oint_{\mathcal{C}_1}\left(\lambda\bI-\bH \right)^{-1}d\lambda = \bu_1^*\bu_1^{*\top}\overset{\Delta}{=}\bP_1^*, \quad \frac{1}{2\pi i}\oint_{\mathcal{C}_1}\left(\lambda\bI-\bX \right)^{-1}d\lambda = \wh\bu_1\wh\bu_1^\top \overset{\Delta}{=}\wh\bP_1.
\end{align}
We denote by 
\begin{align*}
    \Delta\bP_1 = \frac{1}{2\pi i}\oint_{\mathcal{C}_1}\left(\lambda\bI-\bH \right)^{-1}\bW\left(\lambda\bI-\bH \right)^{-1}d\lambda.
\end{align*}
Then we know that
\begin{align*}
    \left\|\wh\bP_1-\bP_1^*-\Delta\bP_1\right\|\leq\frac{1}{2\pi}\oint_{\mathcal{C}_1}\left\|\left(\lambda\bI-\bX \right)^{-1}-\left(\lambda\bI-\bH \right)^{-1} - \left(\lambda\bI-\bH \right)^{-1}\bW\left(\lambda\bI-\bH \right)^{-1}\right\| d\lambda.
\end{align*}
The integrand can be reformulated as,
\begin{align*}
    &\left(\lambda\bI-\bX \right)^{-1}-\left(\lambda\bI-\bH \right)^{-1} - \left(\lambda\bI-\bH \right)^{-1}\bW\left(\lambda\bI-\bH \right)^{-1} \\
    =&\left(\lambda\bI-\bX \right)^{-1} \left[\left(\lambda\bI-\bH \right)-\left(\lambda\bI-\bX \right)\right]\left(\lambda\bI-\bH \right)^{-1} - \left(\lambda\bI-\bH \right)^{-1}\bW\left(\lambda\bI-\bH \right)^{-1} \\
    =&\left[\left(\lambda\bI-\bX \right)^{-1}-\left(\lambda\bI-\bH \right)^{-1}\right] \bW\left(\lambda\bI-\bH \right)^{-1}  \\
    =& \left(\lambda\bI-\bX \right)^{-1} \bW\left(\lambda\bI-\bH \right)^{-1}\bW\left(\lambda\bI-\bH \right)^{-1}.
\end{align*}
As a result, we have, under event $\mathcal{A}_1$,
\begin{align*}
    &\left\|\left(\lambda\bI-\bX \right)^{-1}-\left(\lambda\bI-\bH \right)^{-1} - \left(\lambda\bI-\bH \right)^{-1}\bW\left(\lambda\bI-\bH \right)^{-1}\right\| \\
    \leq & \left\| \left(\lambda\bI-\bX \right)^{-1}\right\|\left\| \left(\lambda\bI-\bH \right)^{-1}\right\|^2\left\| \bW\right\|^2 \lesssim \frac{\theta_{\text{max}}^2 n}{\lambda_1^{*3}},
\end{align*}
where the first inequality uses \eqref{eq:inverse_spectral_bound}. Hence, under event $\mathcal{A}_1$,
\begin{align*}
    \left\|\wh\bP_1-\bP_1^*-\Delta\bP_1\right\|\lesssim\frac{1}{2\pi}\oint_{\mathcal{C}_1}\frac{\theta_{\text{max}}^2 n}{\lambda_1^{*3}}d\lambda  = \frac{\theta_{\text{max}}^2 n r}{\lambda_1^{*3}}\lesssim \frac{\theta_{\text{max}}^2 n}{\lambda_1^{*2}},
\end{align*}
since $r \lesssim \lambda_1$. This immediately yields
\begin{align}\label{eq:T1_bound}
    \left\|\left(\wh\bP_1-\bP_1^*-\Delta\bP_1\right)\bu_1^*\right\|_2\lesssim \frac{\theta_{\text{max}}^2 n}{\lambda_1^{*2}}.
\end{align}
By \eqref{eq:define_p1}, we know that $\bP_1^*\bu_1^* = \bu_1^*$, $\wh\bP_1\bu_1^* = (\wh\bu_1^\top\bu_1^*)\wh\bu_1$. Therefore,
\begin{align*}
   \Delta\bP_1 &= \frac{1}{2\pi i}\oint_{\mathcal{C}_1}\sum_{i=1}^n \sum_{j=1}^n \frac{1}{(\lambda-\lambda_i^*)(\lambda-\lambda_j^*)}\bu_i^*\bu_i^{*\top}\bW\bu_j^*\bu_j^{*\top}d\lambda \\
   & = \sum_{i=1}^n \sum_{j=1}^n \textbf{Res}\left(\frac{1}{(\lambda-\lambda_i^*)(\lambda-\lambda_j^*)},\lambda_1^*\right)\bu_i^*\bu_i^{*\top}\bW\bu_j^*\bu_j^{*\top}   \\
   & = \sum_{i=2}^n \frac{1}{\lambda_1^*-\lambda_i^*}\left(\bu_i^*\bu_i^{*\top}\bW\bu_1^*\bu_1^{*\top}+\bu_1^*\bu_1^{*\top}\bW\bu_i^*\bu_i^{*\top}\right).
\end{align*}
As a result, we know that
\begin{align*}
    \Delta\bP_1 \bu_1^* = \sum_{i=2}^n\frac{\bu_i^{*\top}\bW\bu_1^*}{\lambda_1^*-\lambda_i^*}\bu_i^*.
\end{align*}
Therefore, we obtain,
\begin{align}
    \wh\bu_1 - \bu_1^* - \sum_{i=2}^n\frac{\bu_i^{*\top}\bW\bu_1^*}{\lambda_1^*-\lambda_i^*}\bu_i^* & =\left((\wh\bu_1^\top\bu_1^*)\wh\bu_1 - \bu_1^* - \sum_{i=2}^n\frac{\bu_i^{*\top}\bW\bu_1^*}{\lambda_1^*-\lambda_i^*}\bu_i^*\right) + \left((\wh\bu_1^\top\bu_1^*)\wh\bu_1 - \wh\bu_1\right) \nonumber \\
    &= \underbrace{\left(\wh\bP_1-\bP_1^*-\Delta\bP_1\right)\bu_1^*}_{T_1}+ \underbrace{\left((\wh\bu_1^\top\bu_1^*)\wh\bu_1 - \wh\bu_1\right)}_{T_2}.\label{eq:u1_decompose}
\end{align}
Now, $\|T_1\|_2\lesssim \theta_{\text{max}}^2 n/\lambda_1^{*2}$ by ~\eqref{eq:T1_bound}.
To bound $\|T_2\|_2$, it is enough to bound
$\|(\wh\bu_1^\top\bu_1^*)\wh\bu_1 - \wh\bu_1\|_2 = |\wh\bu_1^\top\bu_1^*-1|$. To this end, note that, 
\begin{align}\label{eq:u1_1}
    |\wh\bu_1^\top\bu_1^*-1| =1-\wh\bu_1^\top\bu_1^* =  \frac{\bu_1^{*\top}\bu_1^*+\wh\bu_1^{\top}\wh\bu_1-\wh\bu_1^\top\bu_1^*-\bu_1^{*\top}\wh\bu_1}{2} = \frac{\left\|\bu_1^*-\wh\bu_1\right\|_2^2}{2},
\end{align}
Also, by Wedin’s sin$\Theta$ Theorem \cite[Theorem 2.9]{chen2021spectral},
\begin{align}\label{eq:u1_2}
    \left\|\bu_1^*-\wh\bu_1\right\|_2\lesssim\frac{\left\|\bW\right\|}{\lambda_1^*}\lesssim\frac{\theta_{\text{max}}\sqrt{n}}{\lambda_1^*},
\end{align}
under $\mathcal{A}_1$. Combining ~\eqref{eq:u1_1} and ~\eqref{eq:u1_2}, we obtain
\begin{align*}
    \|(\wh\bu_1^\top\bu_1^*)\wh\bu_1 - \wh\bu_1\|_2 = \frac{\left\|\bu_1^*-\wh\bu_1\right\|_2^2}{2}\lesssim \frac{\theta_{\text{max}}^2 n}{\lambda_1^{*2}}.
\end{align*}
Therefore, we get, under the event $\mathcal{A}_1$,
\begin{align*}
   \|\boldsymbol{\delta}\|_2 &\le \left\|\wh\bu_1 - \bu_1^* - \sum_{i=2}^n\frac{\bu_i^{*\top}\bW\bu_1^*}{\lambda_1^*-\lambda_i^*}\bu_i^* \right\|_2\\
   &\leq\left\| \left(\wh\bP_1-\bP_1^*-\Delta\bP_1\right)\bu_1^*\right\|_2 + \|(\wh\bu_1^\top\bu_1^*)\wh\bu_1 - \wh\bu_1\|_2\lesssim \frac{\theta_{\text{max}}^2 n}{\lambda_1^{*2}}.
\end{align*}
\end{proof}

\subsection{Proof of Theorem \ref{deltainfty}}\label{deltainftyproof}
\begin{proof}
Recall from ~\eqref{eq:u1_decompose} that
\begin{align}
    \boldsymbol{\delta} = \wh\bu_1-\bu_1^*-\Delta\bP_1\bu_1^* = \left(\wh\bP_1-\bP^\star_1-\Delta\bP_1\right)\bu_1^* +(1-\wh\bu_1^\top\bu_1^*)\wh\bu_1.\label{deltadecomposition}
\end{align}
We begin with bounding $\|(\wh\bP_1-\bP^\star_1-\Delta\bP_1)\bu_1^*\|_{\infty}$. Recall that,
\begin{align*}
    \wh\bP_1-\bP_1^*-\Delta\bP_1 & 
    =\frac{1}{2\pi i}\oint_{\mathcal{C}_1}\left(\lambda\bI-\bX \right)^{-1} \bW\left(\lambda\bI-\bH \right)^{-1}\bW\left(\lambda\bI-\bH \right)^{-1} d\lambda.
\end{align*}
We split the integrand as a sum of the following two quantities,
\begin{align*}
    \bDelta_1 &:= \left(\lambda\bI-\bH \right)^{-1} \bW\left(\lambda\bI-\bH \right)^{-1}\bW\left(\lambda\bI-\bH \right)^{-1};\\
    \bDelta_2 &:= \left[\left(\lambda\bI-\bX \right)^{-1} -\left(\lambda\bI-\bH \right)^{-1} \right]\bW\left(\lambda\bI-\bH \right)^{-1}\bW\left(\lambda\bI-\bH \right)^{-1} \\
    & = \left(\lambda\bI-\bX \right)^{-1}\bW\left(\lambda\bI-\bH \right)^{-1}\bW\left(\lambda\bI-\bH \right)^{-1}\bW\left(\lambda\bI-\bH \right)^{-1},
\end{align*}
so that $\wh\bP_1-\bP_1^*-\Delta\bP_1 = \frac{1}{2\pi i}\oint_{\mathcal{C}_1}\left(\bDelta_1+\bDelta_2\right) d\lambda$. Under the event $\mathcal{A}_1$,
\begin{align}
    \left\|\left[\frac{1}{2\pi i}\oint_{\mathcal{C}_1}\bDelta_2 d\lambda\right] \bu_1^*\right\|_{\infty} &\leq \left\|\left[\frac{1}{2\pi i}\oint_{\mathcal{C}_1}\bDelta_2 d\lambda\right] \bu_1^*\right\|_{2} \leq \left\|\frac{1}{2\pi i}\oint_{\mathcal{C}_1}\bDelta_2 d\lambda \right\| \nonumber \\
    &\leq \frac{1}{2\pi}\oint_{\mathcal{C}_1}\left\|\bDelta_2\right\| d\lambda \lesssim\frac{\theta_{\text{max}}^3 n^{1.5}r}{\lambda_1^{*4}}\lesssim \frac{\theta_{\text{max}}^3 n^{1.5}}{\lambda_1^{*3}},\label{ointDelta2}
\end{align}
where the fourth inequality uses \eqref{eq:inverse_spectral_bound}. It remains to bound $\left\|\left[\frac{1}{2\pi i}\oint_{\mathcal{C}_1}\bDelta_1 d\lambda\right] \bu_1^*\right\|_{\infty}$.
Since $\bDelta_1$ can be expanded as
\begin{align*}
    \bDelta_1 = \sum_{i,j,k=1}^n \frac{1}{(\lambda-\lambda_i^*)(\lambda-\lambda_j^*)(\lambda-\lambda_k^*)} \bu_i^*\bu_i^{*\top}\bW\bu_j^*\bu_j^{*\top}\bW\bu_k^*\bu_k^{*\top}, 
\end{align*}
we have
\begin{align*}
    \frac{1}{2\pi i}\oint_{\mathcal{C}_1}\bDelta_1 d\lambda = \sum_{i,j,k=1}^n\textbf{Res}\left(\frac{1}{(\lambda-\lambda_i^*)(\lambda-\lambda_j^*)(\lambda-\lambda_k^*)},\lambda_1^*\right)\bu_i^*\bu_i^{*\top}\bW\bu_j^*\bu_j^{*\top}\bW\bu_k^*\bu_k^{*\top}.
\end{align*}
Since $\bu_k^{*\top}\bu_1^*= \mathbbm{1}_{k=1}$, we have 
\begin{align*}
    \left[\frac{1}{2\pi i}\oint_{\mathcal{C}_1}\bDelta_1 d\lambda\right] \bu_1^* =& \sum_{i,j=1}^n\textbf{Res}\left(\frac{1}{(\lambda-\lambda_i^*)(\lambda-\lambda_j^*)(\lambda-\lambda_1^*)},\lambda_1^*\right)\bu_i^*\bu_i^{*\top}\bW\bu_j^*\bu_j^{*\top}\bW\bu_1^*  \\
    =&\sum_{i,j=2}^n\frac{1}{(\lambda_1^*-\lambda_i^*)(\lambda_1^*-\lambda_j^*)}\bu_i^*\bu_i^{*\top}\bW\bu_j^*\bu_j^{*\top}\bW\bu_1^* \\
    &-\sum_{i=2}^n\frac{1}{(\lambda_1^*-\lambda_i^*)^2}\left[\bu_i^*\bu_i^{*\top}\bW\bu_1^*\bu_1^{*\top}\bW\bu_1^*+\bu_1^*\bu_1^{*\top}\bW\bu_i^*\bu_i^{*\top}\bW\bu_1^*\right].
\end{align*}
Define the matrices
\begin{align}
    \bN_1 := \sum_{i=2}^n\frac{1}{\lambda_1^*-\lambda_i^*}\bu_i^*\bu_i^{*\top},\quad \bN_2 := \sum_{i=2}^n\frac{1}{(\lambda_1^*-\lambda_i^*)^2}\bu_i^*\bu_i^{*\top}.\label{N1N2}
\end{align}
 Then we can write
\begin{align}
    \left[\frac{1}{2\pi i}\oint_{\mathcal{C}_1}\bDelta_1 d\lambda\right] \bu_1^* = \bN_1\bW\bN_1\bW\bu_1^*-\bN_2\bW\bu_1^*\bu_1^{*\top}\bW\bu_1^*-\bu_1^*\bu_1^{*\top}\bW\bN_2\bW\bu_1^*.\label{ointDelta1}
\end{align}
We analyse the three terms separately now.
\begin{itemize}
    \item[(i)] \textbf{Control} $\left\|\bN_1\bW\bN_1\bW\bu_1^*\right\|_\infty$: We introduce leave-one-out matrix $\bW^{(i)}$ by replacing all the elements in $i$-th row and $i$-th column of original $\bW$ with $0$ for $i\in [n]$. Then, for any $i \in [n]$, $\bW_{i, \cdot}$ is independent of $\bW^{(i)}$. We write
    \begin{align}
        \left\|\bW\bN_1\bW\bu_1^*\right\|_\infty &= \max_{1\leq i\leq n}\left|\bW_{i,\cdot}\bN_1\bW\bu_1^*\right| \nonumber \\
        &\leq \max_{1\leq i\leq n}\left\{\left|\bW_{i,\cdot}\bN_1\bW^{(i)}\bu_1^*\right| +\left|\bW_{i,\cdot}\bN_1\Big(\bW-\bW^{(i)}\Big)\bu_1^*\right|\right\}. \label{u1leaveoneouteq1}
    \end{align}
    For a fixed $i\in [n]$, by Lemma \ref{Wconcentration1} and Lemma \ref{lemmaN1N2}, we obtain
    \begin{align}
        \left|\bW_{i,\cdot}\bN_1\bW^{(i)}\bu_1^*\right|&\lesssim \sqrt{\log n}\theta_{\text{max}}\Big\|\bN_1\bW^{(i)}\bu_1^*\Big\|_2 +\log n \Big\|\bN_1\bW^{(i)}\bu_1^*\Big\|_\infty  \nonumber \\
        &\lesssim \frac{\sqrt{n\log n}\theta_{\text{max}}^2 }{\lambda_1^*} +\log n \left\|\bN_1\bW^{(i)}\bu_1^*\right\|_\infty \label{u1leaveoneouteq2}
    \end{align}
    with probability at least $1-O(n^{-15})$ under $\mathcal{A}_1$. By Corollary \ref{corN1N2} and Lemma \ref{Wconcentration3}, we have
    \begin{align}
        \left\|\bN_1\bW^{(i)}\bu_1^*\right\|_\infty &\lesssim \frac{1}{\lambda_1^*}\left\|\bW^{(i)}\bu_1^*\right\|_\infty +\sqrt{\frac{(K-1)\mu^*}{n\lambda_1^{*2}}}\left\|\bW^{(i)}\bu_1^*\right\|_2  \nonumber \\
        & \lesssim \frac{\left(\sqrt{\log n}+\sqrt{(K-1)\mu^*}\right)\theta_{\text{max}} + \log n\sqrt{\mu^* / n}}{\lambda_1^*} :=\rho_1\label{u1leaveoneouteq3}
    \end{align}
    with probability at least $1-O(n^{-15})$ under event $\mathcal{A}_1$. Plugging \eqref{u1leaveoneouteq3} in \eqref{u1leaveoneouteq2} tells us 
    \begin{align}
        \left|\bW_{i,\cdot}\bN_1\bW^{(i)}\bu_1^*\right|&\lesssim \frac{\sqrt{n\log n}\theta_{\text{max}}^2}{\lambda_1^*} + \log n \rho_1 =: \eta_1\label{u1leaveoneouteq4}
    \end{align}
    with probability at least $1-O(n^{-15})$ under event $\mathcal{A}_1$.
    
    The second summand in \eqref{u1leaveoneouteq1} can be bounded using Lemma \ref{lemmaN1N2} and Lemma \ref{Wconcentration4} as
    \begin{align}
        \left|\bW_{i,\cdot}\bN_1\Big(\bW-\bW^{(i)}\Big)\bu_1^*\right|&\lesssim \left\|\bW_{i,\cdot}\right\|_2\left\|\bN_1\right\|\left\|\Big(\bW-\bW^{(i)}\Big)\bu_1^*\right\|_2  \nonumber \\
        &\lesssim \frac{\left\|\bW\right\|}{\lambda_1^*}\left(\sqrt{\log n}\theta_{\text{max}}+\left(\log n+\sqrt{n}\theta_{\text{max}}\right)\sqrt{\mu^*/n}\right) \nonumber \\
        &\lesssim \frac{\sqrt{n(\mu^*+\log n)}\theta_{\text{max}}^2+\log n\sqrt{\mu^*}\theta_{\text{max}}}{\lambda_1^*} \label{u1leaveoneouteq5}
    \end{align}
    with probability at least $1-O(n^{-15})$. Plugging \eqref{u1leaveoneouteq4} and \eqref{u1leaveoneouteq5} in ~\eqref{u1leaveoneouteq1} we get
    \begin{align*}
    \left\|\bW\bN_1\bW\bu_1^*\right\|_\infty\lesssim &\frac{\sqrt{n \mu^*}\theta_{\text{max}}^2}{\lambda_1^*} + \eta_1
    \end{align*}
    with probability at least $1-O(n^{-10})$. Then by Cororllary \ref{corN1N2} we have
    \begin{align*}
        \left\|\bN_1\bW\bN_1\bW\bu_1^*\right\|_\infty\lesssim &\frac{1}{\lambda_1^*}\left\|\bW\bN_1\bW\bu_1^*\right\|_\infty +\sqrt{\frac{(K-1)\mu^*}{n\lambda_1^{*2}}}\left\|\bW\bN_1\bW\bu_1^*\right\|_2 \\
        \lesssim &\frac{1}{\lambda_1^*}\left\|\bW\bN_1\bW\bu_1^*\right\|_\infty+\sqrt{\frac{(K-1)\mu^*}{n\lambda_1^{*2}}}\left\|\bW\right\|^2\left\|\bN_1\right\|\left\|\bu_1^*\right\|_2\\
        \lesssim & \frac{\sqrt{n \mu^*}\theta_{\text{max}}^2}{\lambda_1^{*2}} + \frac{\eta_1}{\lambda_1^* } + \frac{\sqrt{(K-1)\mu^*}}{\sqrt{n}\lambda_1^*} n\theta_{\text{max}}^2\frac{1}{\lambda_1^*}\\
        \lesssim & \frac{\sqrt{n(K-1)\mu^*}\theta_{\text{max}}^2}{\lambda_1^{*2}} +\frac{\eta_1}{\lambda_1^*}
    \end{align*}
    with probability at least $1-O(n^{-10})$.
    \item[(ii)] \textbf{Control} $\left\|\bN_2\bW\bu_1^*\bu_1^{*\top}\bW\bu_1^*\right\|_\infty$: First, by Lemma \ref{Wconcentration5}, with probability at least $1-O(n^{-15})$, we have $|\bu_1^{*\top}\bW\bu_1^*|\lesssim \sqrt{\log n}$. Second, by Lemma \ref{Wconcentration3} we have
    \begin{align}
        \left\|\bW\bu_1^*\right\|_\infty \lesssim  \sqrt{\log n}\theta_{\text{max}}+\log n\sqrt{\mu^*/n}\label{u1leaveoneouteq6}
    \end{align}
    with probability at least $1-O(n^{-14})$. Using the definition of $\rho_1$ from \eqref{u1leaveoneouteq3}, 
    \begin{align*}
        &\left\|\bN_2\bW\bu_1^*\bu_1^{*\top}\bW\bu_1^*\right\|_\infty = \left|\bu_1^{*\top}\bW\bu_1^*\right|\left\|\bN_2\bW\bu_1^*\right\|_\infty\lesssim \sqrt{\log n}\left\|\bN_2\bW\bu_1^*\right\|_\infty \\
        &\lesssim\sqrt{\log n} \Big(\frac{1}{\lambda_1^{*2}}\left\|\bW\bu_1^*\right\|_\infty+\sqrt{\frac{(K-1)\mu^*}{n\lambda_1^{*4}}}\left\|\bW\bu_1^*\right\|_2\Big) \lesssim \frac{\sqrt{\log n}}{\lambda_1^{*}}\rho_1
    \end{align*}
    with probability at least $1-O(n^{-10})$, where the second inequality uses Corollary \ref{corN1N2}.
    \item[(iii)] \textbf{Control} $\left\|\bu_1^*\bu_1^{*\top}\bW\bN_2\bW\bu_1^*\right\|_\infty$: Using Lemma \ref{lemmaN1N2}, under the event $\mathcal{A}_1$,
    \begin{align*}
        \left\|\bu_1^*\bu_1^{*\top}\bW\bN_2\bW\bu_1^*\right\|_\infty &= \left\|\bu_1^*\right\|_\infty\left|\bu_1^{*\top}\bW\bN_2\bW\bu_1^*\right|\\
        &\leq \left\|\bu_1^*\right\|_\infty\left|\bu_1^*\right\|_2^2\left\|\bW\right\|^2\left\|\bN_2\right\| \leq \frac{\sqrt{\mu^*n}\theta_{\text{max}}^2}{\lambda_1^{*2}}
    \end{align*}
    with probability at least $1-O(n^{-10})$, where the final inequality uses  \eqref{eq:incoherence}.
\end{itemize}
Combine these three parts with~\eqref{ointDelta1}, we get
\begin{align}
    \left\|\left[\frac{1}{2\pi i}\oint_{\mathcal{C}_1}\bDelta_1 d\lambda\right] \bu_1^*\right\|_\infty\lesssim&\frac{\sqrt{n(K-1)\mu^*}\theta_{\text{max}}^2}{\lambda_1^{*2}} + \frac{\eta_1}{\lambda_1} =: \eta_2.\label{ointDelta1bound}
\end{align}
with probability at least $1-O(n^{-10})$. Combining ~\eqref{ointDelta1bound} and ~\eqref{ointDelta2}, we get
\begin{align}
    \left\|(\wh\bP_1-\bP_1-\Delta\bP_1)\bu_1^*\right\|_{\infty}\lesssim&\frac{\theta_{\text{max}}^3n^{1.5}}{\lambda_1^{*3}}+\eta_2.\label{deltadecomposition1}
\end{align}

 It remains to bound $\|(1-\wh\bu_1^\top\bu_1^*)\wh\bu_1\|_\infty$. To this end, using \eqref{eq:u1_1} and \eqref{eq:u1_2},

\begin{align}
    \|(1-\wh\bu_1^\top\bu_1^*)\wh\bu_1\|_\infty &= |1-\wh\bu_1^\top\bu_1^*|\|\wh\bu_1\|_\infty\lesssim \frac{n\theta_{\text{max}}^2}{\lambda_1^{*2}}\left(\left\|\bu_1^*\right\|_\infty+\left\|\wh\bu-\bu_1^*\right\|_\infty\right) \nonumber \\
    &\leq \frac{n\theta_{\text{max}}^2}{\lambda_1^{*2}}\left(\sqrt{\frac{\mu^*}{n}}+\left\|\wh\bu-\bu_1^*\right\|_2\right)\lesssim \frac{\sqrt{\mu^*n}\theta_{\text{max}}^2}{\lambda_1^{*2}}+\frac{\theta_{\text{max}}^3 n^{1.5}}{\lambda_1^{*3}},\label{deltadecomposition2}
\end{align}
where the second inequality is by \eqref{eq:incoherence}. Plugging  \eqref{deltadecomposition1} and \eqref{deltadecomposition2} in \eqref{deltadecomposition}, we obtain
\begin{align*}
    \left\|\boldsymbol{\delta}\right\|_\infty\leq \left\|(\wh\bP_1-\bP_1-\Delta\bP_1)\bu_1^*\right\|_{\infty}+\|(1-\wh\bu_1^\top\bu_1^*)\wh\bu_1\|_\infty
    \lesssim \frac{\theta_{\text{max}}^3n^{1.5}}{\lambda_1^{*3}}+\eta_2
\end{align*}
with probability at least $1-O(n^{-10})$, completing the proof.
\end{proof}

\subsection{Proof of Lemma \ref{u1differenceinfty}} \label{u1differenceinftyproof}
\begin{proof}
By Theorem \ref{u1expansion} we know that $\wh\bu_1-\bu_1^* = \bN_1\bW\bu_1+\boldsymbol{\delta}$. By Lemma \ref{Wspectral}, Corollary \ref{corN1N2} and \eqref{u1leaveoneouteq6} we know with probability at least $1-O(n^{-10})$,
\begin{align*}
    \left\|\bN_1\bW\bu_1\right\|_{\infty}&\leq \frac{1}{\lambda_1^*}\left\|\bW\bu_1\right\|_{\infty}+\sqrt{\frac{(K-1)\mu^*}{n\lambda_1^{*2}}}\left\|\bW\bu_1^*\right\|_2 \\
    &\leq \frac{1}{\lambda_1^*}\left\|\bW\bu_1\right\|_{2, \infty}+\sqrt{\frac{(K-1)\mu^*}{n\lambda_1^{*2}}}\left\|\bW\right\| \lesssim \rho_1,
\end{align*}
where $\rho_1$ is defined by \eqref{u1leaveoneouteq3}. Combine this with Theorem \ref{deltainfty}, we get, with probability at least $1-O(n^{-10})$,
\begin{align*}
    \left\|\wh\bu_1-\bu_1^*\right\|_\infty\leq \left\|\bN_1\bW\bu_1\right\|_\infty +\left\|\boldsymbol{\delta}\right\|_\infty
    \lesssim \rho_1 + \frac{\theta_{\text{max}}^3n^{1.5}}{\lambda_1^{*3}} + \frac{\rho_1 \log n}{\lambda_1}
    \lesssim \rho_1 + \frac{\theta_{\text{max}}^3n^{1.5}}{\lambda_1^{*3}}
\end{align*}
by our assumption $\lambda^\star_1 \gg \log n$. This completes the proof.
\end{proof}

\subsection{Proof of Theorem \ref{firstsubspaceexpansion}} \label{firstsubspaceexpansionproof}
\begin{proof}
We use the same $r$ and $\mathcal{C}_1$ as in the proof of Theorem \ref{u1expansion}. Since 
\begin{align*}
    \bH\left(\lambda\bI-\bH \right)^{-1} &= \sum_{i=1}^n \frac{\lambda_i^*}{\lambda-\lambda_i^*}\bu_i^*\bu_i^{*\top}, \quad
    \bX\left(\lambda\bI-\bX \right)^{-1} = \sum_{i=1}^n \frac{\wh\lambda_i}{\lambda-\wh\lambda_i}\wh\bu_i\wh\bu_i^\top,
\end{align*}
for the same reason as the proof of Theorem \ref{u1expansion}, under event $\mathcal{A}_1$, we have
\begin{align*}
    \wh\lambda_1\wh\bu_1\wh\bu_1^\top - \lambda_1^*\bu_1^*\bu_1^{*\top} &= \frac{1}{2\pi i}\oint_{\mathcal{C}_1}\bX\left(\lambda\bI-\bX \right)^{-1}d\lambda-\frac{1}{2\pi i}\oint_{\mathcal{C}_1}\bH\left(\lambda\bI-\bH \right)^{-1}d\lambda \\
    &=\frac{1}{2\pi i}\oint_{\mathcal{C}_1}\left[\bX\left(\lambda\bI-\bX \right)^{-1}-\bH\left(\lambda\bI-\bH \right)^{-1}\right]d\lambda.
\end{align*}
We expand the integrand as sum of two quantitites in the following way:
\begin{align}
    &\bX\left(\lambda\bI-\bX \right)^{-1}-\bH\left(\lambda\bI-\bH \right)^{-1}   \nonumber \\
    =& \left(\bX-\bH\right)\left(\lambda\bI-\bX \right)^{-1}+\bH\left[\left(\lambda\bI-\bX \right)^{-1}-\left(\lambda\bI-\bH \right)^{-1}\right] \nonumber \\
    =&\bW\left(\lambda\bI-\bX \right)^{-1} + \bH \left(\lambda\bI-\bX \right)^{-1}\bW\left(\lambda\bI-\bH \right)^{-1} \nonumber \\
    =& \bW\left(\lambda\bI-\bH \right)^{-1} + \bH \left(\lambda\bI-\bH \right)^{-1}\bW\left(\lambda\bI-\bH \right)^{-1} \nonumber \\
    &+\bW\left[\left(\lambda\bI-\bX \right)^{-1}-\left(\lambda\bI-\bH \right)^{-1}\right]+\bH \left[\left(\lambda\bI-\bX \right)^{-1}-\left(\lambda\bI-\bH \right)^{-1}\right]\bW\left(\lambda\bI-\bH \right)^{-1} \nonumber \\
    =&\underbrace{\bW\left(\lambda\bI-\bH \right)^{-1} + \bH \left(\lambda\bI-\bH \right)^{-1}\bW\left(\lambda\bI-\bH \right)^{-1}}_{\bDelta_1} \nonumber \\
    &+ \underbrace{\bW \left(\lambda\bI-\bX \right)^{-1}\bW\left(\lambda\bI-\bH \right)^{-1}+\bH \left(\lambda\bI-\bX \right)^{-1}\bW\left(\lambda\bI-\bH \right)^{-1}\bW\left(\lambda\bI-\bH \right)^{-1}}_{\bDelta_2}. \label{firstsubspaceeq1}
\end{align}
For $\bDelta_1$, the contour integral can be calculated as
\begin{align*}
    \frac{1}{2\pi i}\oint_{\mathcal{C}_1}\bDelta_1 d\lambda =& \frac{1}{2\pi i}\oint_{\mathcal{C}_1}\left(\sum_{i=1}^n\frac{1}{\lambda-\lambda_i^*}\bW\bu_i^*\bu_i^{*\top}+\sum_{i=1}^n \sum_{j=1}^n \frac{\lambda_i^*}{(\lambda-\lambda_i^*)(\lambda-\lambda_j^*)}\bu_i^*\bu_i^{*\top}\bW\bu_j^*\bu_j^{*\top} \right) d\lambda \\
    =& \sum_{i=1}^n\textbf{Res}\left(\frac{1}{\lambda-\lambda_i^*},\lambda_1^*\right)\bW\bu_i^*\bu_i^{*\top} \\
    &+ \sum_{i,j=1}^n\textbf{Res}\left(\frac{\lambda_i^*}{(\lambda-\lambda_i^*)(\lambda-\lambda_j^*)},\lambda_1^*\right)\bu_i^*\bu_i^{*\top}\bW\bu_j^*\bu_j^{*\top} \\
    = & \bW\bu_1^*\bu_1^{*\top} +\sum_{i=2}^n\frac{\lambda_i^*}{\lambda_1^*-\lambda_i^*}\bu_i^*\bu_i^{*\top}\bW\bu_1^*\bu_1^{*\top}+
    \sum_{i=2}^n\frac{\lambda_1^*}{\lambda_1^*-\lambda_i^*}\bu_1^*\bu_1^{*\top}\bW\bu_i^*\bu_i^{*\top}  \\
    = & \left(\bI+\sum_{i=2}^n\frac{\lambda_i^*}{\lambda_1^*-\lambda_i^*}\bu_i^*\bu_i^{*\top} \right)\bW\bu_1^*\bu_1^{*\top} + \left(\bW\bu_1^*\bu_1^{*\top}\right)^\top\sum_{i=2}^n\frac{\lambda_1^*}{\lambda_1^*-\lambda_i^*}\bu_i^*\bu_i^{*\top} \\
    = & \left(\sum_{i=1}^n\bu_i^*\bu_i^{*\top}+\sum_{i=2}^n\frac{\lambda_i^*}{\lambda_1^*-\lambda_i^*}\bu_i^*\bu_i^{*\top} \right)\bW\bu_1^*\bu_1^{*\top} + \left(\bW\bu_1^*\bu_1^{*\top}\right)^\top\bN \\
    = & \left(\bu_1^*\bu_1^{*\top}+\sum_{i=2}^n\frac{\lambda_1^*}{\lambda_1^*-\lambda_i^*}\bu_i^*\bu_i^{*\top} \right)\bW\bu_1^*\bu_1^{*\top} + \left(\bW\bu_1^*\bu_1^{*\top}\right)^\top\bN  \\
    = & \bu_1^*\bu_1^{*\top}\bW\bu_1^*\bu_1^{*\top}+ \bN\bW\bu_1^*\bu_1^{*\top}+\left(\bW\bu_1^*\bu_1^{*\top}\right)^\top\bN.
\end{align*}
Next, we bound the spectral norm of $\bDelta_2$ under the event $\mathcal{A}_1$, as:
\begin{align*}
    \left\|\bDelta_2\right\| &\leq \left\|\bW \left(\lambda\bI-\bX \right)^{-1}\bW\left(\lambda\bI-\bH \right)^{-1}\right\| +\left\|\bH \left(\lambda\bI-\bX \right)^{-1}\bW\left(\lambda\bI-\bH \right)^{-1}\bW\left(\lambda\bI-\bH \right)^{-1}\right\| \\
    &\leq \left\|\bW\right\|^2\left\|\left(\lambda\bI-\bX \right)^{-1}\right\|\left\|\left(\lambda\bI-\bH \right)^{-1}\right\| +\left\|\bH\right\| \left\|\left(\lambda\bI-\bX \right)^{-1}\right\| \left\|\bW\right\|^2\left\|\left(\lambda\bI-\bH \right)^{-1}\right\|^2 \\
    &\lesssim \frac{n\theta_{\text{max}}^2}{\lambda_1^{*2}}+\frac{\lambda_1^* n\theta_{\text{max}}^2}{\lambda_1^{*3}}\lesssim \frac{n\theta_{\text{max}}^2}{\lambda_1^{*2}},
\end{align*}
where the third inequality \eqref{eq:inverse_spectral_bound}. As a result, the contour integral of $\bDelta_2$, can be bounded as 
\begin{align*}
    \left\|\frac{1}{2\pi i}\oint_{\mathcal{C}_1}\bDelta_2d\lambda\right\|\lesssim \frac{1}{2\pi}\oint_{\mathcal{C}_1}\frac{n\theta_{\text{max}}^2}{\lambda_1^{*2}}d\lambda  = \frac{n\theta_{\text{max}}^2 r}{\lambda_1^{*2}}\lesssim\frac{n\theta_{\text{max}}^2}{\lambda_1^*}.
\end{align*}
Noting that the contour integral of $\bDelta_2$ is exactly
\begin{align*}
    \wh\lambda_1\wh\bu_1\wh\bu_1^\top - \lambda_1^*\bu_1^*\bu_1^{*\top}-\left[\bu_1^*\bu_1^{*\top}\bW\bu_1^*\bu_1^{*\top}+ \bN\bW\bu_1^*\bu_1^{*\top}+\left(\bW\bu_1^*\bu_1^{*\top}\right)^\top\bN\right]
\end{align*}
gives us the desired result.
\end{proof}

\subsection{Proof of Corollary \ref{cor1}} \label{cor1proof}
\begin{proof}
Since $\bu_2^*,\bu_3^*,\dots,\bu_K^*$ are orthogonal to each other, we have
\begin{align}\label{eq:N_spectral_bound}
   \left\|\bN\right\|=\left\|\sum_{i=2}^n\frac{\lambda_1^*}{\lambda_1^*-\lambda_i^*}\bu_i^*\bu_i^{*\top}\right\|\lesssim\max_{2\leq i\leq n}\frac{\lambda_1^*}{\lambda_1^*-\lambda_i^*}\lesssim 1.
\end{align}
Since $\|\bW\|\lesssim\sqrt{n}\theta_{\text{max}}$ on the set $\mathcal{A}_1$ we have, by \eqref{firstsubspaceexpansion},
\begin{align*}
    \left\|\wh\lambda_1\wh\bu_1\wh\bu_1^\top - \lambda_1^*\bu_1^*\bu_1^{*\top}\right\|&\leq \left\|\bu_1^*\bu_1^{*\top}\bW\bu_1^*\bu_1^{*\top}\right\|+2\left\|\bW\bu_1^*\bu_1^{*\top}\right\|\left\|\bN\right\|+\left\|\bDelta\right\| \\
    &\leq \left\|\bW\right\|+ 2\left\|\bW\right\|\left\|\bN\right\| +\left\|\bDelta\right\|\lesssim \sqrt{n}\theta_{\text{max}}+\frac{n\theta_{\text{max}}^2}{\lambda_1^*}\lesssim \sqrt{n}\theta_{\text{max}},
\end{align*}
where the third inequality used \eqref{eq:N_spectral_bound}.
\end{proof}

\subsection{Proof of Lemma \ref{lemma1}} \label{lemma1proof}
\begin{proof}
Recall the definitions of $L$ and $R$ from \eqref{eq:define_L} and \eqref{Rdefinition} respectively. We write the SVD of $\bL$ as $\bL = \bY_1(\cos \bOmega)\bY_2^\top$, where $\bY_1,\bY_2\in \mathbb{R}^{(K-1)\times(K-1)}$. Then we have $\bR = \bY_1 \bY_2^\top$. Therefore, we have
\begin{align*}
     \left\| \oU\bR -\oU^*\right\|^2 &= \left\|\oU\bY_1 \bY_2^\top -\oU^* \right\|^2 =  \left\|\left(\oU\bY_1 \bY_2^\top -\oU^*\right)^\top\left(\oU\bY_1 \bY_2^\top -\oU^*\right) \right\| \\
     &=\left\|2\bI-\bY_2\bY_1^\top\oU^\top\oU^*-\oU^{*\top}\oU\bY_1 \bY_2^\top\right\| \\
     &=\left\|2\bI-\bY_2\bY_1^\top\bY_1(\cos \bOmega)\bY_2^\top-\bY_2(\cos \bOmega)\bY_1^\top\bY_1 \bY_2^\top\right\| \\
     &=2\left\|\bI-\bY_2(\cos \bOmega)\bY_2^\top\right\| = 2 \left\|\bI-\cos \bOmega\right\|\leq 2 \left\|\bI-\cos^2 \bOmega\right\| = 2\left\|\sin\bOmega\right\|^2.
\end{align*}
By Wedin's sin$\Theta$ Theorem \cite[Theorem 2.9]{chen2021spectral}, we have
\begin{align*}
    \left\|\sin\bOmega\right\|\leq \textbf{dist}(\oU,\oU^*)\leq  \frac{\sqrt{2}\left\|\oX-\oH\right\|}{\sigma_{K-1}(\oX) - \sigma_{K}(\oH)}=\frac{\sqrt{2}\left\|\oW\right\|}{\sigma_K(\bX) - \sigma_{K+1}(\bH)}.
\end{align*}
Recall that, $\|\oW\| \lesssim \sqrt{n}\theta_{\text{max}}$
under event $\mathcal{A}_1$, by \eqref{eq:wbar_op}. On the other hand, we have
\begin{align*}
    \sigma_K(\bX) - \sigma_{K+1}(\bH) &\geq  \sigma_K(\bX) - \left\|\textbf{diag}(\bTheta\bPi\bP\bPi^\top\bTheta)\right\| \\  &\geq \sigma_K(\bH)-\left\|\bW\right\|- \left\|\textbf{diag}(\bTheta\bPi\bP\bPi^\top\bTheta)\right\|\asymp \minsigma^*.
\end{align*}
Therefore, we have
\begin{align}\label{eq:sintheta_bound}
    \left\| \oU\bR -\oU^*\right\|\lesssim \left\|\sin\bOmega\right\|\lesssim \frac{\sqrt{n}\theta_{\text{max}}}{\minsigma^*}.
\end{align}
Again, by \eqref{eq:sintheta_bound},
\begin{align*}
    \left\|\bL-\bR\right\| = \left\|\bY_1(\cos \bOmega)\bY_2^\top-\bY_1\bY_2^\top\right\| = \left\|\bI-\cos \bOmega\right\|\leq 2\left\|\sin\bOmega\right\|^2\lesssim \frac{n\theta_{\text{max}}^2}{\minsigma^{*2}}.
\end{align*}
Since we have assumed $n\theta_{\text{max}}^2 \ll \minsigma^{*2}$, we have $\|\bL-\bR\|=o(1)$
\begin{align*}
    &\sigma_i(\bL)\leq\sigma_i(\bR)+\left\|\bL-\bR\right\|\leq 1+o(1)\leq 2, \\
    &\sigma_i(\bL)\geq\sigma_i(\bR)-\left\|\bL-\bR\right\|\geq 1-o(1) \geq \frac{1}{2}.
\end{align*}
This proves the first part of this lemma. \par
Moving onto the proof of second part, by \eqref{eq:define_vi}, we have $\bv_i^* =\textbf{sgn}\left(\lambda_i^*\right)\bu_i^*$ yielding $\oV^* = \oU^*\bD$. It remains to show that  $\wh\bv_i = \textbf{sgn}(\lambda_i^*)\wh\bu_i$ under event $\mathcal{A}_1$. To this end, set $s := |\{i\in [K]:\lambda_i^*>0\}|$. by Lemma \ref{eigenvaluelemma} we know that
\begin{align*}
    &\lambda_1^*>\lambda_2^*\geq \cdots\geq \lambda_s^*\asymp\beta_n K^{-1}\left\|\btheta\right\|_2^2>0>-\beta_n K^{-1}\left\|\btheta\right\|_2^2\asymp \lambda_{s+1}^*\geq \lambda_{K}^*, \\
    &|\lambda_i^*|\leq \left\|\textbf{diag}(\bTheta\bPi\bP\bPi^\top\bTheta)\right\|\lesssim \theta_{\text{max}}^2 \ll \beta_n K^{-1}\left\|\btheta\right\|_2^2,\quad \forall i>K.
\end{align*}
Recall that $\wh\lambda_1,\wh\lambda_2,\dots,\wh\lambda_K$ are the largest $K$ eigenvalues in magnitude among all $n$ eigenvalues of $\bX$, and $\wh\lambda_1,\wh\lambda_2,\dots,\wh\lambda_K$ are sorted descendingly. By Weyl’s theorem, under event $\mathcal{A}_1$,
\begin{align*}
    &\wh\lambda_1>\wh\lambda_2\geq \cdots\geq \wh\lambda_s\asymp\beta_n K^{-1}\left\|\btheta\right\|_2^2>0>-\beta_n K^{-1}\left\|\btheta\right\|_2^2\asymp \wh\lambda_{s+1}\geq \wh\lambda_{K} ,\\
    &|\wh\lambda_i|\lesssim \sqrt{n}\theta_{\text{max}}+ \left\|\textbf{diag}(\bTheta\bPi\bP\bPi^\top\bTheta)\right\|\lesssim \sqrt{n}\theta_{\text{max}}\ll\beta_n K^{-1}\left\|\btheta\right\|_2^2,\quad \forall i>K.
\end{align*}
As a result, we get $\textbf{sgn}(\lambda_i^*) = \textbf{sgn}(\wh\lambda_i)$. This leads to, using \eqref{eq:define_vi},
\begin{align*}
    \wh\bv_i = \textbf{sgn}(\wh\lambda_i)\wh\bu_i = \textbf{sgn}(\lambda_i^*)\wh\bu_i,
\end{align*}
implying $\oV = \oU\bD$. This completes the proof of the Lemma.
\end{proof}

\subsection{Proof of Lemma \ref{lemma2}}\label{lemma2proof}
\begin{proof}
Recall the defintions of $R$ and $L$ from \eqref{Rdefinition} and \eqref{eq:define_L} respectively. By triangle inequality,
\begin{align*}
    \left\|\bR^\top\oLambda\bR - \oLambda^*\right\|\leq \underbrace{\left\|\bR^\top\oLambda\bR - \bL^\top\oLambda\bL \right\|}_{\alpha_1}+ \underbrace{\left\|\bL^\top\oLambda\bL- \oU^{*\top} \oX\oU^* \right\|}_{\alpha_2}+\underbrace{\left\|\oU^{*\top} \oX\oU^* - \oLambda^*\right\|}_{\alpha_3}.
\end{align*}
We bound the three quantities separately. For the first term $\alpha_1$, we have, on the event $\mathcal{A}_1$, using Lemma \ref{lemma1}
\begin{align*}
    \alpha_1&\leq \left\|(\bR-\bL)^\top\oLambda\bR\right\|+\left\|\bL^\top\oLambda (\bR-\bL)\right\| \\
    &\leq \left\|\bR-\bL\right\|\left\|\oLambda\right\|\left\|\bR\right\|+\left\|\bL\right\|\left\|\oLambda\right\|\left\|\bR-\bL\right\| \\
    &\lesssim \frac{n\theta_{\text{max}}^2}{\minsigma^{*2}}\left\|\oLambda\right\|\leq \frac{n\theta_{\text{max}}^2}{\minsigma^{*2}}\left(\left\|\oLambda^*\right\|+\left\|\oW\right\|\right)\lesssim \frac{n\theta_{\text{max}}^2}{\minsigma^{*2}}\left(\maxsigma^*+\sqrt{n}\theta_{\text{max}}\right)\lesssim \frac{n\theta_{\text{max}}^2\kappa^*}{\minsigma^{*}},
\end{align*}
where the third inequality uses the fact $R$, $L$ are orthogonal matrices and the last inquality uses our assumption that $\sqrt{n} \theta_{\text{max}} \ll \minsigma^{*}$. \par 
For the second quantity $\alpha_2$, one can see that
\begin{align}
    \bL^\top\oLambda\bL- \oU^{*\top} \oX\oU^* = \oU^{*\top}\oU\oLambda\oU^\top\oU^* - \oU^{*\top} \oX\oU^* =  - \oU^{*\top}\oU_\perp \oSigma_\perp \oV_\perp^\top\oU^*,\label{eq1}
\end{align}
where $\oU_\perp,\oSigma_\perp$ and $\oV_\perp$ come from the SVD of $\oX$
\begin{align*}
    \oX\overset{\text{SVD}}{=}\left[\begin{array}{ll}
\oU & \oU_{\perp}
\end{array}\right]\left[\begin{array}{cc}
\oLambda\bD & \boldsymbol{0} \\
\boldsymbol{0} & \oSigma_{\perp}
\end{array}\right]\left[\begin{array}{c}
(\oU\bD)^{\top} \\
\oV_{\perp}^{\top}
\end{array}\right]=\oU \oLambda\oU^\top +\oU_{\perp} \oSigma_{\perp} \oV_{\perp}^{\top}.
\end{align*}
By Weyl's theorem we have
\begin{align}
    \left\| \oSigma_{\perp}\right\|\leq \sigma_{K}(\oH)+\left\|\bW\right\| \leq  \left\|\bW\right\| +\left\|\textbf{diag}(H)\right\|\lesssim \sqrt{n}\theta_{\text{max}}\label{eq2}
\end{align}
under event $\mathcal{A}_1$. On the other hand, by \cite[Lemma 2.5]{chen2021spectral} and Lemma \ref{lemma1} we have 
\begin{align}
    &\left\|\oU^{*\top} \oU_\perp\right\| = \left\|\oU_\perp^\top \oU^*\right\| = \left\|\sin \bOmega\right\|\lesssim \frac{\sqrt{n}\theta_{\text{max}}}{\minsigma^*},  \label{eq3}\\
    &\left\|\oV^\top_\perp\oU^*\right\| = \left\|\oV^\top_\perp\oU^*\bD\right\| = \left\|\oV^\top_\perp\oV^*\right\|= \left\|\sin \bOmega\right\|\lesssim \frac{\sqrt{n}\theta_{\text{max}}}{\minsigma^*}.\label{eq4}
\end{align}
Combine \eqref{eq1}, \eqref{eq2}, \eqref{eq3} and \eqref{eq4} we get
\begin{align*}
    \alpha_2\leq \left\|\oU^{*\top} \oU_\perp\right\|\left\| \oSigma_{\perp}\right\|\left\|\oV^\top_\perp\oU^*\right\|\lesssim \frac{\theta_{\text{max}}^3 n^{1.5}}{\minsigma^{*2}}.
\end{align*}
Next, we analyze $\alpha_3$. By definition we have
\begin{align*}
    \oU^{*\top} \oX\oU^* - \oLambda^* = \oU^{*\top} (\oH+\oW)\oU^* - \oU^{*\top} \oH\oU^* = \oU^{*\top} \oW\oU^*.
\end{align*}
Use the notation of Theorem \ref{firstsubspaceexpansion}, we can write
\begin{align}\label{eq:barw_expand}
    \oW &= \bW -\left[\wh\lambda_1\wh\bu_1\wh\bu_1^\top - \lambda_1^*\bu_1^*\bu_1^{*\top}\right] \nonumber \\
    &= \bW - \bu_1^*\bu_1^{*\top}\bW\bu_1^*\bu_1^{*\top}- \bN\bW\bu_1^*\bu_1^{*\top}-\bu_1^*\bu_1^{*\top}\bW\bN - \bDelta.
\end{align}
Since $\bu_1^*\bu_1^{*\top}\oU^* = \bu_1^*\bu_1^{*\top}[\bu_2^*,\bu_3^*,\dots,\bu_K^*] = \boldsymbol{0}$, we know that
\begin{align*}
    \oU^{*\top} \oW\oU^* = \oU^{*\top} \bW\oU^*-\oU^{*\top} \bDelta\oU^*.
\end{align*}
Now we bound the two terms separately. The second quantity is immediately bounded by Theorem \ref{firstsubspaceexpansion}:
\begin{align*}
    \left\|\oU^{*\top} \bDelta\oU^*\right\|\leq \left\|\bDelta\right\|\lesssim \frac{n\theta_{\text{max}}^2}{\lambda_1^*}.
\end{align*}
On the other hand, similar to \cite[Lemma 2]{yan2021inference}, we use matrix Bernstein inequality \cite[Theorem 6.1.1]{tropp2015introduction} to control $\left\|\oU^{*\top} \bW\oU^*\right\|$. Define $\boldsymbol{Y}_{ii} := \oU^{*\top}_{i,\cdot}\oU^{*}_{i,\cdot}$, $i\in [n]$ and $\boldsymbol{Y}_{ij} := \oU^{*\top}_{i,\cdot}\oU^{*}_{j,\cdot}+\oU^{*\top}_{j,\cdot}\oU^{*}_{i,\cdot}$, $1\leq i< j\leq n$. Then
\begin{align*}
    \oU^{*\top} \bW\oU^* = \sum_{i=1}^n W_{ii}\boldsymbol{Y}_{ii}+\sum_{1\leq i<j\leq n}W_{ij}\boldsymbol{Y}_{ij}.
\end{align*}
 By \eqref{eq:incoherence},
\begin{align*}
    \max_{1\leq i\leq j\leq n}\left\|W_{ij}\boldsymbol{Y}_{ij}\right\|\leq \max_{1\leq i\leq j\leq n}\left\|\boldsymbol{Y}_{ij}\right\| \leq 2\left(\max_{i\in [n]}\left\|\oU^*_{i,\cdot}\right\|_2\right)^2 \leq \frac{2\mu^*(K-1)}{n}.
\end{align*}
Second, since $\oU^{*\top} \bW\oU^*$ is symmetric, we know that 
\begin{align*}
    \left(\oU^{*\top} \bW\oU^*\right)^\top\oU^{*\top} \bW\oU^* = \oU^{*\top} \bW\oU^*\left(\oU^{*\top} \bW\oU^*\right)^\top = \oU^{*\top} \bW\oU^*\oU^{*\top} \bW\oU^*.
\end{align*}
Therefore, we only have to bound $\left\|\mathbb{E}\oU^{*\top} \bW\oU^*\oU^{*\top} \bW\oU^*\right\|$. Since $\|\oU^*\|\leq 1$, we have
\begin{align*}
    \left\|\mathbb{E}\oU^{*\top} \bW\oU^*\oU^{*\top} \bW\oU^*\right\| = \left\|\oU^{*\top}\left(\mathbb{E} \bW\oU^*\oU^{*\top} \bW\right)\oU^*\right\|\leq \left\|\mathbb{E} \bW\oU^*\oU^{*\top} \bW\right\|.
\end{align*}
For any $1\leq i\neq j\leq n$, we have
\begin{align*}
    \left[\mathbb{E} \bW\oU^*\oU^{*\top} \bW\right]_{ij} = \mathbb{E}\left[ \bW\oU^*\oU^{*\top} \bW\right]_{ij} =\mathbb{E}\sum_{k,l=1}^nW_{ik}\left[\oU^*\oU^{*\top}\right]_{kl}W_{lj} = \left[\oU^*\oU^{*\top}\right]_{j, i}\mathbb{E}W_{ij}^2.
\end{align*}
And, for any $1\leq i\leq n$, we have
\begin{align}\label{eq:diag_A}
    \left[\mathbb{E} \bW\oU^*\oU^{*\top} \bW\right]_{ii} &=\mathbb{E}\sum_{k,l=1}^nW_{ik}\left[\oU^*\oU^{*\top}\right]_{kl}W_{li}   = \sum_{j=1}^n\left[\oU^*\oU^{*\top}\right]_{jj}\mathbb{E}W_{ii}^2 \nonumber \\
    &\lesssim  \sum_{j=1}^n\left[\oU^*\oU^{*\top}\right]_{jj}\theta_{\text{max}}^2 = \textbf{tr}\left[\oU^*\oU^{*\top}\right]\theta_{\text{max}}^2 =\left\|\oU^*\right\|_F^2 \theta_{\text{max}}^2 = (K-1)\theta_{\text{max}}^2.
\end{align}
Define $\boldsymbol{A}$ by setting $A_{ij}:=\left(\mathbb{E} \bW\oU^*\oU^{*\top} \bW\right)_{ij} \mathbbm{1}_{i=j}$. Then, by \eqref{eq:diag_A}, 
\begin{align*}
    \left\|\mathbb{E} \bW\oU^*\oU^{*\top} \bW\right\|\leq \left\|\boldsymbol{A}\right\|+\left\|\mathbb{E} \bW\oU^*\oU^{*\top} \bW-\boldsymbol{A}\right\| \leq (K-1)\theta_{\text{max}}^2+\left\|\mathbb{E} \bW\oU^*\oU^{*\top} \bW-\boldsymbol{A}\right\|_F.
\end{align*}
The second summand of the above display can be bounded as:
\begin{align*}
    &\left\|\mathbb{E} \bW\oU^*\oU^{*\top} \bW-\boldsymbol{A}\right\|_F = \sqrt{\sum_{1\leq i\neq j\leq n}\left[\mathbb{E} \bW\oU^*\oU^{*\top} \bW\right]_{ij}^2} = \sqrt{\sum_{1\leq i\neq j\leq n}\left[\oU^*\oU^{*\top}\right]_{ji}^2\left[\mathbb{E}W_{ij}^2\right]^2} \\
    \lesssim &\sqrt{\sum_{i,j=1}^n\left[\oU^*\oU^{*\top}\right]_{ji}^2\theta_{\text{max}}^4} = \left\|\oU^*\oU^{*\top}\right\|_F\theta_{\text{max}}^2\leq \left\|\oU^*\right\|_F^2\theta_{\text{max}}^2 = (K-1)\theta_{\text{max}}^2.
\end{align*}
In conclusion, we get
\begin{align*}
    \left\|\mathbb{E}\oU^{*\top} \bW\oU^*\oU^{*\top} \bW\oU^*\right\|\leq \left\|\mathbb{E} \bW\oU^*\oU^{*\top} \bW\right\|\lesssim  (K-1)\theta_{\text{max}}^2.
\end{align*}
Then by matrix Bernstein inequality \cite[Theorem 6.1.1]{tropp2015introduction}, with probability at least $1-O(n^{-10})$,
\begin{align*}
    \alpha_3 = \left\|\oU^{*\top} \bW\oU^*\right\|\lesssim \sqrt{(K-1)\theta_{\text{max}}^2\log n}+\frac{\mu^*(K-1)}{n}\log n.
\end{align*}
This implies $\alpha_3\lesssim \sqrt{(K-1)\log n}\theta_{\text{max}}$ since we assumed $n^2\theta_{\text{max}}^2\gtrsim  (K-1)\mu^{*2}\log n$. Finally, combining the bounds for $\alpha_1, \alpha_2$ and $\alpha_3$, we get
\begin{align*}
    \left\|\bR^\top\oLambda\bR - \oLambda^*\right\|\leq\alpha_1+\alpha_2+\alpha_3&\lesssim \frac{\kappa^*n\theta_{\text{max}}^2}{\minsigma^{*}} + \frac{\theta_{\text{max}}^3n^{1.5}}{\minsigma^{*2}} +\sqrt{(K-1)\log n}\theta_{\text{max}}  \\
    &\asymp \frac{\kappa^*n\theta_{\text{max}}^2}{\minsigma^{*}} +\sqrt{(K-1)\log n}\theta_{\text{max}},
\end{align*}
since $\maxsigma^*\geq \minsigma^*\gg \sqrt{n}\theta_{\text{max}}$. Also,
\begin{align*}
   \|\bL^\top\oLambda\bL - \oLambda^*\|&\leq \left\|\bL^\top\oLambda\bL- \oU^{*\top} \oX\oU^* \right\|+\left\|\oU^{*\top} \oX\oU^* - \oLambda^*\right\| = \alpha_2+\alpha_3 \\
    &\lesssim\frac{\theta_{\text{max}}^3n^{1.5}}{\minsigma^{*2}} +\sqrt{(K-1)\log n}\theta_{\text{max}},
\end{align*}
completing the proof of the Lemma.
\end{proof}

\subsection{Proof of Lemma \ref{lemma3}}\label{lemma3proof}
\begin{proof}
Since $\oU\oLambda = \oX\oU$, by triangle inequality we have
\begin{align}
    \left\|\oU\oLambda\bL-\oX\oU^*\right\|_{2,\infty} =\left\|\oX\oU\bL-\oX\oU^*\right\|_{2,\infty} \leq \left\|\oH\left(\oU\bL-\oU^*\right)\right\|_{2,\infty}+\left\|\oW\left(\oU\bL-\oU^*\right)\right\|_{2,\infty}.\label{lem3decomposition}
\end{align}
We control the first term on the RHS first. We consider the with self-loop case and the without self-loop case separately.

\noindent\textbf{With self-loop:} The first term can be bounded as
\begin{align*}
    \left\|\oH\left(\oU\bL-\oU^*\right)\right\|_{2,\infty} = \left\|\oU^*\oLambda^*\oU^{*\top}\left(\oU\bL-\oU^*\right)\right\|_{2,\infty}\lesssim  \left\|\oU^*\right\|_{2,\infty}\maxsigma^*\left\|\oU^{*\top}\left(\oU\bL-\oU^*\right)\right\|.
\end{align*}
Let $\bL = \bY_1(\cos \bOmega)\bY_2^\top$ be the SVD of $\bL$, then 
\begin{align*}
    \left\|\oU^{*\top}\left(\oU\bL-\oU^*\right)\right\| = \left\|\bL^\top\bL-\bI\right\| = \left\|\bY\left(\cos^2 \bOmega\right)\bY^\top-\bI\right\|  = \left\|\cos^2 \bOmega-\bI\right\| = \left\|\sin \bOmega\right\|^2.
\end{align*}
By \eqref{eq:sintheta_bound}, we have $\left\|\sin \bOmega\right\|\lesssim\sqrt{n}\theta_{\text{max}}/\minsigma^*$. Therefore,
\begin{align}
    \left\|\oH\left(\oU\bL-\oU^*\right)\right\|_{2,\infty} \lesssim \left\|\oU^*\right\|_{2,\infty}\maxsigma^*\left\|\sin \bOmega\right\|^2\lesssim \frac{\kappa^*}{\minsigma^{*}}\sqrt{(K-1)\mu^*n}\theta_{\text{max}}^2.\label{lem3oH1}
\end{align}

\noindent \textbf{Without self-loop:} The first term can be bounded as
\begin{align*}
     \left\|\oH\left(\oU\bL-\oU^*\right)\right\|_{2,\infty}\leq \left\|\oU^*\oLambda^*\oU^{*\top}\left(\oU\bL-\oU^*\right)\right\|_{2,\infty} +\left\|\left(\oH-\oU^*\oLambda^*\oU^{*\top}\right)\left(\oU\bL-\oU^*\right)\right\|_{2,\infty}.
\end{align*}
The first summand is bounded as the last case. For the second term,
\begin{align*}
    \left\|\oH-\oU^*\oLambda^*\oU^{*\top}\right\|\leq \sigma_{K+1}(\bH)\leq \left\|\textbf{diag}(\bTheta\bPi\bP\bPi^\top\bTheta)\right\|\lesssim \theta_{\text{max}}^2, 
\end{align*}
one can see that 
\begin{align}
    &\left\|\left(\oH-\oU^*\oLambda^*\oU^{*\top}\right)\left(\oU\bL-\oU^*\right)\right\|_{2,\infty}\leq \left\|\oH-\oU^*\oLambda^*\oU^{*\top}\right\|\left\|\oU\bL-\oU^*\right\|  \lesssim \theta_{\text{max}}^2\left\|\oU\bL-\oU^*\right\| \nonumber \\
    \leq&  \theta_{\text{max}}^2\left\|\oU\oU^\top\oU^*-\oU^*\right\| = \theta_{\text{max}}^2\left\|\left(\oU\oU^\top-\oU^*\oU^{*\top}\right)\oU^*\right\| \nonumber \\
    \leq&\theta_{\text{max}}^2 \left\|\oU\oU^\top-\oU^*\oU^{*\top}\right\| = \theta_{\text{max}}^2\left\|\sin\bOmega\right\|\lesssim\frac{\sqrt{n}\theta_{\text{max}}^3}{\minsigma^*}. \label{lem3woselfloopeq2}
\end{align}
Since $\theta_{\text{max}}\lesssim 1$, combine \eqref{lem3oH1} and \eqref{lem3woselfloopeq2} we get 
\begin{align}
    \left\|\oH\left(\oU\bL-\oU^*\right)\right\|_{2,\infty} \lesssim  \frac{\kappa^*}{\minsigma^{*}}\sqrt{(K-1)\mu^*n}\theta_{\text{max}}^2.\label{lem3oH2}
\end{align}

It remains to bound the second term $\left\|\oW\left(\oU\bL-\oU^*\right)\right\|_{2,\infty}$. Recall from \eqref{eq:barw_expand},
\begin{align}
    \oW = \bW - \bu_1^*\bu_1^{*\top}\bW\bu_1^*\bu_1^{*\top}- \bN\bW\bu_1^*\bu_1^{*\top}-\left(\bW\bu_1^*\bu_1^{*\top}\right)^\top\bN - \bDelta.\label{lem3oWexpansion}
\end{align}
We control the five summands of RHS separately.
\begin{enumerate}
\item[(a)] \textbf{Control} $\left\|\bW\left(\oU\bL-\oU^*\right)\right\|_{2,\infty}$: We use the definition of leave-one-out matrix $\bW^{(i)}$ in the proof of Theorem \ref{deltainfty}. Define $\bX^{(i)} := \bH + \bW^{(i)}$ and let $\wh\lambda_1^{(i)}, \wh\lambda_2^{(i)},\dots, \wh\lambda_K^{(i)}$ be the largest $K$ eigenvalues of $\bX^{(i)}$ in magnitude, and they are sorted decreasingly. Let $\wh\bu_1^{(i)}, \wh\bu_2^{(i)}, \dots, \wh\bu_K^{(i)}$ the corresponding eigenvectors. Let $\oX^{(i)} = \bX^{(i)}-\wh\lambda_1^{(i)}\wh\bu_1^{(i)}\wh\bu_1^{(i)\top}$, $\oU^{(i)} = [\wh\bu_2^{(i)}, \dots, \wh\bu_K^{(i)}]$ and $\bL^{(i)} = \oU^{(i)\top}\oU^*$. Then $\oU^{(i)}$ and $\bL^{(i)}$ are independent with $\bW-\bW^{(i)}$. And, one can easily see that the results in Lemma \ref{Wspectral}, Theorem \ref{firstsubspaceexpansion} (we also define $\bDelta^{(i)}$), Corollary \ref{cor1} and Lemma \ref{lemma1} also apply to the leave-one-out matrices. As a result, we have
\begin{align*}
    \left\|\bW\left(\oU\bL-\oU^*\right)\right\|_{2,\infty} &= \max_{1\leq i\leq n}\left\|\bW_{i, \cdot}\left(\oU\bL-\oU^*\right)\right\|_{2} \\
    &\leq \max_{1\leq i\leq n}\left\{\left\|\bW_{i, \cdot}\left(\oU^{(i)}\bL^{(i)}-\oU^*\right)\right\|_{2}+\left\|\bW_{i, \cdot}\left(\oU\bL-\oU^{(i)}\bL^{(i)}\right)\right\|_{2}\right\}
\end{align*}
By Lemma \ref{Wconcentration2}, we have
\begin{align*}
    \left\|\bW_{i, \cdot}\left(\oU^{(i)}\bL^{(i)}-\oU^*\right)\right\|_{2}\lesssim & \sqrt{\log n}\theta_{\text{max}} \left\|\oU^{(i)}\bL^{(i)}-\oU^*\right\|_F +\log n \left\|\oU^{(i)}\bL^{(i)}-\oU^*\right\|_{2,\infty} \\
    \leq& \sqrt{\log n}\theta_{\text{max}} \left\|\oU\bL-\oU^*\right\|_F +\log n \left\|\oU\bL-\oU^*\right\|_{2,\infty} \\ 
    &+ \sqrt{\log n}\theta_{\text{max}}\left\|\oU^{(i)}\bL^{(i)}-\oU\bL\right\|_F+ \log n \left\|\oU^{(i)}\bL^{(i)}-\oU\bL\right\|_{2,\infty}  \\
    \leq& \sqrt{\log n}\theta_{\text{max}} \left\|\oU\bL-\oU^*\right\|_F +\log n \left\|\oU\bL-\oU^*\right\|_{2,\infty} \\ 
    &+ \log n \left\|\oU^{(i)}\bL^{(i)}-\oU\bL\right\|_F
\end{align*}
with probability at least $1-O(n^{-14})$. By \cite[Lemma 2.5]{chen2021spectral}, since $\oU^{*\top}\oU^* = \bI$, we have
\begin{align}
    \left\|\oU\bL-\oU^*\right\|_F &= \left\|\oU\oU^\top\oU^*-\oU^*\right\|_F = \left\|\left(\oU\oU^\top-\oU^*\oU^{*\top}\right)\oU^*\right\|_F \nonumber \\
    &\leq\left\|\oU\oU^\top-\oU^*\oU^{*\top}\right\|_F = \sqrt{2}\left\|\sin \bOmega\right\|_F\leq \sqrt{2(K-1)}\left\|\sin \bOmega\right\| \nonumber \\
    &\lesssim \frac{\sqrt{(K-1)n}}{\minsigma^*}\theta_{\text{max}}, \label{lem3eq18}
\end{align}
where we used \eqref{eq:sintheta_bound} for the last equality. As a result, we have
\begin{align}
    &\left\|\bW_{i, \cdot}\left(\oU^{(i)}\bL^{(i)}-\oU^*\right)\right\|_{2}   \nonumber \\
    \lesssim & \underbrace{\frac{\sqrt{(K-1)n\log n}}{\minsigma^*}\theta_{\text{max}}+\log n \left\|\oU\bL-\oU^*\right\|_{2,\infty}+\log n\left\|\oU^{(i)}\bL^{(i)}-\oU\bL\right\|_F}_{\alpha}.\label{lem3eq7}
\end{align}
On the other hand, we have 
\begin{align}
    \left\|\bW_{i, \cdot}\left(\oU\bL-\oU^{(i)}\bL^{(i)}\right)\right\|_{2}\leq \left\|\bW_{i,\cdot}\right\|_2\left\|\oU\bL-\oU^{(i)}\bL^{(i)}\right\|\leq \sqrt{n}\theta_{\text{max}}\left\|\oU\bL-\oU^{(i)}\bL^{(i)}\right\|_F.\label{lem3eq17}
\end{align}
As a result, it remains to bound $\beta:=\left\|\oU^{(i)}\bL^{(i)}-\oU\bL\right\|_F$. One can see that
\begin{align*}
    \beta = \left\|\left(\oU^{(i)}\oU^{(i)\top}-\oU\oU^{\top}\right)\oU^*\right\|_F\leq\left\|\oU^{(i)}\oU^{(i)\top}-\oU\oU^{\top}\right\|_F .
\end{align*}
By Davis-Kahan's sin$\Theta$ theorem \cite[Theorem 2.7]{chen2021spectral}, we have
\begin{align}
    \beta = \left\|\oU^{(i)}\bL^{(i)}-\oU\bL\right\|_F &\lesssim \frac{\left\|\left(\oX-\oX^{(i)}\right)\oU^{(i)}\right\|_F}{\min_{2\leq j\leq K}\left|\wh\lambda_j\right|-\max_{j>K}\left|\wh\lambda_j^{(i)}\right|} \nonumber \\
    &\lesssim\frac{\left\|\left(\oX-\oX^{(i)}\right)\oU^{(i)}\right\|_F}{\minsigma^*-\left\|\bW\right\|-\left\|\bW^{(i)}\right\|} \lesssim\frac{\left\|\left(\oX-\oX^{(i)}\right)\oU^{(i)}\right\|_F}{\minsigma^*},\label{lem3eq6}
\end{align}
since $\|\bW^{(i)}\|\leq \|\bW\|\lesssim \sqrt{n}\theta_{\text{max}}\ll \minsigma^*$. By Theorem \ref{firstsubspaceexpansion}, $\oX-\oX^{(i)}$ can be decomposed as
\begin{align*}
    \oX-\oX^{(i)} =& \bW-\bW^{(i)}-\bu_1^*\bu_1^{*\top}\Big(\bW-\bW^{(i)}\Big)\bu_1^*\bu_1^{*\top}- \bN\Big(\bW-\bW^{(i)}\Big)\bu_1^*\bu_1^{*\top}\\
    &-\left(\Big(\bW-\bW^{(i)}\Big)\bu_1^*\bu_1^{*\top}\right)^\top\bN - \left(\bDelta-\bDelta^{(i)}\right).
\end{align*}
We bound the numerator of \eqref{lem3eq6} via controlling the five summands of RHS of the above display separately.

\begin{enumerate}
    \item 
\textbf{Control} $\left\|\Big(\bW-\bW^{(i)}\Big)\oU^{(i)}\right\|_F$: By triangle inequality we have
\begin{align}
    \left\|\Big(\bW-\bW^{(i)}\Big)\oU^{(i)}\right\|_F&= \left\|\Big(\bW-\bW^{(i)}\Big)\oU^{(i)}\bL^{(i)}\left(\bL^{(i)}\right)^{-1}\right\|_F  \nonumber \\
    &\lesssim \left\|\Big(\bW-\bW^{(i)}\Big)\oU^{(i)}\bL^{(i)}\right\|_F  \nonumber \\
    &\leq \left\|\Big(\bW-\bW^{(i)}\Big)\oU^*\right\|_F+\underbrace{\left\|\Big(\bW-\bW^{(i)}\Big)\left(\oU^{(i)}\bL^{(i)}-\oU^*\right)\right\|_F}_{=:\vartheta_1} \label{lem3eq15}
\end{align}
On one hand, by Lemma \ref{Wconcentration4} and  \eqref{eq:incoherence} we have
\begin{align}
    &\left\|\Big(\bW-\bW^{(i)}\Big)\oU^*\right\|_F\lesssim \sqrt{\log n}\theta_{\text{max}}\left\|\oU^*\right\|_F+(\sqrt{n}\theta_{\text{max}}+\log n)\left\|\oU^*\right\|_{2,\infty}   \nonumber \\
    &\lesssim \sqrt{(K-1)\log n}\theta_{\text{max}}+(\sqrt{n}\theta_{\text{max}}+\log n)\sqrt{(K-1)\mu^*/n}=:\alpha_1 \label{lem3eq16}
\end{align}
with probability at least $1-O(n^{-14})$. On the other hand, we have
\begin{align*}
    \vartheta^2_1 =& \left\|\bW_{i,\cdot}\left(\oU^{(i)}\bL^{(i)}-\oU^*\right)\right\|_2^2+\sum_{j\in [n],j\neq i}W_{ji}^2\left\|\left(\oU^{(i)}\bL^{(i)}-\oU^*\right)_{j,\cdot}\right\|_2^2.
\end{align*}
Since $\sum_{j\in [n],j\neq i}W_{ji}^2\leq \left\|\bW_{\cdot, i}\right\|_2^2\leq \left\|\bW\right\|^2\lesssim n\theta_{\text{max}}^2$,
we get 
\begin{align*}
    \vartheta^2_1 \lesssim \alpha^2+n\theta_{\text{max}}^2\left\|\oU^{(i)}\bL^{(i)}-\oU^*\right\|^2_{2,\infty},
\end{align*}
where $\alpha$ is defined by \eqref{lem3eq7}. As a result, we have
\begin{align*}
     \vartheta^2_1&\lesssim \alpha+\sqrt{n}\theta_{\text{max}}\left\|\oU^{(i)}\bL^{(i)}-\oU^*\right\|_{2,\infty}\nonumber \\
     &\lesssim \alpha+ \sqrt{n}\theta_{\text{max}} \left\|\oU^{(i)}\bL^{(i)}-\oU\bL\right\|_{2,\infty}+\sqrt{n}\theta_{\text{max}}\left\|\oU\bL-\oU^*\right\|_{2,\infty} \nonumber \\
     &\lesssim \alpha + \sqrt{n}\theta_{\text{max}}\beta +\sqrt{n}\theta_{\text{max}}\left\|\oU\bL-\oU^*\right\|_{2,\infty},
\end{align*}
where $\beta$ is defined by \eqref{lem3eq6}. Combine this with \eqref{lem3eq15} and \eqref{lem3eq16} we have
\begin{align}
    \left\|\Big(\bW-\bW^{(i)}\Big)\oU^{(i)}\right\|_F\lesssim &\; \alpha + \sqrt{n}\theta_{\text{max}}\beta +\sqrt{n}\theta_{\text{max}}\left\|\oU\bL-\oU^*\right\|_{2,\infty} +\alpha_1.\label{lem3eq1}
\end{align}
with probability at least $1-O(n^{-14})$. 

\item \textbf{Control} $\left\|\bu_1^*\bu_1^{*\top}\Big(\bW-\bW^{(i)}\Big)\bu_1^*\bu_1^{*\top}\oU^{(i)}\right\|_F$: Since $\bu_1^*\bu_1^{*\top}\Big(\bW-\bW^{(i)}\Big)\bu_1^*\bu_1^{*\top}\oU^{(i)}$ is a rank-$1$ matrix, we know that
\begin{align*}
    &\left\|\bu_1^*\bu_1^{*\top}\Big(\bW-\bW^{(i)}\Big)\bu_1^*\bu_1^{*\top}\oU^{(i)}\right\|_F = \left\|\bu_1^*\right\|_2\left\|\bu_1^{*\top}\Big(\bW-\bW^{(i)}\Big)\bu_1^*\bu_1^{*\top}\oU^{(i)}\right\|_2 \\
    &\leq \left\|\Big(\bW-\bW^{(i)}\Big)\bu_1^*\bu_1^{*\top}\right\|\leq \left\|\Big(\bW-\bW^{(i)}\Big)\bu_1^*\bu_1^{*\top}\right\|_F.
\end{align*}
By Lemma \ref{Wconcentration4} we know that
\begin{align}\label{eq:wi_quad_frob}
    \left\|\Big(\bW-\bW^{(i)}\Big)\bu_1^*\bu_1^{*\top}\right\|_F&\lesssim \sqrt{\log n}\theta_{\text{max}}\left\|\bu_1^*\bu_1^{*\top}\right\|_F+(\sqrt{n}\theta_{\text{max}}+\log n)\left\|\bu_1^*\bu_1^{*\top}\right\|_{2,\infty}\nonumber \\
    & \leq \sqrt{\log n}\theta_{\text{max}}+(\sqrt{n}\theta_{\text{max}}+\log n)\sqrt{\mu^*/n} =: \rho_2.
\end{align}
with probability at least $1-O(n^{-14})$, where the last inequality uses \eqref{eq:incoherence}. Hence,
\begin{align}
    \left\|\bu_1^*\bu_1^{*\top}\Big(\bW-\bW^{(i)}\Big)\bu_1^*\bu_1^{*\top}\oU^{(i)}\right\|_F\lesssim \rho_2\label{lem3eq2}
\end{align}
with probability at least $1-O(n^{-14})$. 

\item \textbf{Control} $\left\|\bN\Big(\bW-\bW^{(i)}\Big)\bu_1^*\bu_1^{*\top}\oU^{(i)}\right\|_F$: By \eqref{eq:N_spectral_bound}, we have $\left\|\bN\right\|\lesssim 1$ implying 
\begin{align*}
    \left\|\bN\Big(\bW-\bW^{(i)}\Big)\bu_1^*\bu_1^{*\top}\oU^{(i)}\right\|_F&\leq \left\|\bN\right\|\left\|\oU^{(i)}\right\|\left\|\Big(\bW-\bW^{(i)}\Big)\bu_1^*\bu_1^{*\top}\right\|_F \\
    &\lesssim \left\|\Big(\bW-\bW^{(i)}\Big)\bu_1^*\bu_1^{*\top}\right\|_F \lesssim \rho_2.
\end{align*}
 with probability at least $1-O(n^{-14})$ using \eqref{eq:wi_quad_frob}.

\item \textbf{Control} $\left\|\left(\Big(\bW-\bW^{(i)}\Big)\bu_1^*\bu_1^{*\top}\right)^\top\bN\oU^{(i)}\right\|_F$: Since $\left(\Big(\bW-\bW^{(i)}\Big)\bu_1^*\bu_1^{*\top}\right)^\top\bN\oU^{(i)}$ is a rank-$1$ matrix and $\bN = \lambda_1^*\bN_1$, by Lemma \ref{lemmaN1N2} and \eqref{Wconcentration4} we know that
\begin{align}
    \left\|\left(\Big(\bW-\bW^{(i)}\Big)\bu_1^*\bu_1^{*\top}\right)^\top\bN\oU^{(i)}\right\|_F &= \left\|\bu_1^*\right\|_2\left\|\bu_1^{*\top}\Big(\bW-\bW^{(i)}\Big)\bN\oU^{(i)}\right\|_2 \nonumber\\
    &\lesssim \left\|\Big(\bW-\bW^{(i)}\Big)\bu_1^{*}\right\|_2 \lesssim \rho_2 \label{lem3eq4}
\end{align}
with probability at least $1-O(n^{-15})$. 

\item \textbf{Control} $\left\|\left(\bDelta-\bDelta^{(i)}\right)\oU^{(i)}\right\|_F$: By Theorem \ref{firstsubspaceexpansion}, the ranks of $\bDelta$ and $\bDelta^{(i)}$ are at most $3$. Hence,
\begin{align}
    \left\|\left(\bDelta-\bDelta^{(i)}\right)\oU^{(i)}\right\|_F\lesssim \left\|\bDelta-\bDelta^{(i)}\right\|_F \left\|\oU^{(i)}\right\|\lesssim \left\|\bDelta-\bDelta^{(i)}\right\|\lesssim \frac{n\theta_{\text{max}}^2}{\lambda_1^*}.\label{lem3eq5}
\end{align}
\end{enumerate}
Recall the definitions of $\alpha, \beta$ from \eqref{lem3eq7} and  \eqref{lem3eq6} respectively. Combining~\eqref{lem3eq1}-\eqref{lem3eq5}, we get
\begin{align*}
    \left\|\left(\oX-\oX^{(i)}\right)\oU^{(i)}\right\|_F\lesssim&\; \alpha + \sqrt{n}\theta_{\text{max}}\beta +\sqrt{n}\theta_{\text{max}}\left\|\oU\bL-\oU^*\right\|_{2,\infty} + \frac{n\theta_{\text{max}}^2}{\lambda_1^*} +\alpha_1
\end{align*}
with probability at least $1-O(n^{-14})$. This, along with~\eqref{lem3eq6} yields
\begin{align*}
    \beta\lesssim & \frac{\sqrt{n}\theta_{\text{max}}}{\minsigma^*}\beta+\frac{1}{\minsigma^*}\alpha+\frac{\sqrt{n}\theta_{\text{max}}}{\minsigma^*}\left\|\oU\bL-\oU^*\right\|_{2,\infty} + \frac{n\theta_{\text{max}}^2}{\lambda_1^*\minsigma^*} \\
    &+ \frac{\sqrt{(K-1)\log n}\theta_{\text{max}}+(\sqrt{n}\theta_{\text{max}}+\log n)\sqrt{(K-1)\mu^*/n}}{\minsigma^*}
\end{align*}
with probability at least $1-O(n^{-14})$. Since $\sqrt{n}\theta_{\text{max}}\ll\minsigma^*$,
\begin{align*}
    \beta\lesssim & \frac{1}{\minsigma^*}\alpha+\frac{\sqrt{n}\theta_{\text{max}}}{\minsigma^*}\left\|\oU\bL-\oU^*\right\|_{2,\infty} + \frac{n\theta_{\text{max}}^2}{\lambda_1^*\minsigma^*} + \frac{\alpha_1}{\minsigma^*}
\end{align*}
with probability at least $1-O(n^{-14})$. Recall that $\alpha$ is defined as
\begin{align*}
    \alpha = \frac{\sqrt{(K-1)n\log n}}{\minsigma^*}\theta_{\text{max}}+\log n \left\|\oU\bL-\oU^*\right\|_{2,\infty}+\beta  \log n 
\end{align*}
in ~\eqref{lem3eq7}, we have
\begin{align*}
    \alpha \lesssim & \frac{\sqrt{(K-1)n\log n}}{\minsigma^*}\theta_{\text{max}} + \log n \left\|\oU\bL-\oU^*\right\|_{2,\infty} + \frac{\log n}{\minsigma^*}\alpha \\
    &+ \frac{\sqrt{n}\theta_{\text{max}}\log n}{\minsigma^*}\left\|\oU\bL-\oU^*\right\|_{2,\infty} + \frac{n\theta_{\text{max}}^2\log n}{\lambda_1^*\minsigma^*} + \frac{\alpha_1 \log n }{\minsigma^*}
\end{align*}
with probability at least $1-O(n^{-14})$. By our assumption $\max \{\sqrt{n}\theta_{\text{max}}, \log n\}\ll\minsigma^*$, yielding
\begin{align}
    \alpha \lesssim & \log n \left\|\oU\bL-\oU^*\right\|_{2,\infty} +\frac{\sqrt{(K-1)n\log n}}{\minsigma^*}\theta_{\text{max}} + \frac{n\theta_{\text{max}}^2\log n}{\lambda_1^*\minsigma^*}  \nonumber \\
    &+  \frac{(\sqrt{n}\theta_{\text{max}}+\log n)\sqrt{(K-1)\mu^*/n}\log n}{\minsigma^*}\label{lem3eq8}
\end{align}
with probability at least $1-O(n^{-14})$. This also implies
\begin{align}
    \beta\lesssim & \frac{\sqrt{n}\theta_{\text{max}}+\log n}{\minsigma^*}\left\|\oU\bL-\oU^*\right\|_{2,\infty} + \frac{n\theta_{\text{max}}^2}{\lambda_1^*\minsigma^*} \nonumber \\
    &+ \frac{\sqrt{(K-1)\log n}\theta_{\text{max}}+(\sqrt{n}\theta_{\text{max}}+\log n)\sqrt{(K-1)\mu^*/n}}{\minsigma^*}\label{lem3eq10}
\end{align}
with probability at least $1-O(n^{-14})$. Plugging the quantities $\alpha, \beta$ in ~\eqref{lem3eq7} we get
\begin{align*}
\left\|\bW_{i, \cdot}\left(\oU\bL-\oU^*\right)\right\|_{2}\leq&\left\|\bW_{i, \cdot}\left(\oU^{(i)}\bL^{(i)}-\oU^*\right)\right\|_{2}+\left\|\bW_{i, \cdot}\left(\oU\bL-\oU^{(i)}\bL^{(i)}\right)\right\|_{2}  \\
\lesssim& \alpha + \sqrt{n}\theta_{\text{max}}\beta
\end{align*}
with probability at least $1-O(n^{-14})$. Taking a union bound for $i\in [n]$ and noting that $\mathbb{P}(\mathcal{A}_1) \ge 1-O(n^{-10})$, we have with probability at least $1-O(n^{-10})$,
\begin{align}
    \left\|\bW\left(\oU\bL-\oU^*\right)\right\|_{2,\infty}\lesssim& \left(\frac{n \theta_{\text{max}}^2}{\minsigma^*}+\log n\right)\left\|\oU\bL-\oU^*\right\|_{2,\infty} +\frac{\sqrt{(K-1)n\log n}}{\minsigma^*}\theta_{\text{max}} \nonumber \\
&+\frac{(\sqrt{n}\theta_{\text{max}}+\log n) n\theta_{\text{max}}^2}{\lambda_1^*\minsigma^*} + \frac{(\sqrt{n}\theta_{\text{max}}+\log n)^2\sqrt{(K-1)\mu^*/n}}{\minsigma^*},\label{lem3eq9}
\end{align}
where we used the bounds ~\eqref{lem3eq8} and ~\eqref{lem3eq10}.

\item[(b)] \textbf{Control} $\left\|\bu_1^*\bu_1^{*\top}\bW\bu_1^*\bu_1^{*\top}\left(\oU\bL-\oU^*\right)\right\|_{2,\infty}$: We can write it as
\begin{align}\label{eq:lem3term4}
    \left\|\bu_1^*\bu_1^{*\top}\bW\bu_1^*\bu_1^{*\top}\left(\oU\bL-\oU^*\right)\right\|_{2,\infty} = \left\|\bu_1^*\right\|_\infty\left|\bu_1^{*\top}\bW\bu_1^*\right|\left\|\bu_1^{*\top}\left(\oU\bL-\oU^*\right)\right\|_{2}
\end{align}
By Lemma \ref{Wconcentration5}, 
$\left|\bu_1^{*\top}\bW\bu_1^*\right|\lesssim \sqrt{\log n}$
with probability at least $1-O(n^{-15})$. Since $\oU^{*\top}\oU = \bI$,
\begin{align*}
    \left\|\bu_1^{*\top}\left(\oU\bL-\oU^*\right)\right\|_{2}&\leq \left\|\bu_1^{*}\right\|_2\left\|\oU\bL-\oU^*\right\| = \left\|\oU\oU^\top\oU^*-\oU^*\right\| = \left\|\left(\oU\oU^\top-\oU^*\oU^{*\top}\right)\oU^*\right\|  \\
    &\leq \left\|\oU\oU^\top-\oU^*\oU^{*\top}\right\| = \left\|\sin\bOmega\right\|\lesssim\frac{\sqrt{n}\theta_{\text{max}}}{\minsigma^*},
\end{align*}
where the last inequality uses \eqref{eq:sintheta_bound}. Plugging in \eqref{eq:lem3term4},
\begin{align}
    \left\|\bu_1^*\bu_1^{*\top}\bW\bu_1^*\bu_1^{*\top}\left(\oU\bL-\oU^*\right)\right\|_{2,\infty}\lesssim \sqrt{\frac{\mu^*}{n}}\sqrt{\log n}\frac{\sqrt{n}\theta_{\text{max}}}{\minsigma^*} = \frac{\sqrt{\mu^*\log n}\theta_{\text{max}}}{\minsigma^*}\label{lem3eq11}
\end{align}
with probability at least $1-O(n^{-15})$.

\item[(c)] \textbf{Control} $\left\|\bN\bW\bu_1^*\bu_1^{*\top}\left(\oU\bL-\oU^*\right)\right\|_{2,\infty}$: Since $\bN = \lambda_1^*\bN_1$, we apply Lemma \ref{N1N2} to get
\begin{align*}
    &\left\|\bN\bW\bu_1^*\bu_1^{*\top}\left(\oU\bL-\oU^*\right)\right\|_{2,\infty}\\
    \lesssim &\left\|\bW\bu_1^*\bu_1^{*\top}\left(\oU\bL-\oU^*\right)\right\|_{2,\infty}  +\left\|\left(\bN-\bI\right)\bW\bu_1^*\bu_1^{*\top}\left(\oU\bL-\oU^*\right)\right\|_{2,\infty}\\
    \lesssim &\left\|\bW\bu_1^*\bu_1^{*\top}\left(\oU\bL-\oU^*\right)\right\|_{2,\infty}  +\left\|\bN-\bI\right\|_{2,\infty}\left\|\bW\bu_1^*\bu_1^{*\top}\left(\oU\bL-\oU^*\right)\right\|\\
    \lesssim & \left\|\bW\bu_1^*\right\|_{2,\infty}\left\|\bu_1^{*\top}\left(\oU\bL-\oU^*\right)\right\|+\sqrt{\frac{(K-1)\mu^*}{n}}\left\|\bW\right\|\left\|\oU\bL-\oU^*\right\|.
\end{align*}
On one hand, under event $\mathcal{A}_1$, we have $\|\bW\|\lesssim \sqrt{n}\theta_{\text{max}}$ and $\left\|\oU\bL-\oU^*\right\| = \left\|\sin\bOmega\right\|\lesssim \sqrt{n}\theta_{\text{max}}/\minsigma^*$, using \eqref{eq:sintheta_bound}, yielding
\begin{align*}
    \sqrt{\frac{(K-1)\mu^*}{n}}\left\|\bW\right\|\left\|\oU\bL-\oU^*\right\|
    \lesssim \frac{\sqrt{(K-1)n\mu^*}\theta_{\text{max}}^2}{\minsigma^*}.
\end{align*}
On the other hand, by Lemma \ref{Wconcentration3}, $\left\|\bW\bu_1^*\right\|_{2,\infty}\lesssim \sqrt{\log n}\theta_{\text{max}}+\log n\sqrt{\mu^*/n} $ with probability at least $1-O(n^{-14})$. As a result, one can see that 
\begin{align*}
    \left\|\bW\bu_1^*\right\|_{2,\infty}\left\|\bu_1^{*\top}\left(\oU\bL-\oU^*\right)\right\|&\lesssim \left(\sqrt{\log n}\theta_{\text{max}}+\log n\sqrt{\mu^*/n}\right)\left\|\oU\bL-\oU^*\right\|\\
    &\lesssim  \frac{\sqrt{n\log n}\theta_{\text{max}}^2+\log n\sqrt{\mu^*}\theta_{\text{max}}}{\minsigma^*}
\end{align*}
with probability at least $1-O(n^{-14})$. To sum up, we get
\begin{align}
    \left\|\bN\bW\bu_1^*\bu_1^{*\top}\left(\oU\bL-\oU^*\right)\right\|_{2,\infty}\lesssim \frac{(\sqrt{(K-1)\mu^*}+\sqrt{\log n})\sqrt{n}\theta_{\text{max}}^2+\log n\sqrt{\mu^*}\theta_{\text{max}}}{\minsigma^*}\label{lem3eq12}
\end{align}
with probability at least $1-O(n^{-14})$.

\item[(d)] \textbf{Control} $\left\|\left(\bW\bu_1^*\bu_1^{*\top}\right)^\top\bN\left(\oU\bL-\oU^*\right)\right\|_{2,\infty}$: It can be bounded as
\begin{align*}
\left\|\left(\bW\bu_1^*\bu_1^{*\top}\right)^\top\bN\left(\oU\bL-\oU^*\right)\right\|_{2,\infty} &= \left\|\bu_1^*\right\|_{\infty}\left\|\bu_1^{*\top}\bW\bN\left(\oU\bL-\oU^*\right)\right\|_{2}  \\
&\leq \left\|\bu_1^*\right\|_{\infty}\left\|\bu_1^{*}\right\|_2\left\|\bW\right\|\left\|\bN\right\|\left\|\oU\bL-\oU^*\right\|.
\end{align*}
Recall that $\bN = \lambda_1^*\bN_1$, as we have shown in Lemma \ref{lemmaN1N2} and \eqref{eq:sintheta_bound}, $\left\|\bN\right\|\lesssim 1$ and $\left\|\oU\bL-\oU^*\right\| = \left\|\sin\bOmega\right\|\lesssim \sqrt{n}\theta_{\text{max}}/\minsigma^*$. Therefore, we have
\begin{align}
\left\|\left(\bW\bu_1^*\bu_1^{*\top}\right)^\top\bN\left(\oU\bL-\oU^*\right)\right\|_{2,\infty}\lesssim \sqrt{\frac{\mu^*}{n}}\sqrt{n}\theta_{\text{max}}\frac{\sqrt{n}\theta_{\text{max}}}{\minsigma^*} = \frac{\sqrt{\mu^*n}\theta_{\text{max}}^2}{\minsigma^*}.\label{lem3eq13}
\end{align}

\item[(e)] \textbf{Control} $\left\|\bDelta\left(\oU\bL-\oU^*\right)\right\|_{2,\infty}$: According to Lemma \ref{firstsubspaceexpansion}, we simply have
\begin{align*}
    \left\|\bDelta\left(\oU\bL-\oU^*\right)\right\|_{2,\infty}\leq \left\|\bDelta\left(\oU\bL-\oU^*\right)\right\|_F \leq \left\|\bDelta\right\|\left\|\oU\bL-\oU^*\right\|_F\lesssim \frac{n\theta_{\text{max}}^2}{\lambda_1^*}\left\|\oU\bL-\oU^*\right\|_F.
\end{align*}
As we have shown in \eqref{lem3eq18}, we have $\left\|\oU\bL-\oU^*\right\|_F \lesssim \sqrt{(K-1)n}\theta_{\text{max}}/\minsigma^*$. As a result, we have
\begin{align}
    \left\|\bDelta\left(\oU\bL-\oU^*\right)\right\|_{2,\infty}\lesssim \frac{n}{\lambda_1^*}\left\|\oU\bL-\oU^*\right\|_F\lesssim \frac{\sqrt{K-1}\theta_{\text{max}}^3n^{1.5}}{\lambda_1^*\minsigma^*}.\label{lem3eq14}
\end{align}
\end{enumerate}
Finally, we can sum up all the five terms we bounded above. Specifically, a combination of \eqref{lem3eq9}, \eqref{lem3eq11}, \eqref{lem3eq12}, \eqref{lem3eq13}, \eqref{lem3eq14} and \eqref{lem3oWexpansion} yields
\begin{align*}
    \left\|\oW\left(\oU\bL-\oU^*\right)\right\|_{2,\infty}\lesssim& \left(\frac{n \theta_{\text{max}}^2}{\minsigma^*}+\log n\right)\left\|\oU\bL-\oU^*\right\|_{2,\infty} +\frac{\sqrt{(K-1)n\log n}}{\minsigma^*}\theta_{\text{max}} \nonumber \\
&+\frac{(\sqrt{(K-1)n}\theta_{\text{max}}+\log n) n\theta_{\text{max}}^2}{\lambda_1^*\minsigma^*} + \frac{(\sqrt{n}\theta_{\text{max}}+\log n)^2\sqrt{(K-1)\mu^*/n}}{\minsigma^*}
\end{align*}
with probability at least $1-O(n^{-10})$. Plugging in this bound, along with the bound \eqref{lem3oH1} and \eqref{lem3oH2} in \eqref{lem3decomposition} provides the desired conclusion.

\end{proof}

\subsection{Proof of Lemma \ref{lemma4}}\label{lemma4proof}
\begin{proof}
Since $\oX\oU^* = \oU^*\oLambda^*+\oW\oU^*$, we consider the following decomposition
\begin{align}
    \minsigma^*\left\|\oU\bL-\oU^*\right\|_{2,\infty}&\leq \left\|\oU\bL\oLambda^*-\oU^*\oLambda^*\right\|_{2,\infty} \nonumber \\
    &\leq \left\|\oU\oLambda\bL-\oU^*\oLambda^*\right\|_{2,\infty}+\left\|\oU\bL\oLambda^*-\oU\oLambda\bL\right\|_{2,\infty} \nonumber \\
    &\leq \left\|\oU\oLambda\bL-\oX\oU^*\right\|_{2,\infty}+\left\|\oW\oU^*\right\|_{2,\infty}+\left\|\oU\bL\oLambda^*-\oU\oLambda\bL\right\|_{2,\infty}.\label{lem4eq7}
\end{align}
Since the upper bound of $\left\|\oU\oLambda\bL-\oX\oU^*\right\|_{2,\infty}$ follows from Lemma \ref{lemma3}, we only have to deal with the second and third quantity. We begin with the second quantity. By Lemma \ref{lemma1}, since $\sigma_{K-1}(\bL)\geq 1/2$, we have
\begin{align}
    \left\|\oU\bL\oLambda^*-\oU\oLambda\bL\right\|_{2,\infty} &= \left\|\oU\right\|_{2,\infty}\left\|\bL\oLambda^*-\oLambda\bL\right\|\lesssim \left\|\oU\bL\right\|_{2,\infty}\left\|\bL\oLambda^*-\oLambda\bL\right\|  \nonumber \\
    &\leq \left(\left\|\oU^*\right\|_{2,\infty}+ \left\|\oU\bL-\oU^*\right\|_{2,\infty}\right)\left\|\bL\oLambda^*-\oLambda\bL\right\| \nonumber \\
    & \leq \left(\sqrt{\frac{(K-1)\mu^*}{n}}+ \left\|\oU\bL-\oU^*\right\|_{2,\infty}\right)\left\|\bL\oLambda^*-\oLambda\bL\right\| ,\label{lem4eq3}
\end{align}
where the last inequality follows from \eqref{eq:incoherence}. In addition,
\begin{align}
    \left\|\bL\oLambda^*-\oLambda\bL\right\| &=\left\|\left(\bL^\top\right)^{-1}\left(\bL^\top\bL\oLambda^*-\bL^\top\oLambda\bL\right)\right\| \lesssim\left\|\bL^\top\bL\oLambda^*-\bL^\top\oLambda\bL\right\|  \nonumber \\
    &\leq \left\|\oLambda^*-\bL^\top\oLambda\bL\right\| +\left\|\bL^\top\bL\oLambda^*-\oLambda^*\right\| \nonumber \\
    &\lesssim \frac{\theta_{\text{max}}^3 n^{1.5}}{\minsigma^{*2}} +\sqrt{(K-1)\log n}\theta_{\text{max}}+\maxsigma^{*}\left\|\bL^\top\bL-\bI\right\|. \label{lem4eq1}
\end{align}
Since $\bR^\top\bR = \bI$, by Lemma \ref{lemma1} we have
\begin{align}
    \left\|\bL^\top\bL-\bI\right\| &= \left\|\bL^\top\bL-\bR^\top\bR\right\|\leq \|\bL^\top\left(\bL-\bR\right)\|+\|\left(\bL-\bR\right)^\top\bR\| \lesssim \|\bL-\bR\|\lesssim \frac{n\theta_{\text{max}}^2}{\minsigma^{*2}}.\label{lem4eq2}
\end{align}
Combining \eqref{lem4eq1} and \eqref{lem4eq2}, we get
\begin{align}\label{eq:define_xi1}
    \left\|\bL\oLambda^*-\oLambda\bL\right\|&\lesssim \frac{\theta_{\text{max}}^3 n^{1.5}}{\minsigma^{*2}} +\sqrt{(K-1)\log n}\theta_{\text{max}}+\maxsigma^{*}\frac{n\theta_{\text{max}}^2}{\minsigma^{*2}} \nonumber\\
    &\lesssim \sqrt{(K-1)\log n}\theta_{\text{max}}+\frac{ \kappa^* n \theta_{\text{max}}^2}{\minsigma^{*}}=: \xi_1
\end{align}
since $\sqrt{n}\theta_{\text{max}}\ll \minsigma^*\leq \maxsigma^*$. Therefore, by~\eqref{lem4eq3} we have
\begin{align}
    \left\|\oU\bL\oLambda^*-\oU\oLambda\bL\right\|_{2,\infty} \lesssim \xi_1 \left(\sqrt{\frac{(K-1)\mu^*}{n}}+ \left\|\oU\bL-\oU^*\right\|_{2,\infty}\right).\label{lem4eq8}
\end{align}
 We now turn to bound the third summand of \eqref{lem4eq7}, namely, $\left\|\oW\oU^*\right\|_{2,\infty}$. Recall the decomposition of $\oW$ from \eqref{eq:barw_expand}.
By Lemma \ref{Wconcentration3} we have
\begin{align}
    \left\|\bW\oU^*\right\|_{2,\infty}&\lesssim \sqrt{\log n}\theta_{\text{max}}\left\|\oU^*\right\|_F + \log n \left\|\oU^*\right\|_{2,\infty} \nonumber \\
    &\lesssim \sqrt{(K-1)\log n}\theta_{\text{max}} + \log n\sqrt{(K-1)\mu^* / n}\label{lem4eq4}
\end{align}
with probability at least $1-O(n^{-14})$. And, since $\left(\bW\bu_1^*\bu_1^{*\top}\right)^\top\bN\oU^* = \bu_1^*\bu_1^{*\top}\bW\bN\oU^*$ is a rank-$1$ matrix, we have
\begin{align}
    \left\|\left(\bW\bu_1^*\bu_1^{*\top}\right)^\top\bN\oU^*\right\|_{2,\infty} &= \left\|\bu_1^*\right\|_\infty\left\|\bu_1^{*\top}\bW\bN\oU^*\right\|_2 \leq \left\|\bu_1^*\right\|_\infty\left\|\bu_1^{*}\right\|_2\left\|\bW\right\|\left\|\bN\right\|\left\|\oU^*\right\| \nonumber \\
    &\lesssim \sqrt{\frac{\mu^*}{n}}\sqrt{n}\theta_{\text{max}} = \sqrt{\mu^*}\theta_{\text{max}}. \label{lem4eq5}
\end{align}
And, since $\left\|\bDelta\right\|_{2,\infty}\leq\left\|\bDelta\right\|$, we have via \eqref{eq:incoherence},
\begin{align}
    \left\|\bDelta\oU^*\right\|_{2,\infty} \leq \left\|\bDelta\right\|_{2,\infty}\left\|\oU^*\right\|\leq \left\|\bDelta\right\|\left\|\oU^*\right\|\lesssim \frac{n\theta_{\text{max}}^2}{\lambda_1^*}. \label{lem4eq6}
\end{align}
Finally, since $\bu_1^{*\top}\oU^* = \boldsymbol{0}$, we know that
$\left\|\bu_1^*\bu_1^{*\top}\bW\bu_1^*\bu_1^{*\top}\oU^*\right\|_{2,\infty} = \left\|\bN\bW\bu_1^*\bu_1^{*\top}\oU^*\right\|_{2,\infty} = 0$. This, along with \eqref{lem4eq4},
 \eqref{lem4eq5} and \eqref{lem4eq6} yields
\begin{align}
    \left\|\oW\oU^*\right\|_{2,\infty}\lesssim&\sqrt{(K-1)\log n}\theta_{\text{max}} + \log n\sqrt{(K-1)\mu^* / n} + \sqrt{\mu^*}\theta_{\text{max}}+\frac{n\theta_{\text{max}}^2}{\lambda_1^*} \nonumber \\
    &=: \xi_2.\label{lem4eq9}
\end{align}
Combining \eqref{lem4eq7},~\eqref{lem4eq8}, \eqref{lem4eq9} and Lemma \ref{lemma3} we have

\begin{align*}
    \left\|\oU\bL-\oU^*\right\|_{2,\infty}\lesssim&\Big(\frac{\xi_1+\log n}{\minsigma^\star}\Big)\left\|\oU\bL-\oU^*\right\|_{2,\infty}+\frac{\xi_1}{\minsigma^*}\sqrt{\frac{(K-1)\mu^*}{n}} +\frac{\xi_2}{\minsigma^\star}\\
    &+\frac{\kappa^*}{\minsigma^{*2}}\sqrt{(K-1)\mu^*n}\theta_{\text{max}}^2+\frac{\sqrt{(K-1)n\log n}}{\minsigma^{*2}}\theta_{\text{max}} \\
    &+ \frac{(\sqrt{(K-1)n}\theta_{\text{max}}+\log n) n\theta_{\text{max}}^2}{\lambda_1^*\minsigma^{*2}} + \frac{(\sqrt{n}\theta_{\text{max}}+\log n)^2\sqrt{(K-1)\mu^*/n}}{\minsigma^{*2}}  \\
    \lesssim&\Big(\frac{\xi_1+\log n}{\minsigma^\star}\Big)\left\|\oU\bL-\oU^*\right\|_{2,\infty}  +\frac{\xi_2}{\minsigma^\star}\\
    &+\frac{\kappa^*}{\minsigma^{*2}}\sqrt{(K-1)\mu^*n}\theta_{\text{max}}^2+\frac{\sqrt{(K-1)n\log n}}{\minsigma^{*2}}\theta_{\text{max}} + \frac{\sqrt{K-1} n^{1.5}\theta_{\text{max}}^3}{\lambda_1^*\minsigma^{*2}}   
\end{align*}
since $n\gtrsim\mu^*\max\{\log^2n, K-1\}$. When we further have $\xi_1/\minsigma^*\ll 1$, the first summand is negligible and we obtain the desired conclusion.
\end{proof}

\subsection{Proof of Lemma \ref{lemma5}}\label{lemma5proof}
\begin{proof}
From \eqref{firstsubspaceeq1} we know that
\begin{align*}
    \bDelta = & \frac{1}{2\pi i}\oint_{\mathcal{C}_1}\bW \left(\lambda\bI-\bX \right)^{-1}\bW\left(\lambda\bI-\bH \right)^{-1}d\lambda \\
    &+\frac{1}{2\pi i}\oint_{\mathcal{C}_1}\bH \left(\lambda\bI-\bX \right)^{-1}\bW\left(\lambda\bI-\bH \right)^{-1}\bW\left(\lambda\bI-\bH \right)^{-1}d\lambda
\end{align*}
We define the following four matrices
\begin{align*}
    \bDelta_1 &= \bW \left(\lambda\bI-\bH \right)^{-1}\bW\left(\lambda\bI-\bH \right)^{-1} \\
    \bDelta_2 &= \bW \left[\left(\lambda\bI-\bX \right)^{-1}-\left(\lambda\bI-\bH \right)^{-1}\right]\bW\left(\lambda\bI-\bH \right)^{-1} \\
    &= \bW \left(\lambda\bI-\bX \right)^{-1}\bW\left(\lambda\bI-\bH\right)^{-1}\bW\left(\lambda\bI-\bH \right)^{-1}  \\
    \bDelta_3 &= \bH \left(\lambda\bI-\bH \right)^{-1}\bW\left(\lambda\bI-\bH \right)^{-1}\bW\left(\lambda\bI-\bH \right)^{-1}  \\
    \bDelta_4 &= \bH \left[\left(\lambda\bI-\bX \right)^{-1}-\left(\lambda\bI-\bH \right)^{-1}\right]\bW\left(\lambda\bI-\bH \right)^{-1}\bW\left(\lambda\bI-\bH \right)^{-1}  \\
    &= \bH \left(\lambda\bI-\bX \right)^{-1}\bW\left(\lambda\bI-\bH \right)^{-1}\bW\left(\lambda\bI-\bH \right)^{-1}\bW\left(\lambda\bI-\bH \right)^{-1}.
\end{align*}
\textbf{Control} $\left\|\frac{1}{2\pi i}\oint_{\mathcal{C}_1}\bDelta_1 d\lambda\oU^*\right\|_{2,\infty}$: For $2\leq i\leq K$, one can show that
\begin{align*}
    \left[\frac{1}{2\pi i}\oint_{\mathcal{C}_1}\bDelta_1 d\lambda\oU^*\right]_{\cdot,i-1} &= \frac{1}{2\pi i}\oint_{\mathcal{C}_1}\sum_{j,k=1}^n\frac{1}{(\lambda-\lambda_j^*)(\lambda-\lambda_k^*)}\bW\bu_j^*\bu_j^{*\top}\bW\bu_k\bu_k^{*\top}\bu_i^* d\lambda \\
    &= \frac{1}{2\pi i}\oint_{\mathcal{C}_1}\sum_{j=1}^n\frac{1}{(\lambda-\lambda_j^*)(\lambda-\lambda_i^*)}\bW\bu_j^*\bu_j^{*\top}\bW\bu_i^* d\lambda   \\
    & = \sum_{j=1}^n\textbf{Res}\left(\frac{1}{(\lambda-\lambda_j^*)(\lambda-\lambda_i^*)},\lambda_1^*\right)\bW\bu_j^*\bu_j^{*\top}\bW\bu_i^* \\
    & = \frac{1}{\lambda_1^*-\lambda_i^*}\bW\bu_1^*\bu_1^{*\top}\bW\bu_i^* = \frac{\bu_1^{*\top}\bW\bu_i^*}{\lambda_1^*-\lambda_i^*}\bW\bu_1^*.
\end{align*}
As a result, we have
\begin{align*}
    \left\|\frac{1}{2\pi i}\oint_{\mathcal{C}_1}\bDelta_1 d\lambda\oU^*\right\|_{2,\infty} = \left\|\bW\bu_1^*\right\|_{\infty}\sqrt{\sum_{i=2}^K\left(\frac{\bu_1^{*\top}\bW\bu_i^*}{\lambda_1^*-\lambda_i^*}\right)^2}.
\end{align*}
By Lemma \ref{Wconcentration3}, we know that $\left\|\bW\bu_1^*\right\|_{\infty}\lesssim \sqrt{\log n}\theta_{\text{max}}+\log n\sqrt{\mu^*/n}$ with probability at least $1-O(n^{-14})$. For each $2\leq i\leq K$, by Lemma \ref{Wconcentration6} we know that $\left|\bu_1^{*\top}\bW\bu_i^*\right|\lesssim \sqrt{\log n} \theta_{\text{max}}+\log n \sqrt{\mu^*/n}$ with probability at least $1-O(n^{-15})$. As a result, we know that
\begin{align}
    \left\|\frac{1}{2\pi i}\oint_{\mathcal{C}_1}\bDelta_1 d\lambda\oU^*\right\|_{2,\infty}&\lesssim \left(\sqrt{\log n}\theta_{\text{max}}+\log n\sqrt{\mu^*/n}\right)^2\frac{\sqrt{K-1}}{\lambda_1^*}  \label{lem5eq5}
\end{align}
with probability at least $1-O(n^{-14})$.

\textbf{Control} $\left\|\frac{1}{2\pi i}\oint_{\mathcal{C}_1}\bDelta_2 d\lambda\oU^*\right\|$: By definition we have
\begin{align}
    \left\|\frac{1}{2\pi i}\oint_{\mathcal{C}_1}\bDelta_2 d\lambda\oU^*\right\|&\leq \left\|\frac{1}{2\pi i}\oint_{\mathcal{C}_1}\bDelta_2 d\lambda\right\|\leq \frac{1}{2\pi}\oint_{\mathcal{C}_1}\left\|\bDelta_2\right\| d\lambda \nonumber \\
    &\leq \frac{1}{2\pi}\oint_{\mathcal{C}_1}\left\|\bW\right\|^3\left\|\left(\lambda\bI-\bX\right)^{-1}\right\|\left\|\left(\lambda\bI-\bH\right)^{-1}\right\|^2 d\lambda \nonumber \\
    &\lesssim \oint_{\mathcal{C}_1}\frac{(\sqrt{n}\theta_{\text{max}})^3}{\lambda_1^{*3}}d\lambda \lesssim \frac{n^{1.5}\theta_{\text{max}}^3r}{\lambda_1^{*3}}\lesssim \frac{n^{1.5}\theta_{\text{max}}^3}{\lambda_1^{*2}}. \label{lem5eq6}
\end{align}

\textbf{Control} $\left\|\frac{1}{2\pi i}\oint_{\mathcal{C}_1}\bDelta_3 d\lambda\oU^*\right\|_{2,\infty}$: For $2\leq i\leq K$, one can show that
\begin{align*}
    &\left[\frac{1}{2\pi i}\oint_{\mathcal{C}_1}\bDelta_3 d\lambda\oU^*\right]_{\cdot,i-1} = \frac{1}{2\pi i}\oint_{\mathcal{C}_1}\sum_{j,k,l=1}^n\frac{\lambda_j^*}{(\lambda-\lambda_j^*)(\lambda-\lambda_k^*)(\lambda-\lambda_l^*)}\bu_j^*\bu_j^{*\top}\bW\bu_k^*\bu_k^{*\top}\bW\bu_l\bu_l^{*\top}\bu_i^* d\lambda  \\
    =& \frac{1}{2\pi i}\oint_{\mathcal{C}_1}\sum_{j,k=1}^n\frac{\lambda_j^*}{(\lambda-\lambda_j^*)(\lambda-\lambda_k^*)(\lambda-\lambda_i^*)}\bu_j^*\bu_j^{*\top}\bW\bu_k^*\bu_k^{*\top}\bW\bu_i^* d\lambda  \\
    =&\sum_{j,k=1}^n\textbf{Res}\left(\frac{\lambda_j^*}{(\lambda-\lambda_j^*)(\lambda-\lambda_k^*)(\lambda-\lambda_i^*)},\lambda_1^*\right)\bu_j^*\bu_j^{*\top}\bW\bu_k^*\bu_k^{*\top}\bW\bu_i^*  \\
    =& \sum_{j=2}^n\frac{\lambda_j^*}{(\lambda_1^*-\lambda_j^*)(\lambda_1^*-\lambda_i^*)}\bu_j^*\bu_j^{*\top}\bW\bu_1^*\bu_1^{*\top}\bW\bu_i^*  \\
    & + \sum_{j=2}^n\frac{\lambda_1^*}{(\lambda_1^*-\lambda_j^*)(\lambda_1^*-\lambda_i^*)}\bu_1^*\bu_1^{*\top}\bW\bu_j^*\bu_j^{*\top}\bW\bu_i^* \\
    &-\frac{\lambda_1^*}{(\lambda_1^*-\lambda_i^*)^2}\bu_1^*\bu_1^{*\top}\bW\bu_1^*\bu_1^{*\top}\bW\bu_i^*  \\
    =& \frac{1}{\lambda_1^*-\lambda_i^*}\left(\bN-\bI\right)\bW\bu_1^*\bu_1^{*\top}\bW\bu_i^* +\frac{1}{\lambda_1^*-\lambda_i^*}\bu_1^*\bu_1^{*\top}\bW\bN\bW\bu_i^* \\
    &-\frac{\lambda_1^*}{(\lambda_1^*-\lambda_i^*)^2}\bu_1^*\bu_1^{*\top}\bW\bu_1^*\bu_1^{*\top}\bW\bu_i^*.
\end{align*}
Let 
\begin{align*}
    \bC_1 &= \textbf{diag}\left(\frac{1}{\lambda_1^*-\lambda_2^*}, \frac{1}{\lambda_1^*-\lambda_3^*},\dots, \frac{1}{\lambda_1^*-\lambda_K^*}\right);  \\
    \bC_2 &= \textbf{diag}\left(\frac{\lambda_1^*}{(\lambda_1^*-\lambda_2^*)^2}, \frac{\lambda_1^*}{(\lambda_1^*-\lambda_3^*)^2},\dots, \frac{\lambda_1^*}{(\lambda_1^*-\lambda_K^*)^2}\right).
\end{align*}
Then we have
\begin{align}
    \frac{1}{2\pi i}\oint_{\mathcal{C}_1}\bDelta_3 d\lambda\oU^* =& \left(\bN-\bI\right)\bW\bu_1^*\bu_1^{*\top}\bW\oU^*\bC_1 +\bu_1^*\bu_1^{*\top}\bW\bN\bW\oU^*\bC_1 \nonumber \\
    &-\bu_1^*\bu_1^{*\top}\bW\bu_1^*\bu_1^{*\top}\bW\oU^*\bC_2.\label{lem5eq1}
\end{align}
Note that $\bN = \lambda_1^*\bN_1$ for $\bN_1$ defined in ~\eqref{N1N2}. In the proof of Lemma \ref{lemmaN1N2} we actually show that $ \left\|\bN-\bI\right\|_{2,\infty}\lesssim \sqrt{\frac{(K-1)\mu^*}{n}}$, which implies
\begin{align}
    \left\|\left(\bN-\bI\right)\bW\bu_1^*\bu_1^{*\top}\bW\oU^*\bC_1\right\|_{2,\infty}&\leq \left\|\bN-\bI\right\|_{2,\infty}\left\|\bW\bu_1^*\bu_1^{*\top}\bW\oU^*\bC_1\right\| \nonumber \\
    &\leq \left\|\bN-\bI\right\|_{2,\infty}\left\|\bW\right\|^2\left\|\bC_1\right\|  \nonumber \\
    &\lesssim \sqrt{\frac{(K-1)\mu^*}{n}} \frac{n\theta_{\text{max}}^2}{\lambda_1^{*}} = \sqrt{\frac{(K-1)\mu^*n}{\lambda_1^{*2}}}\theta_{\text{max}}^2. \label{lem5eq2}
\end{align}
For $\bu_1^*\bu_1^{*\top}\bW\bN\bW\oU^*\bC_1$ and $\bu_1^*\bu_1^{*\top}\bW\bu_1^*\bu_1^{*\top}\bW\oU^*\bC_2$, since both of them are rank-$1$ matrices, we have, using  $\left\|\bN\right\|\lesssim 1$ from Lemma \ref{lemmaN1N2}, 
\begin{align}
    \left\|\bu_1^*\bu_1^{*\top}\bW\bN\bW\oU^*\bC_1\right\|_{2,\infty} &= \left\|\bu_1^*\right\|_\infty\left\|\bu_1^{*\top}\bW\bN\bW\oU^*\bC_1\right\| \nonumber \\
    &\leq \left\|\bu_1^*\right\|_\infty\left\|\bW\right\|^2\left\|\bN\right\|\left\|\bC_1\right\| \lesssim \sqrt{\frac{\mu^*}{n}}\frac{n\theta_{\text{max}}^2}{\lambda_1^*} = \sqrt{\frac{\mu^*n}{\lambda_1^{*2}}}\theta_{\text{max}}^2\label{lem5eq3}
\end{align}
and
\begin{align}
    \left\|\bu_1^*\bu_1^{*\top}\bW\bu_1^*\bu_1^{*\top}\bW\oU^*\bC_2\right\|_{2,\infty} &= \left\|\bu_1^*\right\|_\infty\left\|\bu_1^{*\top}\bW\bu_1^*\bu_1^{*\top}\bW\oU^*\bC_2\right\| \nonumber \\
    &\leq \left\|\bu_1^*\right\|_\infty\left\|\bW\right\|^2\left\|\bN\right\|\left\|\bC_2\right\| \lesssim = \sqrt{\frac{\mu^*n}{\lambda_1^{*2}}}\theta_{\text{max}}^2. \label{lem5eq4}
\end{align}
Now, combining ~\eqref{lem5eq1},~\eqref{lem5eq2},~\eqref{lem5eq3} ~\eqref{lem5eq4} yields
\begin{align}
    \Big\|\frac{1}{2\pi i}\oint_{\mathcal{C}_1}\bDelta_3 d\lambda\oU^*\Big\|\lesssim \sqrt{\frac{(K-1)\mu^*n}{\lambda_1^{*2}}}\theta_{\text{max}}^2.  \label{lem5eq7}
\end{align}

\textbf{Control} $\left\|\frac{1}{2\pi i}\oint_{\mathcal{C}_1}\bDelta_4 d\lambda\oU^*\right\|$:
By definition we have
\begin{align}
    \Big\|\frac{1}{2\pi i}\oint_{\mathcal{C}_1}\bDelta_4 d\lambda\oU^*\Big\|&\leq \Big\|\frac{1}{2\pi i}\oint_{\mathcal{C}_1}\bDelta_4 d\lambda\Big\|\leq \frac{1}{2\pi}\oint_{\mathcal{C}_1}\left\|\bDelta_4\right\| d\lambda \nonumber \\
    &\leq \frac{1}{2\pi}\oint_{\mathcal{C}_1}\left\|\bH\right\|\|\left(\lambda\bI-\bX\right)^{-1}\|\left\|\bW\right\|^3\|\left(\lambda\bI-\bH\right)^{-1}\|^3 d\lambda  \nonumber \\
    &\lesssim \oint_{\mathcal{C}_1}\frac{\lambda_1^*(\sqrt{n}\theta_{\text{max}})^3}{\lambda_1^{*4}}d\lambda \lesssim \frac{n^{1.5}\theta_{\text{max}}^3r}{\lambda_1^{*3}}\lesssim \frac{n^{1.5}\theta_{\text{max}}^3}{\lambda_1^{*2}}, \label{lem5eq8}
\end{align}
where the second inequality uses \eqref{eq:inverse_spectral_bound}. \par 
Next, we combine the four bounds obtained above. More precisely, using ~\eqref{lem5eq5}, ~\eqref{lem5eq6}, ~\eqref{lem5eq7} and ~\eqref{lem5eq8} with triangle inequality, we obtain
\begin{align*}
    \left\|\bDelta\oU^*\right\|_{2,\infty}\leq & \sum\limits_{i=1}^4 \left\|\frac{1}{2\pi i}\oint_{\mathcal{C}_1}\bDelta_i d\lambda\oU^*\right\|_{2,\infty}\\
    \leq & \sum\limits_{i=1,3}\left\|\frac{1}{2\pi i}\oint_{\mathcal{C}_1}\bDelta_i d\lambda\oU^*\right\|_{2,\infty}+ \sum\limits_{i=2,4}\left\|\frac{1}{2\pi i}\oint_{\mathcal{C}_1}\bDelta_i d\lambda\oU^*\right\|  \\
    \lesssim  &\left(\sqrt{\log n}\theta_{\text{max}}+\log n\sqrt{\mu^*/n}\right)^2\frac{\sqrt{K-1}}{\lambda_1^*}+ \frac{n^{1.5}\theta_{\text{max}}^3}{\lambda_1^{*2}}+ \sqrt{\frac{(K-1)\mu^*n}{\lambda_1^{*2}}}\theta_{\text{max}}^2\\
    \lesssim & \frac{\sqrt{(K-1)\mu^*n}\theta_{\text{max}}^2}{\lambda_1^*}+\frac{\sqrt{K-1}\mu^*\log ^2n}{n\lambda_1^*}+\frac{n^{1.5}\theta_{\text{max}}^3}{\lambda_1^{*2}}
\end{align*}
with probability at least $1-O(n^{-10})$.
\end{proof}

\subsection{Proof of Theorem \ref{mainthmmatrixdenoising}}\label{mainthmmatrixdenoisingproof}
\begin{proof}
As for the first step, we have
\begin{align*}
    \oX\oU^*\left(\oLambda^*\right)^{-1} =& \left(\oH+\oW\right)\oU^*\left(\oLambda^*\right)^{-1}  = \oU^* +\oW\oU^*\left(\oLambda^*\right)^{-1}  \\
    =&\oU^* +\left[\bW-\bu_1^*\bu_1^{*\top}\bW\bu_1^*\bu_1^{*\top}- \bN\bW\bu_1^*\bu_1^{*\top}-\left(\bW\bu_1^*\bu_1^{*\top}\right)^\top\bN\right]\oU^*\left(\oLambda^*\right)^{-1} \\
    &-\bDelta \oU^*\left(\oLambda^*\right)^{-1}  \\
    =&\oU^* +\left[\bW-\bu_1^*\bu_1^{*\top}\bW\bN\right]\oU^*\left(\oLambda^*\right)^{-1} -\bDelta \oU^*\left(\oLambda^*\right)^{-1},
\end{align*}
and
\begin{align*}
    \oU\bR&=\oU\bR \oLambda^*\left(\oLambda^*\right)^{-1} = \left[\oU\oLambda\bR +\oU\left(\bR \oLambda^* - \oLambda\bR\right)\right]\left(\oLambda^*\right)^{-1}  \\
    &= \left[\oU\oLambda\bL +\oU\oLambda\left(\bR-\bL\right)+\oU\left(\bR \oLambda^* - \oLambda\bR\right)\right]\left(\oLambda^*\right)^{-1} \\
    &= \left[\oX\oU^*+\left(\oU\oLambda\bL-\oX\oU^*\right) +\oU\oLambda\left(\bR-\bL\right)+\oU\left(\bR \oLambda^* - \oLambda\bR\right)\right]\left(\oLambda^*\right)^{-1} \\
    &= \oX\oU^*\left(\oLambda^*\right)^{-1}+\left[\left(\oU\oLambda\bL-\oX\oU^*\right) +\oU\oLambda\left(\bR-\bL\right)+\oU\left(\bR \oLambda^* - \oLambda\bR\right)\right]\left(\oLambda^*\right)^{-1}
\end{align*}
As a result, we know that
\begin{align*}
    \boldsymbol{\Psi}_{\oU} 
    = \left[\left(\oU\oLambda\bL-\oX\oU^*\right)+\oU\oLambda\left(\bR-\bL\right)+\oU\left(\bR \oLambda^* - \oLambda\bR\right)-\bDelta \oU^*\right]\left(\oLambda^*\right)^{-1}.
\end{align*}
This imples that 
\begin{align*}
    \left\|\boldsymbol{\Psi}_{\oU}\right\|_{2,\infty}\leq\frac{\left\|\oU\oLambda\bL-\oX\oU^*\right\|_{2,\infty}+\left\|\oU\oLambda\left(\bR-\bL\right)\right\|_{2,\infty}+\left\|\oU\left(\bR \oLambda^* - \oLambda\bR\right)\right\|_{2,\infty}+\left\|\bDelta \oU^*\right\|_{2,\infty}}{\minsigma^*}.
\end{align*}
We now bound the numerator. Note that the last term is already bounded by Lemma \ref{lemma5}.

\textbf{Control} $\left\|\oU\oLambda\bL-\oX\oU^*\right\|_{2,\infty}$: Combine \eqref{lem4eq7} and \eqref{lem4eq8} we know that
\begin{align*}
    \left\|\oU\bL-\oU^*\right\|_{2,\infty}\leq  &\frac{1}{\minsigma^*}\left\|\oU\oLambda\bL-\oX\oU^*\right\|_{2,\infty}+\frac{1}{\minsigma^*}\left\|\oW\oU^*\right\|_{2,\infty}+\frac{1}{\minsigma^*}\left\|\oU\bL\oLambda^*-\oU\oLambda\bL\right\|_{2,\infty} \\
    \lesssim & \frac{1}{\minsigma^*}\left\|\oU\oLambda\bL-\oX\oU^*\right\|_{2,\infty}+\frac{1}{\minsigma^*}\left\|\oW\oU^*\right\|_{2,\infty}\\
    &+\frac{\xi_1}{\minsigma^\star}\left(\sqrt{\frac{(K-1)\mu^*}{n}}+ \left\|\oU\bL-\oU^*\right\|_{2,\infty}\right),
 \end{align*}
 where $\xi_1$ is defined by \eqref{eq:define_xi1}. Since we assumed $\xi_1/\minsigma^*\ll 1$, we further have
\begin{align*}
    \left\|\oU\bL-\oU^*\right\|_{2,\infty}
    \lesssim & \frac{1}{\minsigma^*}\left\|\oU\oLambda\bL-\oX\oU^*\right\|_{2,\infty}+ \frac{1}{\minsigma^*}\left\|\oW\oU^*\right\|_{2,\infty} + \frac{\xi_1}{\minsigma^\star} \sqrt{\frac{(K-1)\mu^*}{n}}.
\end{align*}
Plugging this in Lemma \ref{lemma3} we get
\begin{align*}
    \left\|\oU\oLambda\bL-\oX\oU^*\right\|_{2,\infty}\lesssim &\Bigg(\frac{n \theta_{\text{max}}^2}{\minsigma^{*2}}+\frac{\log n}{\minsigma^*}\Bigg)\Bigg(\left\|\oU\oLambda\bL-\oX\oU^*\right\|_{2,\infty}+\left\|\oW\oU^*\right\|_{2,\infty}+ \sqrt{\frac{(K-1)\mu^*}{n}}\frac{\xi_1}{\minsigma^\star} \Bigg) \\
    &+\frac{\sqrt{(K-1)n\log n}}{\minsigma^*}\theta_{\text{max}}+\frac{\kappa^*}{\minsigma^{*}}\sqrt{(K-1)\mu^*n}\theta_{\text{max}}^2\\
    &+\frac{(\sqrt{(K-1)n}\theta_{\text{max}}+\log n) n\theta_{\text{max}}^2}{\lambda_1^*\minsigma^*} + \frac{(\sqrt{n}\theta_{\text{max}}+\log n)^2\sqrt{(K-1)\mu^*/n}}{\minsigma^*}.
\end{align*}
Since $\max\{\sqrt{n}\theta_{\text{max}}, \log n\}\ll \minsigma^*$, and by \eqref{lem4eq9}, we know that
\begin{align*}
    \left\|\oU\oLambda\bL-\oX\oU^*\right\|_{2,\infty}\lesssim &\left(\frac{n \theta_{\text{max}}^2}{\minsigma^{*2}}+\frac{\log n}{\minsigma^*}\right)\left\|\oW\oU^*\right\|_{2,\infty}+\frac{\sqrt{(K-1)n\log n}}{\minsigma^*}\theta_{\text{max}}+\frac{\kappa^*}{\minsigma^{*}}\sqrt{(K-1)\mu^*n}\theta_{\text{max}}^2 \\
    &+ \left(\frac{n \theta_{\text{max}}}{\minsigma^{*2}}+\log n\right) \sqrt{\frac{(K-1)\mu^*}{n}}\frac{\xi_1}{\minsigma^\star} \\
    &+\frac{(\sqrt{(K-1)n}\theta_{\text{max}}+\log n) n\theta_{\text{max}}^2}{\lambda_1^*\minsigma^*} + \frac{(\sqrt{n}\theta_{\text{max}}+\log n)^2\sqrt{(K-1)\mu^*/n}}{\minsigma^*} \\
    \lesssim & \frac{\sqrt{(K-1)n\log n}+\kappa^*\sqrt{(K-1)\mu^*n}\theta_{\text{max}}^2 }{\minsigma^*}+\frac{n\theta_{\text{max}}^2\log n + \sqrt{K-1} n^{1.5}\theta_{\text{max}}^3}{\lambda_1^*\minsigma^*}
\end{align*}
since we assumed $\mu^*\log^2 n\lesssim n$. According to Lemma \ref{eigenvaluelemma}, since $K^{-1}\left\|\btheta\right\|_2^2\lesssim \lambda_1^*$ and $\minsigma\asymp \beta_n K^{-1}\left\|\btheta\right\|_2^2$, we have, using $K\log n\lesssim n$,
\begin{align}
     \left\|\oU\oLambda\bL-\oX\oU^*\right\|_{2,\infty}&\lesssim \frac{K^{1.5}\log^{0.5}n}{n^{0.5}\beta_n\theta_{\text{max}}^2}+\frac{\kappa^* K^{1.5}\sqrt{\mu^*}}{n^{0.5}\beta_n}+\frac{K^2\log n}{n\beta_n\theta_{\text{max}}^2} + \frac{K^{2.5}}{n^{0.5}\beta_n\theta_{\text{max}}}  \nonumber \\
     &\lesssim \frac{K^{1.5}\log^{0.5}n}{n^{0.5}\beta_n\theta_{\text{max}}^2}+\frac{\kappa^* K^{1.5}\sqrt{\mu^*}}{n^{0.5}\beta_n} + \frac{K^{2.5}}{n^{0.5}\beta_n\theta_{\text{max}}}\label{denoisingthmeq4}.
\end{align}

\textbf{Control} $\left\|\oU\oLambda\left(\bR-\bL\right)\right\|_{2,\infty}+\left\|\oU\left(\bR \oLambda^* - \oLambda\bR\right)\right\|_{2,\infty}$: By Lemma \ref{lemma1} we have
\begin{align*}
    \left\|\oU\oLambda\left(\bR-\bL\right)\right\|_{2,\infty}&\leq \left\|\oU\right\|_{2,\infty}\left\|\oLambda\left(\bR-\bL\right)\right\| \leq \maxsigma^*\left\|\oU\right\|_{2,\infty}\left\|\bR-\bL\right\|  \nonumber \\
    &\lesssim \frac{n\theta_{\text{max}}^2\maxsigma^*}{\minsigma^{*2}}\left\|\oU\right\|_{2,\infty} = \frac{\kappa^*n\theta_{\text{max}}^2}{\minsigma^{*2}}\left\|\oU\right\|_{2,\infty}.
\end{align*}
By Lemma \ref{lemma2} we have
\begin{align*}
    \left\|\oU\left(\bR \oLambda^* - \oLambda\bR\right)\right\|_{2,\infty}&\leq \left\|\oU\right\|_{2,\infty}\left\|\bR \oLambda^* - \oLambda\bR\right\| = \left\|\oU\right\|_{2,\infty}\left\| \oLambda^* - \bR^\top\oLambda\bR\right\| \\
    &\lesssim \left(\frac{\kappa^*n\theta_{\text{max}}^2}{\minsigma^{*}} +\sqrt{(K-1)\log n}\theta_{\text{max}}\right)\left\|\oU\right\|_{2,\infty}  
\end{align*}
As a result, by Lemma \ref{eigenvaluelemma} we know that
\begin{align}
    \left\|\oU\oLambda\left(\bR-\bL\right)\right\|_{2,\infty}+\left\|\oU\left(\bR \oLambda^* - \oLambda\bR\right)\right\|_{2,\infty}&\lesssim \left(\frac{\kappa^*n\theta_{\text{max}}^2}{\minsigma^{*}} +\sqrt{(K-1)\log n}\theta_{\text{max}}\right)\left\|\oU\right\|_{2,\infty}  \nonumber \\
    &\lesssim \left(\frac{\kappa^*K}{\beta_n}+K^{0.5}\log^{0.5}n\theta_{\text{max}}\right)\left\|\oU\right\|_{2,\infty}\label{denoisingthmeq1}
\end{align}
By lemma \ref{lemma1} we know that $\sigma_{K-1}(\bL)\geq 1/2$. As a result, by Lemma \ref{lemma4} and Lemma \ref{eigenvaluelemma} we have
\begin{align}
    \left\|\oU\right\|_{2,\infty} \leq & 2\left\|\oU\bL\right\|_{2,\infty}\lesssim \left\|\oU\bL-\oU^*\right\|_{2,\infty} + \left\|\oU^*\right\|_{2,\infty} \nonumber \\
    \lesssim &\frac{\kappa^*}{\minsigma^{*2}}\sqrt{(K-1)\mu^*n}\theta_{\text{max}}^2+\frac{\sqrt{(K-1)n\log n}}{\minsigma^{*2}}\theta_{\text{max}} + \frac{\sqrt{K-1} n^{1.5}\theta_{\text{max}}^3}{\lambda_1^*\minsigma^{*2}}   \nonumber \\
    &+\frac{\sqrt{(K-1)\log n}\theta_{\text{max}} + \log n\sqrt{(K-1)\mu^* / n} + \sqrt{\mu^*}\theta_{\text{max}}}{\minsigma^*}+\frac{n\theta_{\text{max}}^2}{\lambda_1^*\minsigma^*} + \sqrt{\frac{(K-1)\mu^*}{n}}  \nonumber \\
    \lesssim &\frac{\kappa^*}{\minsigma^{*2}}\sqrt{(K-1)\mu^*n}\theta_{\text{max}}^2+\frac{\sqrt{(K-1)n\log n}}{\minsigma^{*2}}\theta_{\text{max}} + \frac{\sqrt{K-1} n^{1.5}\theta_{\text{max}}^3}{\lambda_1^*\minsigma^{*2}}   \nonumber \\
    &+\frac{\sqrt{(K-1)\log n}\theta_{\text{max}} }{\minsigma^*}+\frac{n\theta_{\text{max}}^2}{\lambda_1^*\minsigma^*} + \sqrt{\frac{(K-1)\mu^*}{n}}  \nonumber \\
    \lesssim & \frac{\kappa^*K^{2.5}\sqrt{\mu^*}}{n^{1.5}\beta_n^2\theta_{\text{max}}^2}+\frac{K^{2.5}\log^{0.5}n}{n^{1.5}\beta_n^2\theta_{\text{max}}^3}+\frac{K^{3.5}}{n^{1.5}\beta_n^2\theta_{\text{max}}^3} +\frac{K^{1.5}\log^{0.5}n}{n\beta_n\theta_{\text{max}}}+\frac{K^2}{n\beta_n\theta_{\text{max}}^2} +\frac{K^{0.5}\sqrt{\mu^*}}{n^{0.5}} :=\xi_3\label{denoisingthmeq2}
\end{align}
with probability at least $1-O(n^{-10})$.
Combine ~\eqref{denoisingthmeq1} and ~\eqref{denoisingthmeq2} we get
\begin{align}
    &\left\|\oU\oLambda\left(\bR-\bL\right)\right\|_{2,\infty}+\left\|\oU\left(\bR \oLambda^* - \oLambda\bR\right)\right\|_{2,\infty}  \lesssim \xi_3 \left(\frac{\kappa^*K}{\beta_n}+K^{0.5}\log^{0.5}n\theta_{\text{max}}\right)\label{denoisingthmeq5}
\end{align}
with probability at least $1-O(n^{-10})$. 

\textbf{Control} $\left\|\bDelta \oU^*\right\|_{2,\infty} $: Plugging Lemma \ref{eigenvaluelemma} in Lemma \ref{lemma5} we get:
\begin{align}
\left\|\bDelta\oU^*\right\|_{2,\infty} &\lesssim\frac{\sqrt{(K-1)\mu^*n}\theta_{\text{max}}^2}{\lambda_1^*}+\frac{\sqrt{K-1}\mu^*\log ^2n}{n\lambda_1^*}+\frac{n^{1.5}\theta_{\text{max}}^3}{\lambda_1^{*2}} \nonumber \\
&\lesssim \frac{K^{1.5}\sqrt{\mu^*}}{n^{0.5}} +\frac{K^{1.5}\mu^*\log ^2n}{n^2\theta_{\text{max}}^2} +\frac{K^2}{n^{0.5}\theta_{\text{max}}} \lesssim \frac{K^{1.5}\sqrt{\mu^*}}{n^{0.5}} +\frac{K^2}{n^{0.5}\theta_{\text{max}}}\label{denoisingthmeq7}
\end{align} since $\mu^*\log^2 n\lesssim n$.

Combine \eqref{denoisingthmeq4}, \eqref{denoisingthmeq5} and Lemma \eqref{denoisingthmeq7} we get
\begin{align*}
    &\left\|\oU\oLambda\bL-\oX\oU^*\right\|_{2,\infty}+\left\|\oU\oLambda\left(\bR-\bL\right)\right\|_{2,\infty}+\left\|\oU\left(\bR \oLambda^* - \oLambda\bR\right)\right\|_{2,\infty}+\left\|\bDelta \oU^*\right\|_{2,\infty}  \\
    \lesssim & \xi_3 \left(\frac{\kappa^*K}{\beta_n}+K^{0.5}\log^{0.5}n\theta_{\text{max}}\right) \\
    &+\frac{K^{1.5}\log^{0.5}n}{n^{0.5}\beta_n\theta_{\text{max}}^2}+\frac{\kappa^* K^{1.5}\sqrt{\mu^*}}{n^{0.5}\beta_n} + \frac{K^{2.5}}{n^{0.5}\beta_n\theta_{\text{max}}}+\frac{K^{1.5}\sqrt{\mu^*}}{n^{0.5}} +\frac{K^2}{n^{0.5}\theta_{\text{max}}}  \\
     \lesssim & \xi_3 \left(\frac{\kappa^*K}{\beta_n}+K^{0.5}\log^{0.5}n\theta_{\text{max}}\right) +\frac{K^{1.5}\log^{0.5}n}{n^{0.5}\beta_n\theta_{\text{max}}^2}+ \frac{K^{2.5}}{n^{0.5}\beta_n\theta_{\text{max}}}=:\xi_4
\end{align*}
with probability at least $1-O(n^{-10})$, since 
    $n\gtrsim \max\left\{\mu^*\log^2 n, K\log n\right\}$, by our assumption.
Therefore, for $\boldsymbol{\Psi}_{\oU}$ we have
\begin{align*}
\left\|\boldsymbol{\Psi}_{\oU}\right\|_{2,\infty}\lesssim \frac{\xi_4}{\theta_{\text{max}}^2}.
\end{align*}
with probability at least $1-O(n^{-10})$, completing the proof.
\end{proof}

\subsection{Proof of Theorem \ref{mainthmrexpansion}}\label{mainthmrexpansionproof}
\begin{proof}
Note that by \cite[Lemma B.2]{jin2017estimating}, we have $(\bu_1^*)_i\asymp \theta_i/\left\|\btheta\right\|_2$ for all $i\in [n]$. Combine this with Lemma \ref{u1differenceinfty}, we can ensure $(\wh\bu_1)_i>0$ and $(\wh\bu_1)_i\asymp (\bu_1^*)_i$ for all $i\in [n]$ as long as
\begin{align*}
   \left\|\wh\bu_1-\bu_1^*\right\|_\infty&\lesssim \frac{K\log^{0.5} n+K^{1.5}\sqrt{\mu^*}}{n\theta_{\text{max}}}+\frac{K\sqrt{\mu^*}\log n}{n^{1.5}\theta_{\text{max}}^2}+\frac{K^3}{n^{1.5}\theta_{\text{max}}^3} \\
   &\ll \min_{i\in [n]}(\bu_1^*)_i\asymp \frac{\min_{i\in [n]}\theta_i}{\left\|\btheta\right\|_2}\asymp \frac{1}{\sqrt{n}}.
\end{align*}

We have already derived the expansion of $\wh\bu_1-\bu_1^*$ and $\oU\bR-\oU^*$ in Theorem \ref{u1expansion}, Theorem \ref{deltainfty} and Theorem \ref{mainthmmatrixdenoising}. For $i\in [n]$ we have 
\begin{align*}
    &(\wh\bu_1)_i = (\bu_1^*)_i+\left[\bN_1\bW\bu_1^*\right]_i+\delta_i , \\
    &[\oU\bR]_{i,\cdot}^\top = [\oU^*]_{i,\cdot}^\top+\bw_i + \left[\boldsymbol{\Psi}_{\oU}\right]_{i,\cdot}^\top.
\end{align*}
Since $[\oU\bR]_{i,\cdot}^\top = \bR^\top[\oU]_{i,\cdot}^\top$ and $\wh\br_i =[\oU]_{i,\cdot}^\top /(\wh\bu_1)_i $, we know that $[\oU\bR]_{i,\cdot}^\top / (\wh\bu_1)_i = \bR^\top\wh\br_i$. As a result, we have
\begin{align*}
    \bR^\top\wh\br_i-\br_i^* =& \frac{[\oU\bR]_{i,\cdot}^\top }{ (\wh\bu_1)_i} - \frac{[\oU^*]_{i,\cdot}^\top}{(\bu_1^*)_i} = \frac{[\oU^*]_{i,\cdot}^\top+\bw_i + \left[\boldsymbol{\Psi}_{\oU}\right]_{i,\cdot}^\top}{(\bu_1^*)_i+\left[\bN_1\bW\bu_1^*\right]_i+\delta_i} -\frac{[\oU^*]_{i,\cdot}^\top}{(\bu_1^*)_i}  \\
    = & \frac{(\bu_1^*)_i\bw_i+ (\bu_1^*)_i\left[\boldsymbol{\Psi}_{\oU}\right]_{i,\cdot}^\top -\left[\bN_1\bW\bu_1^*\right]_i[\oU^*]_{i,\cdot}^\top - \delta_i[\oU^*]_{i,\cdot}^\top}{\left((\bu_1^*)_i+\left[\bN_1\bW\bu_1^*\right]_i+\delta_i\right)(\bu_1^*)_i}   \\
    = & \frac{(\bu_1^*)_i\bw_i - \left[\bN_1\bW\bu_1^*\right]_i[\oU^*]_{i,\cdot}^\top}{(\bu_1^*)_i^2}\cdot \frac{(\bu_1^*)_i}{(\bu_1^*)_i+\left[\bN_1\bW\bu_1^*\right]_i+\delta_i}  \\
    &+\frac{(\bu_1^*)_i\left[\boldsymbol{\Psi}_{\oU}\right]_{i,\cdot}^\top - \delta_i[\oU^*]_{i,\cdot}^\top}{\left((\bu_1^*)_i+\left[\bN_1\bW\bu_1^*\right]_i+\delta_i\right)(\bu_1^*)_i} \\
    = & \frac{1}{(\bu_1^*)_i}\left(\bw_i-\left[\bN_1\bW\bu_1^*\right]_i\br_i^*\right)\left(1-\frac{\left[\bN_1\bW\bu_1^*\right]_i+\delta_i}{(\bu_1^*)_i+\left[\bN_1\bW\bu_1^*\right]_i+\delta_i}\right) \\
    & +\frac{\left[\boldsymbol{\Psi}_{\oU}\right]_{i,\cdot}^\top-\delta_i\br_i^*}{(\bu_1^*)_i+\left[\bN_1\bW\bu_1^*\right]_i+\delta_i}.
\end{align*}
We let 
\begin{align*}
    \gamma_i = -\frac{\left[\bN_1\bW\bu_1^*\right]_i+\delta_i}{(\bu_1^*)_i+\left[\bN_1\bW\bu_1^*\right]_i+\delta_i}\quad \text{and} \quad [\boldsymbol{\Psi}_{\br}]_{i,\cdot}^\top = \frac{\left[\boldsymbol{\Psi}_{\oU}\right]_{i,\cdot}^\top-\delta_i\br_i^*}{(\bu_1^*)_i+\left[\bN_1\bW\bu_1^*\right]_i+\delta_i}.
\end{align*}
Then the expansion can be written as 
\begin{align*}
    \bR^\top\wh\br_i-\br_i^* = \frac{1+\gamma_i}{(\bu_1^*)_i}\left(\bw_i-\left[\bN_1\bW\bu_1^*\right]_i\br_i^*\right)+[\boldsymbol{\Psi}_{\br}]_{i,\cdot}^\top.
\end{align*}
On the other hand, by Lemma \ref{u1differenceinfty} we know that
\begin{align*}
    \max_{1\leq i\leq n}\left|\left[\bN_1\bW\bu_1^*\right]_i+\delta_i\right| = \left\|\wh\bu_1-\bu_1^*\right\|_\infty&\lesssim \frac{K\log^{0.5} n+K^{1.5}\sqrt{\mu^*}}{n\theta_{\text{max}}}+\frac{K\sqrt{\mu^*}\log n}{n^{1.5}\theta_{\text{max}}^2}+\frac{K^3}{n^{1.5}\theta_{\text{max}}^3}\\
    &\ll  \min_{1\leq i\leq n}(\bu_1^*)_i\asymp \frac{1}{\sqrt{n}}.
\end{align*}
Therefore, $\gamma_i$ and $[\boldsymbol{\Psi}_{\br}]_{i,\cdot}^\top$ can be controlled as
\begin{align*}
    &|\gamma_i|\lesssim  \frac{\left\|\wh\bu_1-\bu_1^*\right\|_\infty}{(\bu_1^*)_i}\lesssim  \frac{K\log^{0.5} n+K^{1.5}\sqrt{\mu^*}}{n^{0.5}\theta_{\text{max}}}+\frac{K\sqrt{\mu^*}\log n}{n\theta_{\text{max}}^2}+\frac{K^3}{n\theta_{\text{max}}^3}, \\
    &\left\|[\boldsymbol{\Psi}_{\br}]_{i,\cdot}\right\|_2 \lesssim \frac{\left\|\left[\boldsymbol{\Psi}_{\oU}\right]_{i,\cdot}\right\|_2+|\delta_i|\left\|\br_i^*\right\|_2}{(\bu_1^*)_i} \lesssim \sqrt{n}\left(\left\|\boldsymbol{\Psi}_{\oU}\right\|_{2,\infty}+\left\|\br_i^*\right\|_2\left\|\boldsymbol{\delta}\right\|_{\infty}\right).
\end{align*}
\end{proof}

\subsection{Proof of Theorem \ref{estimationerr}}\label{estimationerrproof}
\begin{proof}
By Theorem \ref{mainthmrexpansion}, with probability at least $1-O(n^{-10})$, for all $i\in [n]$ we have
\begin{align}
    \left\|\bR^\top\wh\br_i-\br_i^*\right\|_2 &\lesssim  \frac{1}{(\bu_1^*)_i}\left(\left\|\bw_i\right\|_2+\frac{1}{\lambda_1^*}\left[\bN\bW\bu_1^*\right]_i\left\|\br_i^*\right\|_2\right)+\left\|[\boldsymbol{\Psi}_{\br}]_{i,\cdot}\right\|_2  \nonumber  \\
    &\lesssim \frac{1}{(\bu_1^*)_i}\left(\left\|\bw_i\right\|_2+\left\|\boldsymbol{\Psi}_{\oU}\right\|_{2,\infty} +\frac{1}{\lambda_1^*}\left|\left[\bN\bW\bu_1^*\right]_i\right|\left\|\br_i^*\right\|_2+\left\|\boldsymbol{\delta}\right\|_\infty\left\|\br_i^*\right\|_2\right).\label{estimationerreq0}
\end{align}

On one hand, we know that 
\begin{align}
    \left\|\bw_i\right\|_2&\leq \left\|\bW_{i, \cdot}\oU^*\left(\oLambda^*\right)^{-1}\right\|_2+\left\|\left[\bu_1^*\bu_1^{*\top}\bW\bN\right]_{i,\cdot}\oU^*\left(\oLambda^*\right)^{-1}\right\|_2 \nonumber \\
    &= \left\|\bW_{i, \cdot}\oU^*\left(\oLambda^*\right)^{-1}\right\|_2+(\bu_1^*)_i\left\|\bu_1^{*\top}\bW\bN\oU^*\left(\oLambda^*\right)^{-1}\right\|_2. \label{estimationerreq1}
\end{align}
By Lemma \ref{Wconcentration2} we know that
\begin{align}
    \left\|\bW_{i, \cdot}\oU^*\left(\oLambda^*\right)^{-1}\right\|_2&\lesssim \sqrt{\log n}\theta_{\text{max}}\left\|\oU^*\left(\oLambda^*\right)^{-1}\right\|_F +\log n \left\|\oU^*\left(\oLambda^*\right)^{-1}\right\|_{2,\infty} \nonumber\\
    &\lesssim \frac{\sqrt{(K-1)\log n}\theta_{\text{max}}}{\minsigma^*}+\frac{\log n}{\minsigma^*}\sqrt{\frac{(K-1)\mu^*}{n}}\label{estimationerreq2}
\end{align}
with probability at least $1-O(n^{-14})$. And, by Lemma \ref{Wconcentration6} and Lemma \ref{lemmaN1N2} we have
\begin{align}
    &\left\|\bu_1^{*\top}\bW\bN\oU^*\left(\oLambda^*\right)^{-1}\right\|_2 = \sqrt{\sum_{i=2}^{K}\left(\frac{\bu_1^{*\top}\bW\bN\bu_i^*}{\lambda_i^*}\right)^2}\leq\frac{1}{\minsigma^*}\sqrt{\sum_{i=2}^{K}\left(\bu_1^{*\top}\bW\bN\bu_i^*\right)^2} \nonumber\\
    \lesssim & \frac{\sqrt{K-1}}{\minsigma^*}\left( \sqrt{\log n} \theta_{\text{max}}+\log n \sqrt{\mu^*/n}\right)\left\|\bN\right\| \lesssim \frac{\sqrt{(K-1)\log n}\theta_{\text{max}}}{\minsigma^*}+\frac{\log n}{\minsigma^*}\sqrt{\frac{(K-1)\mu^*}{n}} \label{estimationerreq3}
\end{align}
with probability at least $1-O(n^{-14})$. Plugging \eqref{estimationerreq2} and \eqref{estimationerreq3} in \eqref{estimationerreq1} we get 
\begin{align*}
    \left\|\bw_i\right\|_2&\lesssim (1+(\bu_1^*)_i)\left(\frac{\sqrt{(K-1)\log n}\theta_{\text{max}}}{\minsigma^*}+\frac{\log n}{\minsigma^*}\sqrt{\frac{(K-1)\mu^*}{n}}\right)\\
    &\lesssim \frac{\sqrt{(K-1)\log n}\theta_{\text{max}}}{\minsigma^*}+\frac{\log n}{\minsigma^*}\sqrt{\frac{(K-1)\mu^*}{n}}
\end{align*}
with probability at least $1-O(n^{-14})$. Combine this with Theorem \ref{mainthmmatrixdenoising}, we get
\begin{align}
\left\|\bw_i\right\|_2+\left\|\boldsymbol{\Psi}_{\oU}\right\|_{2,\infty}\lesssim \frac{K^{1.5}\log^{0.5}n}{n\beta_n \theta_{\text{max}}}+\frac{K^{1.5}\log n\sqrt{\mu^*}}{n^{1.5}\beta_n\theta_{\text{max}}^2}\label{estimationerreq4}
\end{align}
for all $i\in [n]$ with probability at least $1-O(n^{-10})$. On the other hand, by Corollary \ref{corN1N2} and Lemma \ref{Wconcentration3} we know that
\begin{align*}
    \frac{1}{\lambda_1^*}\left|\left[\bN\bW\bu_1^*\right]_i\right|&\leq \left\|\bN_1\bW\bu_1^*\right\|_\infty\lesssim \frac{1}{\lambda_1^*}\left\|\bW\bu_1^*\right\|_\infty +\sqrt{\frac{(K-1)\mu^*}{n\lambda_1^{*2}}}\left\|\bW\bu_1^*\right\|_2 \\
    &\lesssim \frac{1}{\lambda_1^*}\left(\sqrt{\log n} \theta_{\text{max}}+\log n\sqrt{\mu^*/n}\right)+\sqrt{\frac{(K-1)\mu^*}{n\lambda_1^{*2}}}\sqrt{n}\theta_{\text{max}} \\
    &= \frac{\left(\sqrt{(K-1)\mu^*}+\sqrt{\log n}\right)\theta_{\text{max}}+\log n\sqrt{\mu^*/n}}{\lambda_1^*}
\end{align*}
for all $i\in [n]$ with probability at least $1-O(n^{-13})$. Combine this with Theorem \ref{mainthmu1}, we get 
\begin{align}
    \frac{1}{\lambda_1^*}\left|\left[\bN\bW\bu_1^*\right]_i\right|\left\|\br_i^*\right\|_2+\left\|\boldsymbol{\delta}\right\|_\infty\left\|\br_i^*\right\|_2\lesssim \frac{K\sqrt{\mu^*}+K^{0.5}\log^{0.5}n}{n\theta_{\text{max}}}+\frac{K^{0.5}\log n\sqrt{\mu^*}}{n\theta_{\text{max}}^2}\label{estimationerreq5}
\end{align}
for all $i\in [n]$ with probability at least $1-O(n^{-10})$. 

Plugging \eqref{estimationerreq4} and \eqref{estimationerreq5} in \eqref{estimationerreq0} we get
\begin{align*}
    \left\|\bR^\top\wh\br_i-\br_i^*\right\|_2&\lesssim \frac{1}{(\bu_1^*)_i}\left(\frac{K\sqrt{\mu^*}}{n\theta_{\text{max}}}+\frac{K^{0.5}\log n\sqrt{\mu^*}}{n\theta_{\text{max}}^2}+\frac{K^{1.5}\log^{0.5}n}{n\beta_n \theta_{\text{max}}}+\frac{K^{1.5}\log n\sqrt{\mu^*}}{n^{1.5}\beta_n\theta_{\text{max}}^2}\right)  \\
    &\lesssim \frac{K\sqrt{\mu^*}}{n^{0.5}\theta_{\text{max}}}+\frac{K^{0.5}\log n\sqrt{\mu^*}}{n^{0.5}\theta_{\text{max}}^2}+\frac{K^{1.5}\log^{0.5}n}{n^{0.5}\beta_n \theta_{\text{max}}}+\frac{K^{1.5}\log n\sqrt{\mu^*}}{n\beta_n\theta_{\text{max}}^2}
\end{align*}
for all $i\in [n]$ with probability at least $1-O(n^{-10})$. 
\end{proof}

\subsection{Proof of Theorem \ref{SPmainthm}}\label{SPmainthmproof}
\begin{proof}
Let $\mathcal{K} = \{i_1,i_2,\dots, i_K\}$. For the same reason as the proof of \cite[Lemma E.1]{jin2017estimating}, we know that
\begin{align*}
    \max_{k\in [K]}\min_{i\in \mathcal{K}}\left\|\wh\br_i-\bb_k^*\right\|_2 &\lesssim \max_{i\in [n]}\left\|\bR^\top\wh\br_i-\br_i^*\right\|_2 \\
    &\lesssim \frac{K\sqrt{\mu^*}}{n^{0.5}\theta_{\text{max}}}+\frac{K^{0.5}\log n\sqrt{\mu^*}}{n^{0.5}\theta_{\text{max}}^2}+\frac{K^{1.5}\log^{0.5}n}{n^{0.5}\beta_n \theta_{\text{max}}}+\frac{K^{1.5}\log n\sqrt{\mu^*}}{n\beta_n\theta_{\text{max}}^2}
\end{align*}
with probability at least $1-O(n^{-10})$. And we denote by $\mathcal{A}_2$ this event. We let $\rho(k) = \argmin_{l\in [K]}\left\|\wh\br_{i_l}-\bb_k^*\right\|_2$ for $k\in [K]$. By triangle inequality, under event $\mathcal{A}_2$ we know that
\begin{align}
    \left\|\br_{i_{\rho(k)}}^*-\bb_k^*\right\|_2&\leq \left\|\br_{i_{\rho(k)}}^*-\wh\br_{i_{\rho(k)}}\right\|_2+\left\|\wh\br_{i_{\rho(k)}}-\bb_k^*\right\|_2 \nonumber \\
    &\lesssim \frac{K\sqrt{\mu^*}}{n^{0.5}\theta_{\text{max}}}+\frac{K^{0.5}\log n\sqrt{\mu^*}}{n^{0.5}\theta_{\text{max}}^2}+\frac{K^{1.5}\log^{0.5}n}{n^{0.5}\beta_n \theta_{\text{max}}}+\frac{K^{1.5}\log n\sqrt{\mu^*}}{n\beta_n\theta_{\text{max}}^2}. \label{SPeq1}
\end{align}
If $i_{\rho(k)}\notin \VV_k$, we know that $\|\br_{i_{\rho(k)}}^*-\bb_k^*\|_2\geq \min_{l\in [K]}\min_{i\in [n]\backslash \VV_l}\left\|\br_i^*-\bb_l^*\right\|_2 = \Delta_{\br}$. In this case, \eqref{SPeq1} cannot hold for appropriately chosen $C_{\text{SP}}$. Therefore, for appropriately chosen $C_{\text{SP}}$, we must have $i_{\rho(k)}\in \VV_k$. This also implies that $\rho$ is a permutation of $[K]$, since the cardinality of $\mathcal{K}$ is exactly $K$. 

For any $k\in [K]$ and $j\in [n]$, if $j\in \VV_k$, by triangle inequality we have
\begin{align*}
    \left\|\wh\br_j-\wh\br_{i_{\rho(k)}}\right\|_2 &=\left\|\bR^\top\left(\wh\br_j-\wh\br_{i_{\rho(k)}}\right)\right\|_2\leq \left\|\bR^\top\wh\br_j-\br_j^*\right\|_2 +\left\|\br_j^* - \br_{i_{\rho(k)}}^*\right\|_2 + \left\| \br_{i_{\rho(k)}}^* - \bR^\top\wh\br_{i_{\rho(k)}}\right\|_2 \\
    &\leq \left\|\bb_k^*-\bb_k^*\right\|_2+2\max_{i\in [n]}\left\|\bR^\top\wh\br_i-\br_i^*\right\|_2  \\
    &\lesssim \frac{K\sqrt{\mu^*}}{n^{0.5}\theta_{\text{max}}}+\frac{K^{0.5}\log n\sqrt{\mu^*}}{n^{0.5}\theta_{\text{max}}^2}+\frac{K^{1.5}\log^{0.5}n}{n^{0.5}\beta_n \theta_{\text{max}}}+\frac{K^{1.5}\log n\sqrt{\mu^*}}{n\beta_n\theta_{\text{max}}^2}
\end{align*}
under event $\mathcal{A}_2$. As a result, for for appropriately chosen $C_{\text{SP}}$, it holds that $\|\wh\br_j-\wh\bb_{\rho(k)}'\|_2 = \|\wh\br_j-\wh\br_{i_{\rho(k)}}\|_2\leq \phi$. In other words, we must have $j\in \wh\VV_{\rho(k)}$. On the other hand, if $j\notin \VV_k$, again by triangle inequality we have
\begin{align*}
    \left\|\br_j^* - \bb_{k}^*\right\|_2 = \left\|\br_j^* - \br_{i_{\rho(k)}}^*\right\|_2&\leq \left\|\br_j^*-\bR^\top\wh\br_j\right\|_2 + \left\|\bR^\top\left(\wh\br_j-\wh\br_{i_{\rho(k)}}\right)\right\|_2 + \left\| \bR^\top\wh\br_{i_{\rho(k)}}-\br_{i_{\rho(k)}}^*\right\|_2 \\
    & = \left\|\br_j^*-\bR^\top\wh\br_j\right\|_2 + \left\|\wh\br_j-\wh\br_{i_{\rho(k)}}\right\|_2 + \left\| \bR^\top\wh\br_{i_{\rho(k)}}-\br_{i_{\rho(k)}}^*\right\|_2.
\end{align*}
As a result, $\|\wh\br_j-\wh\br_{i_{\rho(k)}}\|_2$ can be lower bounded as
\begin{align*}
    \left\|\wh\br_j-\wh\br_{i_{\rho(k)}}\right\|_2 &\geq \left\|\br_j^* - \bb_{k}^*\right\|_2 - \left\|\br_j^*-\bR^\top\wh\br_j\right\|_2-  \left\| \bR^\top\wh\br_{i_{\rho(k)}}-\br_{i_{\rho(k)}}^*\right\|_2 \\
    &\geq \min_{l\in [K]}\min_{i\in [n]\backslash \VV_l}\left\|\br_i^*-\bb_l^*\right\|_2 - 2 \max_{i\in [n]}\left\|\bR^\top\wh\br_i-\br_i^*\right\|_2  \\
    &= \Delta_{\br} - 2 \max_{i\in [n]}\left\|\bR^\top\wh\br_i-\br_i^*\right\|_2 > \phi
\end{align*}
as long as $C_{\text{SP}}$ satisfies 
\begin{align*}
    \max_{i\in [n]}\left\|\bR^\top\wh\br_i-\br_i^*\right\|_2\leq \frac{C_{\text{SP}}}{4}\varepsilon_1
\end{align*}
under event $\mathcal{A}_2$, we have $\|\wh\br_j-\wh\br_{i_{\rho(k)}}\|_2 > \phi$. This implies $j \notin \wh\VV_{\rho(k)}$. To sum up, we have $\wh\VV_{\rho(k)} = \VV_k$ for all $k\in [K]$ under event $\mathcal{A}_2$.
\end{proof}

\subsection{Proof of Corollary \ref{vertexexpansion}}\label{vertexexpansionproof}
\begin{proof}
    Let $\rho(\cdot)$ be the permutation from Theorem \ref{SPmainthm}. According to Theorem \ref{SPmainthm}, with probability at least $1-O(n^{-10})$, we have $\wh\VV_{\rho(k)} = \VV_{k}$ for all $k\in[K]$. Combine this fact with Algorithm \ref{alg1} one can see that
    \begin{align*}
        \wh\bb_{\rho(k)} = \frac{1}{|\wh\VV_{\rho(k)}|}\sum_{i\in \wh\VV_{\rho(k)}}\wh\br_i = \frac{1}{|\VV_k|}\sum_{i\in \VV_{k}}\wh\br_i.
    \end{align*}
    As a result, we can write
    \begin{align*}
         \bR^\top\wh\bb_{\rho(k)}-\bb_k^* = \frac{1}{|\VV_k|}\sum_{i\in \VV_{k}}\left(\wh\br_i - \bb_k^*\right) = \frac{1}{|\VV_k|}\sum_{i\in \VV_{k}}\left(\wh\br_i - \br_i^*\right).
    \end{align*}
    Then by \eqref{eq:approxerror} we know that
    \begin{align*}
        \left\|[\boldsymbol{\Psi}_{\bb}]_{k,\cdot}\right\|_2 &= \left\|\frac{1}{|\VV_k|}\sum_{i\in \VV_{k}}\left(\wh\br_i - \br_i^*\right) - \frac{1}{\left|\VV_k\right|}\sum_{i\in \VV_k}\Delta\br_i\right\|_2 \\
        &\leq \frac{1}{|\VV_k|}\left\|\sum_{i\in \VV_{k}}\left(\wh\br_i - \br_i^*-\Delta\br_i\right)\right\|_2 \lesssim \varepsilon_2
    \end{align*}
    with probability at least $1-O(n^{-10})$.
\end{proof}

\subsection{Proof of Corollary \ref{membershipreconstructionlem1}} \label{membershipreconstructionlem1proof}
\begin{proof}
Considering the trace on both sides of \eqref{eqfirstsubspaceexpansion} we get
\begin{align*}
    \wh\lambda_1 - \lambda_1^* &= \textbf{Tr}\left[\wh\lambda_1\wh\bu_1\wh\bu_1^\top - \lambda_1^*\bu_1^*\bu_1^{*\top}\right] \\
    &=\textbf{Tr}\left[\bu_1^*\bu_1^{*\top}\bW\bu_1^*\bu_1^{*\top}+ \bN\bW\bu_1^*\bu_1^{*\top}+\left(\bW\bu_1^*\bu_1^{*\top}\right)^\top\bN + \bDelta\right] \\
    &= \textbf{Tr}\left[\bW\bu_1^*\bu_1^{*\top}\bu_1^*\bu_1^{*\top}+ \bN\bW\bu_1^*\bu_1^{*\top}+\bN^\top\bW\bu_1^*\bu_1^{*\top}\right] + \textbf{Tr}\left[\bDelta\right]  \\
    &=\textbf{Tr}\left[\bW\bu_1^*\bu_1^{*\top}+ 2\bN\bW\bu_1^*\bu_1^{*\top}\right]+ \textbf{Tr}\left[\bDelta\right].
\end{align*}
And, from \eqref{eqfirstsubspaceexpansion} we also know that 
\begin{align*}
    \bDelta &= \wh\lambda_1\wh\bu_1\wh\bu_1^\top - \lambda_1^*\bu_1^*\bu_1^{*\top} - \bu_1^*\bu_1^{*\top}\bW\bu_1^*\bu_1^{*\top}- \bN\bW\bu_1^*\bu_1^{*\top}-\left(\bW\bu_1^*\bu_1^{*\top}\right)^\top\bN  \\
    & = \wh\lambda_1\wh\bu_1\wh\bu_1^\top - \bu_1^*\left(\lambda_1^*\bu_1^{*\top} + \bu_1^{*\top}\bW\bu_1^*\bu_1^{*\top}+ \bu_1^{*\top}\bW\bN\right)-\bN\bW\bu_1^*\bu_1^{*\top}.
\end{align*}
As a result, the rank of $\bDelta$ is at most $3$. Therefore, the trace of $\bDelta$ can be bounded as  $\left|\textbf{Tr}\left[\bDelta\right]\right|\leq \textbf{Rank}\left[\bDelta\right]\left\|\bDelta\right\|\lesssim \frac{n\theta_{\text{max}}^2}{\lambda_1^*}$, completing the proof.
\end{proof}

\subsection{Proof of Lemma \ref{membershipreconstructionlem2}} \label{membershipreconstructionlem2proof}
\begin{proof}
By Corollary \ref{vertexexpansion} we know that
\begin{align}
    \wh\bb_{\rho(k)}^\top\oLambda\wh\bb_{\rho(k)} =& \left(\bb_k^* +\Delta\bb_k + [\boldsymbol{\Psi}_{\bb}]_{k,\cdot}\right)^\top\bR^\top\oLambda \bR\left(\bb_k^* +\Delta\bb_k + [\boldsymbol{\Psi}_{\bb}]_{k,\cdot}\right)  \nonumber \\
     =& \left(\bb_k^* +\Delta\bb_k + [\boldsymbol{\Psi}_{\bb}]_{k,\cdot}\right)^\top\oLambda^*\left(\bb_k^* +\Delta\bb_k + [\boldsymbol{\Psi}_{\bb}]_{k,\cdot}\right) \nonumber \\
     &+\left(\bb_k^* +\Delta\bb_k + [\boldsymbol{\Psi}_{\bb}]_{k,\cdot}\right)^\top\left(\bR^\top\oLambda \bR-\oLambda^*\right)\left(\bb_k^* +\Delta\bb_k + [\boldsymbol{\Psi}_{\bb}]_{k,\cdot}\right). \label{quadraticeq1}
\end{align}
On one hand, by Lemma \ref{lemma2} we have
\begin{align}
    &\left|\left(\bb_k^* +\Delta\bb_k + [\boldsymbol{\Psi}_{\bb}]_{k,\cdot}\right)^\top\left(\bR^\top\oLambda \bR-\oLambda^*\right)\left(\bb_k^* +\Delta\bb_k + [\boldsymbol{\Psi}_{\bb}]_{k,\cdot}\right)\right| \nonumber \\
    \leq & \left\|\bR^\top\oLambda \bR-\oLambda^*\right\|\left\|\bb_k^* +\Delta\bb_k + [\boldsymbol{\Psi}_{\bb}]_{k,\cdot}\right\|^2_2 \nonumber \\
    \lesssim & \left(\frac{\kappa^*n\theta_{\text{max}}^2}{\minsigma^{*}} +\sqrt{(K-1)\log n}\theta_{\text{max}}\right)\left(\sqrt{K-1}+\varepsilon_1\right)^2\lesssim \frac{\kappa^*K^2}{\beta_n}+K^{1.5}\theta_{\text{max}}\log^{0.5}n\label{quadraticeq2}
\end{align}
with probability exceeding $1-O(n^{-10})$. On the other hand, by triangle inequality we know that
\begin{align}
    &\left|\left(\bb_k^* +\Delta\bb_k + [\boldsymbol{\Psi}_{\bb}]_{k,\cdot}\right)^\top\oLambda^*\left(\bb_k^* +\Delta\bb_k + [\boldsymbol{\Psi}_{\bb}]_{k,\cdot}\right) -\bb_k^{*\top}\oLambda^*\bb_k^* -2\bb_k^{*\top}\oLambda^*\Delta\bb_k\right| \nonumber \\
    \leq &\left|[\boldsymbol{\Psi}_{\bb}]_{k,\cdot}^\top\oLambda^*\left(\bb_k^* +\Delta\bb_k + [\boldsymbol{\Psi}_{\bb}]_{k,\cdot}\right)\right| +\left|\left(\bb_k^* +\Delta\bb_k\right)^\top\oLambda^*[\boldsymbol{\Psi}_{\bb}]_{k,\cdot}\right| +\left|\Delta\bb_k^{\top}\oLambda^*\Delta\bb_k\right| \nonumber \\
    \lesssim & \left(\sqrt{K-1} \varepsilon_2+\varepsilon_1^2\right)\maxsigma^*. \label{quadraticeq3}
\end{align}
Plugging \eqref{quadraticeq2} and \eqref{quadraticeq3} in \eqref{quadraticeq1} we get
\begin{align*}
    &\left|\wh\bb_{\rho(k)}^\top\oLambda\wh\bb_{\rho(k)} - \bb_k^{*\top}\oLambda^*\bb_k^* -2\bb_k^{*\top}\oLambda^*\Delta\bb_k\right|  \\
    \lesssim & \frac{\kappa^*K^2}{\beta_n}+K^{1.5}\mu^* \theta_{\text{max}}\log^{0.5}n + \left(K^{0.5} \varepsilon_2+\varepsilon_1^2\right)\maxsigma^*
\end{align*}
with probability at least $1-O(n^{-10})$. Moreover, the left hand side is a small order term compared to $K^{0.5}\varepsilon_1\maxsigma^*$, which controls $\bb_k^{*\top}\oLambda^*\Delta\bb_k$. As a result, the estimation error can be controlled by $K^{0.5}\varepsilon_1\maxsigma^*$,
\end{proof}

\subsection{Proof of Lemma \ref{membershipreconstructionlem3}} \label{membershipreconstructionlem3proof}
\begin{proof}
We denote by
\begin{align*}
\boldsymbol{\wh B} = \begin{bmatrix}
\wh\bb_1 & \wh\bb_2 & \dots & \wh\bb_K\\
1 & 1 & \dots & 1
\end{bmatrix}\in \mathbb{R}^{K\times K},
\end{align*}
By the definition of $\wh\ba_i$, we know that $\boldsymbol{\wh B}\wh\ba_i = [\wh\br_i^\top,1]^\top$. We also denote by
\begin{align*}
    \widetilde\bR = \begin{bmatrix}
\bR & \boldsymbol{0}_{(K-1)\times 1}\\
\boldsymbol{0}_{1\times (K-1)} & 1 
\end{bmatrix}\in \mathbb{R}^{K\times K}
\end{align*}
Let $\rho(\cdot)$ be the permutation from Theorem \ref{SPmainthm}, then we have
\begin{align*}
    \widetilde{\bR}^\top\begin{bmatrix}
    \wh\br_i\\
    1 
    \end{bmatrix}
    =\widetilde{\bR}^\top \boldsymbol{\wh B}\wh\ba_i  = \widetilde{\bR}^\top\sum_{j=1}^K \begin{bmatrix}
    \wh\bb_j\\
    1 
    \end{bmatrix}(\wh\ba_i)_j 
    = \widetilde{\bR}^\top\sum_{j=1}^K \begin{bmatrix}
    \wh\bb_{\rho(j)}\\
    1 
    \end{bmatrix}(\wh\ba_i)_{\rho(j)}
    = \widetilde{\bR}^\top \rho(\boldsymbol{\wh B})\rho(\wh\ba_i).
\end{align*}
Therefore, we can write
\begin{align}
    \bB^*\left(\rho(\wh\ba_i)-\ba_i^*\right) &= \bB^*\rho(\wh\ba_i) - \begin{bmatrix}
    \br_i^*\\
    1 
    \end{bmatrix}
    = \widetilde{\bR}^\top 
    \begin{bmatrix}
    \wh\br_i\\
    1 
    \end{bmatrix}
-\begin{bmatrix}
    \br_i^*\\
    1 
    \end{bmatrix} - \left(\widetilde{\bR}^\top\rho(\boldsymbol{\wh B})-\boldsymbol{B}^*\right) \rho(\wh\ba_i) \nonumber \\
    &= \widetilde{\bR}^\top 
    \begin{bmatrix}
    \wh\br_i\\
    1 
    \end{bmatrix}
-\begin{bmatrix}
    \br_i^*\\
    1 
    \end{bmatrix} - \left(\widetilde{\bR}^\top\rho(\boldsymbol{\wh B})-\boldsymbol{B}^*\right) \ba^*_i - \left(\widetilde{\bR}^\top\rho(\boldsymbol{\wh B})-\boldsymbol{B}^*\right) \left(\rho(\wh\ba_i)-\ba^*_i\right).\label{aexpansioneq1}
\end{align}
Denote by $\boldsymbol{\Psi}_{\bB} = [\boldsymbol{\Psi}_{\bb}, \boldsymbol{0}_{K\times 1}]^\top$. By Corollary \ref{vertexexpansion} we know that
\begin{align}
    \widetilde{\bR}^\top\rho(\boldsymbol{\wh B})-\boldsymbol{B}^* = \Delta\bB+\boldsymbol{\Psi}_{\bB}, \label{aexpansioneq2}
\end{align}
and 
\begin{align*}
    &\left\|\widetilde{\bR}^\top\rho(\boldsymbol{\wh B})-\boldsymbol{B}^*\right\|\leq \left\|\widetilde{\bR}^\top\rho(\boldsymbol{\wh B})-\boldsymbol{B}^*\right\|_F \leq \sqrt{K}\sup_{i\in [n]}\left\|\bR^\top\wh\br_i-\br_i^*\right\|_2\lesssim \sqrt{K}\varepsilon_1; \\
    &\left\|\boldsymbol{\Psi}_{\bB}\right\|\leq \sqrt{K}\left\|\boldsymbol{\Psi}_{\bb}\right\|_{2,\infty} \lesssim \sqrt{K}\varepsilon_2.
\end{align*}
Next, by Theorem \ref{mainthmrexpansion},
\begin{align}
    \widetilde{\bR}^\top 
    \begin{bmatrix}
    \wh\br_i\\
    1 
    \end{bmatrix}
-\begin{bmatrix}
    \br_i^*\\
    1 
    \end{bmatrix} = \begin{bmatrix}
    \Delta\br_i\\
    0 
    \end{bmatrix} + \begin{bmatrix}
    [\boldsymbol{\Psi}_{\br}]_{i,\cdot}^\top\\
    0 
    \end{bmatrix}, \label{aexpansioneq3}
\end{align}
where $\|[\boldsymbol{\Psi}_{\br}]_{i,\cdot}\|\lesssim\varepsilon_2$. Plugging \eqref{aexpansioneq2} and \eqref{aexpansioneq3} in \eqref{aexpansioneq1}, we get
\begin{align}
    \bB^*\left(\rho(\wh\ba_i)-\ba_i^*\right) = \begin{bmatrix}
    \Delta\br_i\\
    0 
    \end{bmatrix} - \Delta\bB\ba_i^* +\begin{bmatrix}
    [\boldsymbol{\Psi}_{\br}]_{i,\cdot}^\top\\
    0 
    \end{bmatrix} - \boldsymbol{\Psi}_{\bB}\ba_i^* - \left(\widetilde{\bR}^\top\rho(\boldsymbol{\wh B})-\boldsymbol{B}^*\right) \left(\rho(\wh\ba_i)-\ba^*_i\right).\label{aexpansioneq4}
\end{align}
Since $\bb_1^*,\bb_2^*,\dots, \bb_K^*$ form a simplex and all $\br_i^*,i\in [n]$ are inside it, we know that the entries of $\ba_i$ are within $[0,1]$. As a result, we know that $\|\ba_1^*\|_2^2 = \sum_{j=1}^K(\ba_i^*)_j^2\leq \sum_{j=1}^K(\ba_i^*)_j = 1$. That is to say, we have
\begin{align}
    \left\|\begin{bmatrix}
    [\boldsymbol{\Psi}_{\br}]_{i,\cdot}^\top\\
    0 
    \end{bmatrix} - \boldsymbol{\Psi}_{\bB}\ba_i^*\right\|_2\leq \|[\Psi_{\br}]_{i,\cdot}\|+ \left\|\boldsymbol{\Psi}_{\bB}\right\|\left\|\ba_i^*\right\|_2\lesssim \sqrt{K}\varepsilon_2. \label{aexpansioneq5}
\end{align}
On the other hand, we have 
\begin{align}
    \left\|\left(\widetilde{\bR}^\top\rho(\boldsymbol{\wh B})-\boldsymbol{B}^*\right) \left(\rho(\wh\ba_i)-\ba^*_i\right)\right\|_2\lesssim \sqrt{K}\varepsilon_1 \left\|\rho(\wh\ba_i)-\ba^*_i\right\|_2. \label{aexpansioneq6}
\end{align}
According to \cite[(C.26)]{jin2017estimating}, we know that $\|(\bB^*)^{-1}\|\lesssim 1/\sqrt{K}$. As a result, plugging \eqref{aexpansioneq5} and \eqref{aexpansioneq6} in \eqref{aexpansioneq4} we get 
\begin{align}
    \left\|\rho(\wh\ba_i)-\ba^*_i - (\bB^*)^{-1}\begin{bmatrix}
    \Delta\br_i\\
    0 
    \end{bmatrix}+(\bB^*)^{-1}\Delta\bB\ba_i^*\right\|_2\lesssim \varepsilon_2+\varepsilon_1\left\|\rho(\wh\ba_i)-\ba^*_i\right\|_2.\label{aexpansioneq7}
\end{align}
And, since 
\begin{align}
    \left\|(\bB^*)^{-1}\begin{bmatrix}
    \Delta\br_i\\
    0 
    \end{bmatrix}\right\|_2\leq \left\|(\bB^*)^{-1}\right\|\left\|\Delta\br_i\right\|_2\lesssim \frac{\varepsilon_1}{\sqrt{K}}\label{aexpansioneq8}
\end{align}
and
\begin{align}
    \left\|(\bB^*)^{-1}\Delta\bB^*\ba_i^*\right\|_2\lesssim \frac{\left\|\Delta\bB\right\|_F\left\|\ba^*_i\right\|_2}{\sqrt{K}}\leq \frac{\sqrt{K}\max_{k\in [K]}\left\|\Delta\bb_k\right\|_2}{\sqrt{K}}\lesssim \varepsilon_1.\label{aexpansioneq9}
\end{align}
Combine \eqref{aexpansioneq8}, \eqref{aexpansioneq9} with \eqref{aexpansioneq7} we get
\begin{align*}
    \left\|\rho(\wh\ba_i)-\ba^*_i\right\|_2\lesssim \varepsilon_1 +\varepsilon_2+\varepsilon_1\left\|\rho(\wh\ba_i)-\ba^*_i\right\|_2.
\end{align*}
Since $\varepsilon_2\lesssim \varepsilon_1\leq C$, we know that $\left\|\rho(\wh\ba_i)-\ba^*_i\right\|_2\lesssim \varepsilon_1$, completing the proof.
\end{proof}

\subsection{Proof of Theorem \ref{Piexpansion}}\label{Piexpansionproof}
\begin{proof}
First we focus on $\wh c_k, k\in [K]$. Define $f(x) := \sqrt{x}$ and $c_k^* := (\lambda_1^*+\bb_k^{*\top}\oLambda^*\bb_k^*)^{-1/2}$. Set $x_0 := \lambda_1^*+\bb_k^{*\top}\oLambda^*\bb_k^*$, $x_1 := \wh\lambda_1+\wh\bb_{\rho(k)}^{\top}\oLambda\wh\bb_{\rho(k)}$. By Taylor expansion, we know that there exists some $\tilde{x}$ between $x_0$ and $x_1$, such that
\begin{align*}
    \frac{1}{\wh c_{\rho(k)}}-\frac{1}{c_k^*} = f(x_1)-f(x_0)=f'(x_0)(x_1-x_0)+\frac{f''(\tilde{x})}{2}(x_1-x_0)^2.
\end{align*}
Combining Corollary \ref{membershipreconstructionlem1} and Lemma \ref{membershipreconstructionlem2}, we have
\begin{align*}
    &\left|x_1-x_0-\left(\textbf{Tr}\left[\bW\bu_1^*\bu_1^{*\top}+ 2\bN\bW\bu_1^*\bu_1^{*\top}\right]+2\bb_k^{*\top}\oLambda^*\Delta\bb_k\right)\right| \\
    \leq &\left|\textbf{Tr}\left[\boldsymbol\Delta\right]\right|+\left\|\boldsymbol{\psi}\right\|_\infty\lesssim \frac{\kappa^*K^2}{\beta_n}+K^{1.5} \theta_{\text{max}}\log^{0.5}n + \left(K^{0.5} \varepsilon_2+\varepsilon_1^2\right)\maxsigma^*.
\end{align*}
Note that $f'(x_0) = 0.5x_0^{-1/2} = 0.5 c_k^*$. Since $K^{0.5}\varepsilon_1\maxsigma^*\ll x_0\asymp n\theta_{\text{max}}^2$, we have $f''(\tilde{x})\asymp f''(x_0)\asymp c_k^{*3}$. Hence,
\begin{align}
    &\left|\frac{1}{\wh c_{\rho(k)}}-\frac{1}{c_k^*}-\frac{c_k^*}{2}\left(\textbf{Tr}\left[\bW\bu_1^*\bu_1^{*\top}+ 2\bN\bW\bu_1^*\bu_1^{*\top}\right]+2\bb_k^{*\top}\oLambda^*\Delta\bb_k\right)\right|  \nonumber \\
    \lesssim & \left(\frac{\kappa^*K^2}{\beta_n}+K^{1.5} \theta_{\text{max}}\log^{0.5}n + \left(K^{0.5} \varepsilon_2+\varepsilon_1^2\right)\maxsigma^*\right)c_k^*+K\varepsilon_1^2\maxsigma^{*2}c_k^{*3}. \label{cinverseexpansion}
\end{align}
Again by Taylor expansion, we know that there exists some $\tilde{x}'$ between $x_0$ and $x_1$, such that
\begin{align*}
    \frac{1}{\wh c_{\rho(k)}}-\frac{1}{c_k^*} = f(x_1)-f(x_0)=f'(\tilde{x}')(x_1-x_0).
\end{align*}
Since $\sqrt{K-1}\varepsilon_1\maxsigma^*\ll x_0\asymp n\theta_{\text{max}}^2$, we have $f'(\tilde{x}')\asymp f'(x_0)\asymp c_k^{*}$. So,
\begin{align}
    \left|\frac{1}{\wh c_{\rho(k)}}-\frac{1}{c_k^*}\right|\asymp c_k^{*}\left|x_1-x_0\right|\lesssim K^{0.5}\varepsilon_1\maxsigma^* c_k^*.\label{cinverseestimationeq}
\end{align}
Second, for the permutation $\rho(\cdot)$ from Theorem \ref{SPmainthm} and any $i\in [n]$, we have
\begin{align*}
    \frac{(\wh\ba_i)_{\rho(k)}}{\wh c_{\rho(k)}} - \frac{(\ba_i^*)_k}{c_k^*} &= \frac{(\wh\ba_i)_{\rho(k)} - (\ba_i^*)_k}{c_k^*} + \left(\frac{1}{\wh c_{\rho(k)}}-\frac{1}{c_k^*}\right)(\ba_i^*)_k +\left(\frac{1}{\wh c_{\rho(k)}}-\frac{1}{c_k^*}\right)\left((\wh\ba_i)_{\rho(k)} - (\ba_i^*)_k\right) \\
    &= \frac{(\rho(\wh\ba_i) - \ba_i^*)_k}{c_k^*} + \left(\frac{1}{\wh c_{\rho(k)}}-\frac{1}{c_k^*}\right)(\ba_i^*)_k +\left(\frac{1}{\wh c_{\rho(k)}}-\frac{1}{c_k^*}\right)(\rho(\wh\ba_i) - \ba_i^*)_k.
\end{align*}
Combine Lemma \ref{membershipreconstructionlem3} and \eqref{cinverseexpansion}, we have the following expansion 
\begin{align*}
    \frac{(\wh\ba_i)_{\rho(k)}}{\wh c_{\rho(k)}} - \frac{(\ba_i^*)_k}{c_k^*} = \frac{(\Delta\ba_i)_k}{c_k^*} +\frac{c_k^*}{2}\left(\textbf{Tr}\left[\bW\bu_1^*\bu_1^{*\top}+ 2\bN\bW\bu_1^*\bu_1^{*\top}\right]+2\bb_k^{*\top}\oLambda^*\Delta\bb_k\right)(\ba_i^*)_k +\left[\boldsymbol{\Psi}_{\ba/\bc}\right]_{i, k},
\end{align*}
where 
\begin{align*}
\left|\left[\boldsymbol{\Psi}_{\ba/\bc}\right]_{i, k}\right| \leq & \left|\frac{1}{\wh c_{\rho(k)}}-\frac{1}{c_k^*}-\frac{c_k^*}{2}\left(\textbf{Tr}\left[\bW\bu_1^*\bu_1^{*\top}+ 2\bN\bW\bu_1^*\bu_1^{*\top}\right]+2\bb_k^{*\top}\oLambda^*\Delta\bb_k\right)\right| 
(\ba_i^*)_k \nonumber \\
&+ \frac{\left|(\rho(\wh\ba_i) - \ba_i^*-\Delta\ba_i)_k\right|}{c_k^*} +\left|\frac{1}{\wh c_{\rho(k)}}-\frac{1}{c_k^*}\right| \left|(\rho(\wh\ba_i) - \ba_i^*)_k\right| \nonumber\\
\lesssim & \left(\frac{\kappa^*K^2}{\beta_n}+K^{1.5} \theta_{\text{max}}\log^{0.5}n + \left(K^{0.5} \varepsilon_2+\varepsilon_1^2\right)\maxsigma^*\right)c_k^*+K\varepsilon_1^2\maxsigma^{*2}c_k^{*3} \nonumber \\
&+ \frac{\varepsilon_2+\varepsilon_1^2}{c_k^*}+K^{0.5}\varepsilon_1^2\maxsigma^* c_k^*\\
\overset{(i)}{\lesssim} & \frac{1}{\left\|\btheta\right\|_2}\left(\frac{\kappa^*K^2}{\beta_n}+K^{1.5} \theta_{\text{max}}\log^{0.5}n + \left(K^{0.5} \varepsilon_2+\varepsilon_1^2\right)\maxsigma^*\right) + \left\|\btheta\right\|_2(\varepsilon_2+\varepsilon_1^2). 
\end{align*}
Here $(i)$ holds because $c_k^*\asymp 1/\left\|\btheta\right\|_2$ according to \cite[C.22]{jin2017estimating} and $\maxsigma^*\asymp \beta_n K^{-1}\left\|\btheta\right\|_2^2$ according to Lemma \ref{eigenvaluelemma}. In terms of the estimation error, by Theorem \ref{membershipreconstructionlem3} and \eqref{cinverseestimationeq} we have
\begin{align}
    \left|\frac{(\wh\ba_i)_{\rho(k)}}{\wh c_{\rho(k)}} - \frac{(\ba_i^*)_k}{c_k^*}\right| \leq \left|\frac{(\wh\ba_i)_{\rho(k)}}{\wh c_{\rho(k)}} - \frac{(\ba_i^*)_k}{\wh c_{\rho(k)}}\right| + \left|\frac{(\ba_i^*)_k}{\wh c_{\rho(k)}} - \frac{(\ba_i^*)_k}{c_k^*}\right|\lesssim \frac{\varepsilon_1}{c_k^*}+K^{0.5}\varepsilon_1\maxsigma^* c_k^*\lesssim \varepsilon_1\left\|\btheta\right\|_2. \label{acestimationeq}
\end{align}

Now we are ready to derive the expansion of $\wh\bpi_i$. For any $i\in [n]$ and $k\in [K]$, we have
\begin{align*}
    &\wh\bpi_i(\rho(k))-\bpi_i^*(k) = \frac{(\wh\ba_i)_{\rho(k)}/\wh c_{\rho(k)}}{\sum_{l=1}^K (\wh\ba_i)_{\rho(l)}\wh / c_{\rho(l)}} - \frac{(\ba_i^*)_{k}/ c^*_k}{\sum_{l=1}^K (\ba_i^*)_{l} / c^*_{l}}  \\
    =&\frac{\sum_{l\neq k,l\in [K]}\left((\wh\ba_i)_{\rho(k)}/\wh c_{\rho(k)}\right)\cdot\left((\ba_i^*)_{l} / c^*_{l}\right)-\left((\wh\ba_i)_{\rho(l)}/\wh c_{\rho(l)}\right)\cdot\left((\ba_i^*)_{k} / c^*_{k}\right)}{\left(\sum_{l=1}^K (\wh\ba_i)_{\rho(l)} / \wh c_{\rho(l)}\right)\left(\sum_{l=1}^K (\ba_i^*)_{l} / c^*_{l}\right)} \\
   =&(1+\eta_i)\Delta \bpi_i(k) +\left[\boldsymbol{\Psi}_{\bPi}\right]_{i, k},
\end{align*}
where
\begin{align*}
    \eta_i =& \frac{\sum_{l=1}^K (\ba_i^*)_{l} / c^*_{l} - (\wh\ba_i)_{\rho(l)} / \wh c_{\rho(l)}}{\sum_{l=1}^K (\wh\ba_i)_{\rho(l)} / \wh c_{\rho(l)}},  \\
    \Delta \bpi_i(k) =& \frac{1}{\left(\sum_{l=1}^K (\ba^*_i)_{l} /c^*_{l}\right)^2}\Bigg\{\sum_{l\neq k, l\in [K]} \textbf{Tr}\left[\bW\bu_1^*\bu_1^{*\top}+ 2\bN\bW\bu_1^*\bu_1^{*\top}\right]\left(\frac{c_k^*}{2 c_l^*} - \frac{c_l^*}{2 c_k^*}\right)(\ba_i^*)_k(\ba_i^*)_l \\
    &\quad  \quad +\frac{(\Delta\ba_i)_k(\ba_i^*)_l - (\Delta\ba_i)_l(\ba_i^*)_k}{c_k^*c_l^*}+\left(\frac{\bb_k^{*\top}\oLambda^*\Delta\bb_k c_k^*}{c_l^*} - \frac{\bb_l^{*\top}\oLambda^*\Delta\bb_l c_l^*}{c_k^*}\right)(\ba_i^*)_k(\ba_i^*)_l \Bigg\}, \\
    \left[\boldsymbol{\Psi}_{\bPi}\right]_{i, k} =& \frac{\sum_{l\neq k, l\in [K]}  \left[\boldsymbol{\Psi}_{\ba/\bc}\right]_{i, k}(\ba_i^*)_l/c_l^* - \left[\boldsymbol{\Psi}_{\ba/\bc}\right]_{i, l}(\ba_i^*)_k/c_k^*}{\left(\sum_{l=1}^K (\wh\ba_i)_{\rho(l)} / \wh c_{\rho(l)}\right)\left(\sum_{l=1}^K (\ba_i^*)_{l} / c^*_{l}\right)}.
\end{align*}
Since $\varepsilon_1\leq C$ for some appropriate $C>0$, we have
\begin{align*}
    \left|\frac{c_k^*}{\wh c_{\rho(k)}}-1\right|\lesssim K^{0.5}\varepsilon_1\maxsigma^* c_k^{*2}\lesssim K^{0.5}\varepsilon_1\frac{\maxsigma^*}{\left\|\btheta\right\|_2^2}\lesssim  \frac{\varepsilon_1}{K^{0.5}}\Rightarrow \left|\frac{c_k^*}{\wh c_{\rho(k)}}-1\right|\leq 0.5.
\end{align*}
As a result, we have
\begin{align*}
    \sum_{l=1}^K \frac{(\wh\ba_i)_{\rho(l)}}{\wh c_{\rho(l)}}\geq \sum_{l=1}^K \frac{(\wh\ba_i)_{\rho(l)}}{\max_{t\in [K]}\wh c_{t}}  = \frac{1}{\max_{t\in [K]}\wh c_{t}}\gtrsim \frac{1}{\max_{t\in [K]} c^*_{t}}\asymp  \left\|\btheta\right\|_2.
\end{align*}
Combine this with \eqref{acestimationeq}, for $\eta_i, i\in [n]$ we have
\begin{align*}
    |\eta_i|&\leq \frac{\sum_{l=1}^K \left|(\ba_i^*)_{l} / c^*_{l} - (\wh\ba_i)_{\rho(l)} / \wh c_{\rho(l)}\right|}{\sum_{l=1}^K (\wh\ba_i)_{\rho(l)} / \wh c_{\rho(l)}}\lesssim  \frac{K\varepsilon_1\left\|\btheta\right\|_2}{\left\|\btheta\right\|_2} = K\varepsilon_1.
\end{align*}
Since $\sum_{l=1}^K (\ba_i^*)_{l} / c^*_{l}\geq 1/\max_{t\in [K]}c_t^*\asymp 1/\left\|\btheta\right\|_2$, we have
\begin{align*}
    \left|\left[\boldsymbol{\Psi}_{\bPi}\right]_{i, k}\right|&\lesssim \frac{\sum_{l\neq k, l\in [K]}  \left|\left[\boldsymbol{\Psi}_{\ba/\bc}\right]_{i, k}\right|(\ba_i^*)_l\left\|\btheta\right\|_2 + \left|\left[\boldsymbol{\Psi}_{\ba/\bc}\right]_{i, l}\right|(\ba_i^*)_k\left\|\btheta\right\|_2}{\left\|\btheta\right\|_2^2}\\
    &\leq \frac{K\max_{j\in [n], l\in [K]}\left|\left[\boldsymbol{\Psi}_{\ba/\bc}\right]_{j, l}\right|}{\left\|\btheta\right\|_2} \\
    &\lesssim \frac{K}{\left\|\btheta\right\|_2^2}\left(\frac{\kappa^*K^2}{\beta_n}+K^{1.5} \theta_{\text{max}}\log^{0.5}n + \left(K^{0.5} \varepsilon_2+\varepsilon_1^2\right)\maxsigma^*\right) + K(\varepsilon_2+\varepsilon_1^2).
\end{align*}

\end{proof}

\subsection{Proof of Theorem \ref{distributionthm}}\label{distributionthmproof}
\begin{proof}
Define 
\begin{align*}
    \Delta\bpi_{\mathcal{I}} := \left(\Delta\bpi_{i_1}(k_1), \Delta\bpi_{i_2}(k_2),\dots, \Delta\bpi_{i_r}(k_r)\right)^\top= \sum_{1\leq i\leq j\leq n}W_{ij}\bomega_{ij}.
\end{align*}
By Berry-Esseen theorem \cite{raivc2019multivariate}, for any convex set $\mathcal{D}\subset \mathbb{R}^r$, we have
\begin{align}
    &\left|\mathbb{P}(\Delta\bpi_\mathcal{I} \in \mathcal{D}) - \mathbb{P}(\mathcal{N}(\boldsymbol{0}_r,\bSigma) \in \mathcal{D})\right|\lesssim r^{1/4}\sum_{1\leq i\leq j \leq n} \mathbb{E}\left[\left\|\bSigma^{-1/2}W_{ij}\bomega_{ij}\right\|_2^3\right] \nonumber \\
    \lesssim& r^{1/4}\sum_{1\leq i\leq j \leq n} \left\|\bSigma^{-1/2}\bomega_{ij}\right\|_2^3\mathbb{E}\left[\left|W_{ij}\right|^3\right] \leq r^{1/4}\sum_{1\leq i\leq j \leq n} \left\|\bSigma^{-1/2}\bomega_{ij}\right\|_2^3 H_{ij}(1-H_{ij}) \nonumber \\
    \lesssim& r^{1/4}\max_{1\leq i\leq j\leq n}\left\|\bSigma^{-1/2}\bomega_{ij}\right\|_2 \sum_{1\leq i\leq j \leq n} \left\|\bSigma^{-1/2}\bomega_{ij}\right\|_2^2 H_{ij}(1-H_{ij}).\label{distributionaleq1}
\end{align}
Since $\bSigma$ is the covariance of $\Delta\phi_\mathcal{I}$, we know that
\begin{align}
    &\sum_{1\leq i\leq j \leq n} \left\|\bSigma^{-1/2}\bomega_{ij}\right\|_2^2 H_{ij}(1-H_{ij})  = \sum_{1\leq i\leq j \leq n} \bomega_{ij}^\top \bSigma^{-1}\bomega_{ij} H_{ij}(1-H_{ij}) \nonumber \\
    =& \sum_{1\leq i\leq j \leq n} \textbf{Tr}\left[ \bSigma^{-1}\bomega_{ij}\bomega_{ij}^\top\right]H_{ij}(1-H_{ij}) = \textbf{Tr}\left[  \bSigma^{-1}\sum_{1\leq i\leq j \leq n}\bomega_{ij}\bomega_{ij}^\top H_{ij}(1-H_{ij})\right] \nonumber \\
    =& \textbf{Tr}\left[  \bSigma^{-1}\sum_{1\leq i\leq j \leq n}\mathbb{E}\left[(W_{ij}\bomega_{ij})(W_{ij}\bomega_{ij})^\top\right]\right] = \textbf{Tr}\left[  \bSigma^{-1}\bSigma\right] = r.\label{distributionaleq2}
\end{align}
Combine \eqref{distributionaleq1} and \eqref{distributionaleq2} we know that
\begin{align}
    \left|\mathbb{P}(\Delta\bpi_\mathcal{I} \in \mathcal{D}) - \mathbb{P}(\mathcal{N}(\boldsymbol{0}_r,\bSigma) \in \mathcal{D})\right|\lesssim r^{5/4}\max_{1\leq i\leq j\leq n}\left\|\bSigma^{-1/2}\bomega_{ij}\right\|_2.\label{distributionaleq3}
\end{align}

It remains to control $\left|\mathbb{P}(\wh\bpi_\mathcal{I} - \bpi_\mathcal{I} \in \mathcal{D}) - \mathbb{P}(\Delta\bpi_\mathcal{I} \in \mathcal{D}) \right|$. For any convex set $\mathcal{D}\subset \mathbb{R}^r$ and point $x\in \mathbb{R}^r$, we define
\begin{align*}
    \delta_{\mathcal{D}}(x):= \begin{cases}-\min_{y\in \mathbb{R}^r \backslash \mathcal{D}}\left\|x- y\right\|_2, & \text { if } x \in \mathcal{D} \\ \min_{y\in\mathcal{D} }\left\|x- y\right\|_2, & \text { if } x \notin \mathcal{D}\end{cases} \text{ and } \mathcal{D}^{\varepsilon}:=\left\{x \in \mathbb{R}^r: \delta_{\mathcal{D}}(x) \leq \varepsilon\right\}.
\end{align*}
With this definition, we have 
\begin{align*}
    \mathbb{P}(\bSigma^{-1/2}\Delta\bpi_{\mathcal{I}}\in \mathcal{D}^{-\varepsilon}) = &\mathbb{P}\left(\bSigma^{-1/2}\Delta\bpi_{\mathcal{I}}\in \mathcal{D}^{-\varepsilon}, \left\|\bSigma^{-1/2}\left(\wh\bpi_{\mathcal{I}}- \bpi_{\mathcal{I}}-\Delta_{\mathcal{I}}\right)\right\|_2\leq \varepsilon\right) \\
    & +\mathbb{P}\left(\bSigma^{-1/2}\Delta\bpi_{\mathcal{I}}\in \mathcal{D}^{-\varepsilon}, \left\|\bSigma^{-1/2}\left(\wh\bpi_{\mathcal{I}}- \bpi_{\mathcal{I}}-\Delta_{\mathcal{I}}\right)\right\|_2> \varepsilon\right) \\
    \leq & \mathbb{P}\left(\bSigma^{-1/2}(\wh\bpi_{\mathcal{I}}- \bpi_{\mathcal{I}})\in \mathcal{D}\right) + \mathbb{P}\left( \left\|\bSigma^{-1/2}\left(\wh\bpi_{\mathcal{I}}- \bpi_{\mathcal{I}}-\Delta_{\mathcal{I}}\right)\right\|_2> \varepsilon\right) .
\end{align*}
Taking $\varepsilon = \lambda_{r}^{-1/2}(\bSigma)r^{1/2}\varepsilon_3$, by Theorem \ref{Piexpansion} we know that $\|\bSigma^{-1/2}\left(\wh\bpi_{\mathcal{I}}- \bpi_{\mathcal{I}}-\Delta_{\mathcal{I}}\right)\|_2\leq \varepsilon$ with probability at least $1-O(n^{-9})$. As a result, we have 
\begin{align}
    \mathbb{P}(\bSigma^{-1/2}\Delta\bpi_{\mathcal{I}}\in \mathcal{D}^{-\varepsilon}) \leq \mathbb{P}\left(\bSigma^{-1/2}(\wh\bpi_{\mathcal{I}}- \bpi_{\mathcal{I}})\in \mathcal{D}\right) +O(n^{-9}).\label{distributionaleq4}
\end{align}
On the other hand, one can see that
\begin{align}
    &|\mathbb{P}(\bSigma^{-1/2}\Delta\bpi_{\mathcal{I}}\in \mathcal{D}^{-\varepsilon}) - \mathbb{P}(\bSigma^{-1/2}\Delta\bpi_{\mathcal{I}}\in \mathcal{D})| \leq |\mathbb{P}(\bSigma^{-1/2}\Delta\bpi_{\mathcal{I}}\in \mathcal{D}^{-\varepsilon}) - \mathbb{P}(\mathcal{N}(\boldsymbol{0}_r,\bI_r)\in \mathcal{D}^{-\varepsilon})| \nonumber \\
    & + |\mathbb{P}(\mathcal{N}(\boldsymbol{0}_r,\bI_r)\in \mathcal{D}^{-\varepsilon}) - \mathbb{P}(\mathcal{N}(\boldsymbol{0}_r,\bI_r)\in \mathcal{D})| + |\mathbb{P}(\mathcal{N}(\boldsymbol{0}_r,\bI_r)\in \mathcal{D}) - \mathbb{P}(\bSigma^{-1/2}\Delta\bpi_{\mathcal{I}}\in \mathcal{D}) |. \label{distributionaleq5}
\end{align}
By \cite[Theorem 1.2]{raivc2019multivariate} we know that
\begin{align*}
    |\mathbb{P}(\mathcal{N}(\boldsymbol{0}_r,\bI_r)\in \mathcal{D}^{-\varepsilon}) - \mathbb{P}(\mathcal{N}(\boldsymbol{0}_r,\bI_r)\in \mathcal{D})|\lesssim r^{1/4} \varepsilon  \lesssim \lambda_{r}^{-1/2}(\bSigma)r^{3/4}\varepsilon_3.
\end{align*}
Plugging this and \eqref{distributionaleq3} in \eqref{distributionaleq5}, we have
\begin{align*}
    |\mathbb{P}(\bSigma^{-1/2}\Delta\bpi_{\mathcal{I}}\in \mathcal{D}^{-\varepsilon}) - \mathbb{P}(\bSigma^{-1/2}\Delta\bpi_{\mathcal{I}}\in \mathcal{D})|\lesssim r^{5/4}\max_{1\leq i\leq j\leq n}\left\|\bSigma^{-1/2}\bomega_{ij}\right\|_2 + \lambda_{r}^{-1/2}(\bSigma)r^{3/4}\varepsilon_3.
\end{align*}
Combine this with \eqref{distributionaleq4}, and by the arbitrariness of $\mathcal{D}$, we get 
\begin{align*}
    \mathbb{P}(\Delta\bpi_{\mathcal{I}}\in \mathcal{D}) \leq \mathbb{P}\left((\wh\bpi_{\mathcal{I}}- \bpi_{\mathcal{I}})\in \mathcal{D}\right) +O\left(r^{5/4}\max_{1\leq i\leq j\leq n}\left\|\bSigma^{-1/2}\bomega_{ij}\right\|_2 + \lambda_{r}^{-1/2}(\bSigma)r^{3/4}\varepsilon_3\right).
\end{align*}
Similarly, it can also be shown that
\begin{align*}
    \mathbb{P}\left((\wh\bpi_{\mathcal{I}}- \bpi_{\mathcal{I}})\in \mathcal{D}\right) \leq \mathbb{P}(\Delta\bpi_{\mathcal{I}}\in \mathcal{D}) +O\left(r^{5/4}\max_{1\leq i\leq j\leq n}\left\|\bSigma^{-1/2}\bomega_{ij}\right\|_2 + \lambda_{r}^{-1/2}(\bSigma)r^{3/4}\varepsilon_3\right).
\end{align*}
As a result, we know that 
\begin{align}
    \left|\mathbb{P}(\wh\bpi_\mathcal{I} - \bpi_\mathcal{I} \in \mathcal{D}) - \mathbb{P}(\Delta\bpi_\mathcal{I} \in \mathcal{D}) \right|\lesssim r^{5/4}\max_{1\leq i\leq j\leq n}\left\|\bSigma^{-1/2}\bomega_{ij}\right\|_2 + \lambda_{r}^{-1/2}(\bSigma)r^{3/4}\varepsilon_3.\label{distributionaleq6}
\end{align}
Combine \eqref{distributionaleq3} and \eqref{distributionaleq6}, we get
\begin{align*}
    \left|\mathbb{P}(\wh\bpi_\mathcal{I} - \bpi_\mathcal{I} \in \mathcal{D}) - \mathbb{P}(\mathcal{N}(\boldsymbol{0}_r,\bSigma) \in \mathcal{D})\right| \lesssim r^{5/4}\max_{1\leq i\leq j\leq n}\left\|\bSigma^{-1/2}\bomega_{ij}\right\|_2 + \lambda_{r}^{-1/2}(\bSigma)r^{3/4}\varepsilon_3.
\end{align*}
\end{proof}

\subsection{Auxiliary Lemmas}
\begin{lem}\label{lemmaN1N2}
    For $\bN_1$ and $\bN_2$ defined in ~\eqref{N1N2}, we have
    \begin{align*}
        &\left\|\bN_1-\frac{1}{\lambda_1^*}\bI\right\|_{2,\infty}\lesssim \sqrt{\frac{(K-1)\mu^*}{n\lambda_1^{*2}}}, \quad\; \left\|\bN_1\right\|\lesssim \frac{1}{\lambda_1^*} \\
        &\left\|\bN_2-\frac{1}{\lambda_1^{*2}}\bI\right\|_{2,\infty}\lesssim \sqrt{\frac{(K-1)\mu^*}{n\lambda_1^{*4}}},\quad \left\|\bN_2\right\|\lesssim \frac{1}{\lambda_1^{*2}}.
    \end{align*}
\end{lem}
\begin{proof} 
We prove the desired results in the following two settings: with self-loop and without self-loop. When self-loops are allowed, the rank of $\bH$ is exactly $K$. When there is no self-loop, $\bH$ is approximately rank $K$.

\noindent\textbf{With self-loop:} 
The spectral norm bounds follow directly from ~\eqref{N1N2}
\begin{align*}
    \left\|\bN_1\right\|  = \max_{2\leq i\leq n}\frac{1}{\lambda_1^*-\lambda_i^*}\lesssim \frac{1}{\lambda_1^* },\quad \left\|\bN_2\right\|  = \max_{2\leq i\leq n}\frac{1}{(\lambda_1^*-\lambda_i^*)^2}\lesssim \frac{1}{\lambda_1^{*2} }.
\end{align*}
By definition, for $\bN_1$ we have
\begin{align*}
    \bN_1-\frac{1}{\lambda_1^*}\bI = \sum_{i=2}^n\frac{1}{\lambda_1^*-\lambda_i^*}\bu_i^*\bu_i^{*\top}  - \frac{1}{\lambda_1^*}\sum_{i=1}^n\bu_i^*\bu_i^{*\top}= -\frac{1}{\lambda_1^*}\bu_1^*\bu_1^{*\top}+\sum_{i=2}^K\frac{\lambda_i^*}{\lambda_1^*(\lambda_1^*-\lambda_i^*)}\bu_i^*\bu_i^{*\top}.
\end{align*}
On one hand, by \eqref{eq:incoherence} we have
\begin{align*}
    \left\|\frac{1}{\lambda_1^*}\bu_1^*\bu_1^{*\top}\right\|_{2,\infty} = \frac{1}{\lambda_1^*}\left\|\bu_1^*\right\|_\infty\left\|\bu_1^*\right\|_2\leq \sqrt{\frac{\mu^*}{n\lambda_1^{*2}}}.
\end{align*}
On the other hand, defining $$\bC := \textbf{diag}\left(\frac{\lambda_2^*}{\lambda_1^*(\lambda_1^*-\lambda_2^*)}, \frac{\lambda_3^*}{\lambda_1^*(\lambda_1^*-\lambda_3^*)},\dots, \frac{\lambda_K^*}{\lambda_1^*(\lambda_1^*-\lambda_K^*)}\right),$$ 
by \eqref{eq:incoherence} we have,
\begin{align*}
    \left\|\sum_{i=2}^K\frac{\lambda_i^*}{\lambda_1^*(\lambda_1^*-\lambda_i^*)}\bu_i^*\bu_i^{*\top}\right\|_{2,\infty} &= \left\|\oU\bC\oU^\top\right\|_{2,\infty}\leq \left\|\oU\right\|_{2,\infty}\left\|\bC\oU^\top\right\| \lesssim \sqrt{\frac{(K-1)\mu^*}{n\lambda_1^{*2}}}.
\end{align*}
As a result, we get $\left\|\bN_1-\frac{1}{\lambda_1^*}\bI\right\|_{2,\infty}\lesssim\sqrt{\frac{(K-1)\mu^*}{n\lambda_1^{*2}}}$. Similarly, for $\bN_2$ we have 
\begin{align*}
    \bN_2-\frac{1}{\lambda_1^{*2}}\bI = -\frac{1}{\lambda_1^{*2}}\bu_1^*\bu_1^{*\top}+\sum_{i=2}^K\left(\frac{1}{(\lambda_1^*-\lambda_i^*)^2}-\frac{1}{\lambda_1^{*2}}\right)\bu_i^*\bu_i^{*\top}.
\end{align*}
Again, we have $\|\bu_1^*\bu_1^{*\top} / \lambda_1^{*2}\|_{2,\infty}\lesssim \sqrt{\mu^*/(n\lambda_1^{*4})}$. Defining
$$\bC_1 = \textbf{diag}\left(\frac{1}{(\lambda_1^*-\lambda_2^*)^2}-\frac{1}{\lambda_1^{*2}},\frac{1}{(\lambda_1^*-\lambda_3^*)^2}-\frac{1}{\lambda_1^{*2}}, \dots, \frac{1}{(\lambda_1^*-\lambda_K^*)^2}-\frac{1}{\lambda_1^{*2}}\right),$$
we have
\begin{align*}
    \left\|\sum_{i=2}^K\left(\frac{1}{(\lambda_1^*-\lambda_i^*)^2}-\frac{1}{\lambda_1^{*2}}\right)\bu_i^*\bu_i^{*\top}\right\|_{2,\infty} &= \left\|\oU\bC_1\oU^\top\right\|_{2,\infty}\leq \left\|\oU\right\|_{2,\infty}\left\|\bC_1\oU^\top\right\| \lesssim \sqrt{\frac{(K-1)\mu^*}{n\lambda_1^{*4}}}.
\end{align*}

Combining together, we get the desired conclusion
    $\|\bN_2-\frac{1}{\lambda_1^{*2}}\bI\|_{2,\infty}\lesssim\sqrt{\frac{(K-1)\mu^*}{n\lambda_1^{*4}}}$. 

\noindent\textbf{Without self-loop:} 
The spectral norm bounds follow same as the previous case. For the rest, by definition of $\bN_1$,
\begin{align*}
    \bN_1-\frac{1}{\lambda_1^*}\bI = \sum_{i=2}^n\frac{1}{\lambda_1^*-\lambda_i^*}\bu_i^*\bu_i^{*\top}  - \frac{1}{\lambda_1^*}\sum_{i=1}^n\bu_i^*\bu_i^{*\top}=& -\frac{1}{\lambda_1^*}\bu_1^*\bu_1^{*\top}+\sum_{i=2}^K\frac{\lambda_i^*}{\lambda_1^*(\lambda_1^*-\lambda_i^*)}\bu_i^*\bu_i^{*\top}  \\
    &+\sum_{i=K+1}^n\frac{\lambda_i^*}{\lambda_1^*(\lambda_1^*-\lambda_i^*)}\bu_i^*\bu_i^{*\top}
\end{align*}
The bounds on the first two summands are same as before. Hence it remains to control $\|\sum_{i=K+1}^n\frac{\lambda_i^*}{\lambda_1^*(\lambda_1^*-\lambda_i^*)}\bu_i^*\bu_i^{*\top}\|_{2,\infty}$. One can see that
\begin{align*}
    \left\|\sum_{i=K+1}^n\frac{\lambda_i^*}{\lambda_1^*(\lambda_1^*-\lambda_i^*)}\bu_i^*\bu_i^{*\top}\right\|_{2,\infty}&\leq \left\|\sum_{i=K+1}^n\frac{\lambda_i^*}{\lambda_1^*(\lambda_1^*-\lambda_i^*)}\bu_i^*\bu_i^{*\top}\right\| \leq \max_{K+1\leq i\leq n} \left|\frac{\lambda_i^*}{\lambda_1^*(\lambda_1^*-\lambda_i^*)}\right| \\
    &\lesssim \frac{\left\|\textbf{diag}(\bTheta\bPi\bP\bPi^\top\bTheta)\right\|}{\lambda_1^{*2}}\leq \frac{\theta_{\text{max}}^2}{\lambda_1^{*2}}.
\end{align*}
As a result, we get 
\begin{align*}
    \left\|\bN_1-\frac{1}{\lambda_1^*}\bI\right\|_{2,\infty}\lesssim\sqrt{\frac{(K-1)\mu^*}{n\lambda_1^{*2}}} +\frac{\theta_{\text{max}}^2}{\lambda_1^{*2}}\lesssim\sqrt{\frac{(K-1)\mu^*}{n\lambda_1^{*2}}},
\end{align*}
since $\sqrt{n}\theta_{\text{max}}^2\lesssim \sqrt{n}\theta_{\text{max}}\ll \lambda_1^*$. Similarly, for $\bN_2$ we have 
\begin{align*}
    \bN_2-\frac{1}{\lambda_1^{*2}}\bI = -\frac{1}{\lambda_1^{*2}}\bu_1^*\bu_1^{*\top}+\sum_{i=2}^K\left(\frac{1}{(\lambda_1^*-\lambda_i^*)^2}-\frac{1}{\lambda_1^{*2}}\right)\bu_i^*\bu_i^{*\top}+\sum_{i=K+1}^n\left(\frac{1}{(\lambda_1^*-\lambda_i^*)^2}-\frac{1}{\lambda_1^{*2}}\right)\bu_i^*\bu_i^{*\top}.
\end{align*}
We bound the first two summands as before. For the third term,
\begin{align*}
    &\left\|\sum_{i=K+1}^n\left(\frac{1}{(\lambda_1^*-\lambda_i^*)^2}-\frac{1}{\lambda_1^{*2}}\right)\bu_i^*\bu_i^{*\top}\right\|_{2,\infty}\leq \left\|\sum_{i=K+1}^n\left(\frac{1}{(\lambda_1^*-\lambda_i^*)^2}-\frac{1}{\lambda_1^{*2}}\right)\bu_i^*\bu_i^{*\top}\right\| \\
    \lesssim & \max_{K+1\leq i\leq n}\left|\frac{1}{(\lambda_1^*-\lambda_i^*)^2}-\frac{1}{\lambda_1^{*2}}\right|\lesssim \max_{K+1\leq i\leq n}\left|\frac{\lambda_i^*}{\lambda_1^{*3}}\right|\leq \frac{\left\|\textbf{diag}(\bTheta\bPi\bP\bPi^\top\bTheta)\right\|}{\lambda_1^{*3}}\leq \frac{\theta_{\text{max}}^2}{\lambda_1^{*3}}.
\end{align*}
Combine them together, we get
\begin{align*}
    \left\|\bN_2-\frac{1}{\lambda_1^{*2}}\bI\right\|_{2,\infty}\lesssim\sqrt{\frac{(K-1)\mu^*}{n\lambda_1^{*4}}} + \frac{\theta_{\text{max}}^2}{\lambda_1^{*3}}\lesssim \sqrt{\frac{(K-1)\mu^*}{n\lambda_1^{*4}}},
\end{align*} 
since $\sqrt{n}\theta_{\text{max}}^2\lesssim \sqrt{n}\theta_{\text{max}}\ll \lambda_1^*$. This completes the proof.
\end{proof}

 From Lemma \ref{lemmaN1N2}, We immediately have the following corollary.
\begin{cor}\label{corN1N2}
    For $\bN_1$ and $\bN_2$ defined in ~\eqref{N1N2}, we have for all $\bx\in\mathbb{R}^n$
    \begin{align*}
        &\left\|\bN_i\bx\right\|_{\infty}\lesssim \frac{1}{\lambda_1^{*i}}\left\|\bx\right\|_{\infty}+\sqrt{\frac{(K-1)\mu^*}{n\lambda_1^{*2i}}}\left\|\bx\right\|_{2}, \quad i=1,2. 
    \end{align*}
\end{cor}

\begin{lem}\label{Wconcentration1}
    For any $i\in [n]$ and a fixed vector $\bx\in \mathbb{R}^n$, we have
    \begin{align*}
        \left|\bW_{i,\cdot}\bx\right|\lesssim \sqrt{\log n }\theta_{\text{max}} \left\|\bx\right\|_2  + \log n \left\|\bx\right\|_\infty
    \end{align*}
    with probability at least $1-O(n^{-15})$. Here the constant hidden in $\lesssim$ is free of $n$, $\bx$ and $\theta_{\text{max}}$.
\end{lem}
\begin{proof}
Since $|W_{ij}|\leq 1$, by Bernstein inequality,with probability at least $1-O(n^{-15})$, 
\begin{align*}
    \left|\bW_{i,\cdot}\bx\right|&\lesssim  \sqrt{\log n \sum_{j=1}^n \mathbb{E}\left[W_{ij}^2\right]x_j^2}  + \log n \left\|\bx\right\|_\infty  \lesssim \sqrt{\log n }\theta_{\text{max}} \left\|\bx\right\|_2  + \log n \left\|\bx\right\|_\infty
\end{align*}
\end{proof}

\begin{lem}\label{Wconcentration2}
    For any $i\in [n]$ and a fixed matrix $\boldsymbol{A}\in \mathbb{R}^{n\times m}$, we have with probability at least $1-O(n^{-15}m)$,
    \begin{align*}
        \left\|\bW_{i,\cdot}\boldsymbol{A}\right\|_{2}\lesssim \sqrt{\log n}\theta_{\text{max}}\left\|\boldsymbol{A}\right\|_F +\log n\left\|\boldsymbol{A}\right\|_{2, \infty}.
    \end{align*}
\end{lem}
\begin{proof}
    Taking $E = \bW_{i,\cdot}$ in \cite[Lemma 5]{yan2021inference} gives us the desired statement.
\end{proof}
\begin{lem}\label{Wconcentration3}
    For any fixed matrix $\boldsymbol{A}\in \mathbb{R}^{n\times m}$, we have with probability at least $1-O(n^{-14}m)$,
    \begin{align*}
        \left\|\bW\boldsymbol{A}\right\|_{2, \infty}\lesssim \sqrt{\log n}\theta_{\text{max}}\left\|\boldsymbol{A}\right\|_F +\sqrt{m}\log n\left\|\boldsymbol{A}\right\|_{\text{max}}.
    \end{align*}
\end{lem}
\begin{proof}
    Recall that
    $\|\bW\boldsymbol{A}\|_{2,\infty} = \max_{1\leq i\leq n}\|\bW_{i,\cdot}\boldsymbol{A}\|_2$.
    Applying Lemma \ref{Wconcentration2} for $i\in [n]$, we get the desired conclusion.
\end{proof}

\begin{lem}\label{Wconcentration4}
    For any $i\in [n]$ and a fixed matrix $\boldsymbol{A}\in \mathbb{R}^{n\times m}$, we have  with probability at least $1-O(n^{-15}m)$,
    \begin{align*}
        \left\|\Big(\bW-\bW^{(i)}\Big)\boldsymbol{A}\right\|_{F}\lesssim \sqrt{\log n}\theta_{\text{max}}\left\|\boldsymbol{A}\right\|_F  + (\sqrt{n}\theta_{\text{max}}+\log n)\left\|\boldsymbol{A}\right\|_{2,\infty}.
    \end{align*}
\end{lem}
\begin{proof}
By definition we have
\begin{align}
    \left\|\Big(\bW-\bW^{(i)}\Big)\boldsymbol{A}\right\|_{F} &= \sqrt{\left\|\bW_{i,\cdot}\boldsymbol{A}\right\|_{2}^2+\sum_{j\in [n],j\neq i}W_{ji}^2\left\|\boldsymbol{A}_{i,\cdot}\right\|_2^2} \nonumber \\
    &\lesssim \left\|\bW_{i,\cdot}\boldsymbol{A}\right\|_{2}+\left\|\bW_{\cdot, i}\right\|_2\left\|\boldsymbol{A}_{i,\cdot}\right\|_2. \label{Wconcentration4eq1}
\end{align}
By Bernstein inequality, we know that
\begin{align}
    \left|\sum_{j\in [n]}W_{ji}^2-\sum_{j\in [n]}\mathbb{E}\left[W_{ji}^2\right]\right|&\lesssim \sqrt{\log n \sum_{j=1}^n \mathbb{E}\left[W_{ji}^4\right]} + \log n \nonumber \\
    &\lesssim \sqrt{n\log n \theta_{\text{max}}^2}+\log n \lesssim \sqrt{n\log n}\theta_{\text{max}}.\label{Wconcentration4eq2}
\end{align}
with probability at least $1-O(n^{-15})$. On the other hand, we know that
\begin{align}
    \left|\sum_{j\in [n]}\mathbb{E}\left[W_{ji}^2\right]\right|\lesssim n \theta_{\text{max}}^2.\label{Wconcentration4eq3}
\end{align}
Combine \eqref{Wconcentration4eq2} and \eqref{Wconcentration4eq3} we know that 
\begin{align*}
    \left|\sum_{j\in [n]}W_{ji}^2\right|\lesssim \sqrt{n\log n}\theta_{\text{max}}+ n \theta_{\text{max}}^2\lesssim n \theta_{\text{max}}^2
\end{align*}
with probability at least $1-O(n^{-15})$. Plugging this as well as Lemma \ref{Wconcentration2} in \eqref{Wconcentration4eq1} we get 
\begin{align*}
    \left\|\Big(\bW-\bW^{(i)}\Big)\boldsymbol{A}\right\|_{F}\lesssim \sqrt{\log n}\theta_{\text{max}}\left\|\boldsymbol{A}\right\|_F  + (\sqrt{n}\theta_{\text{max}}+\log n)\left\|\boldsymbol{A}\right\|_{2,\infty}
\end{align*}
with probability at least $1-O(n^{-15}m)$.
\end{proof}

\begin{lem}\label{Wconcentration5}
    For any fixed $\bu,\bv\in \mathbb{R}^n$ , we have with probability at least $1-O(n^{-15})$,
    \begin{align*}
        \left|\bu^{\top}\bW\bv\right|\lesssim \sqrt{\log n}\|\bu\|_2\|\bv\|_2.
    \end{align*}
    In particular, for any $i, j\in [n]$ , we have with probability at least $1-O(n^{-15})$,
    \begin{align*}
        \left|\bu_i^{*\top}\bW\bu_j^*\right|\lesssim \sqrt{\log n}.
    \end{align*}
\end{lem}
\begin{proof}
Since $|W_{kl}|\leq 1$, by Hoeffding's inequality, we have
    \begin{align*}
        \mathbb{P}\left(\left|\bu^{\top}\bW\bv\right|\geq t\right)&\leq 2\exp\left(-\frac{2t^2}{\sum_{k=1}^n\left(2u_kv_k\right)^2+\sum_{1\leq k<l\leq n}\left(2u_kv_l+2u_lv_k\right)^2}\right) \\
        &\leq 2\exp\left(-\frac{t^2}{2\sum_{k=1}^nu_k^2v_k^2+4\sum_{1\leq k<l\leq n}[u_k^2v_l^2+u_l^2v_k^2]}\right) \\
        &\leq 2\exp\left(-\frac{t^2}{4\sum_{k,l=1}^nu_k^2v_l^2}\right) = 2\exp\left(-\frac{t^2}{4\|\bu\|_2^2\|\bv\|_2^2}\right).
    \end{align*}
    As a result, with probability at least $1-O(n^{-15})$, we have $|\bu^{\top}\bW\bv|\lesssim \sqrt{\log n}\|\bu\|_2\|\bv\|_2$.

\end{proof}

\begin{lem}\label{Wconcentration6}
    For any $\bv\in \mathbb{R}^n$ with $\|\bv\|_2 = 1$ , with probability at least $1-O(n^{-15})$,
    \begin{align*}
        \left|\bu_1^{*\top}\bW\bv\right|\lesssim \sqrt{\log n} \theta_{\text{max}}+\log n \sqrt{\mu^*/n}.
    \end{align*}
\end{lem}
\begin{proof}
We write
\begin{align*}
    \bu_1^{*\top}\bW\bv = \sum_{i=1}^n W_{ii}(\bu_1^*)_i v_i +\sum_{1\leq i<j\leq n}W_{ij}((\bu_1^*)_i v_j + (\bu_1^*)_j v_i).
\end{align*}
Since $\|\bv\|_2 = 1$, we know, by \eqref{eq:incoherence},
\begin{align*}
    \max_{i\in [n]}|(\bu_1^*)_i v_i|\leq \sqrt{\mu^*/n}\quad \text{and}\quad\max_{1\leq i<j\leq n}|(\bu_1^*)_i v_j + (\bu_1^*)_j v_i|\leq 2\sqrt{\mu^*/n}.
\end{align*}
As a result, by Bernstein inequality we have, with probability at least $1-O(n^{-15})$,
\begin{align*}
    \left|\bu_1^{*\top}\bW\bv\right|\lesssim &\sqrt{\log n \Bigg(\sum_{i=1}^n ((\bu_1^*)_i v_i)^2\mathbb{E}[W_{ii}^2]+\sum_{1\leq i<j\leq n}((\bu_1^*)_i v_j + (\bu_1^*)_j v_i)^2\mathbb{E}[W_{ij}^2]\Bigg)} \\
    &+\log n \sqrt{\mu^*/n} \\
    \lesssim & \sqrt{\log n \sum_{i,j=1}^n ((\bu_1^*)_i v_j)^2\theta_{\text{max}}^2}+\log n \sqrt{\mu^*/n}
    \lesssim \sqrt{\log n} \theta_{\text{max}}+\log n \sqrt{\mu^*/n}
\end{align*}

\end{proof}

\subsection{The Incoherence Parameter $\mu^*$}\label{IncoherenceExplanation}
In this subsection, we aim to show that the incoherence condition \eqref{eq:incoherence} holds with $\mu^*\asymp 1$. As we pointed out in Remark \ref{rem1}, \cite[Lemma C.3]{jin2017estimating} and Assumption \ref{assn:theta_order} guarantees $\left\|\bu_1^*\right\|_{\infty}
      \lesssim \sqrt{\frac{1}{n}}$. Therefore it remains to show 
\begin{align*}
    \left\|\oU^*\right\|_{2, \infty} \lesssim  \sqrt{\frac{K}{n}}.
\end{align*}
We use the notation $\oU^*_1$ for $\oU^*$ when there exists self-loops, and use the notation $\oU^*_2$ for $\oU^*$ when the self-loops are not allowed. 

\noindent\textbf{With self-loop:} Recall the definition of $\bB^*$ in Lemma \ref{membershipreconstructionlem3}. By \cite[(C.26)]{jin2017estimating}, we know that $\left\|\bB^*\right\|\lesssim \sqrt{K}$. As a result, we have $\|\bb^*_k\|_2\leq \sqrt{K}$ for $1\leq k\leq K$. Since $\br_i^*, i\in [n]$ are convex combinations of $\bb_k^*$s, one can see that $\|\br^*_i\|_2\leq \sqrt{K}$ for $1\leq i\leq n$. Hence, by the definition of $\br_i^*$ in \eqref{eq:define_r_star}, 
\begin{align*}
    \left\|\left(\oU^*_1\right)_{i,\cdot}\right\|_2 = \left\|(\bu_1^*)_i\br_i^*\right\|_2\lesssim \sqrt{\frac{K}{n}}.
\end{align*}

\noindent\textbf{Without self-loop:} Define $\bR_{\oU}$ as the rotation matrix matches $\oU_1^*$ and $\oU_2^*$
\begin{align*}
    \bR_{\oU} := \argmin_{\boldsymbol{O}\in \mathcal{O}^{(K-1)\times (K-1)}}\left\|\oU_1^*\boldsymbol{O}-\oU^*_2\right\|_F.  
\end{align*}
By Wedin's sin$\Theta$ Theorem \cite[Theorem 2.9]{chen2021spectral}, we have
\begin{align*}
    \left\|\oU_1^*\bR_{\oU}-\oU^*_2\right\|\lesssim \frac{\left\|\textbf{diag}(\bTheta\bPi\bP\bPi^\top\bTheta)\right\|}{\minsigma^*}\lesssim \frac{\theta_{\text{max}}^2}{\beta_n K^{-1}n\theta_{\text{max}}^2}\lesssim \frac{K}{\beta_n n}.
\end{align*}
As a result, we have
\begin{align*}
    \left\|\oU^*_2\right\|_{2,\infty}&\leq \left\|\oU^*_1\bR_{\oU}\right\|_{2,\infty} + \left\|\oU_1^*\bR_{\oU}-\oU^*_2\right\|_{2,\infty} \leq\left\|\oU^*_1\bR_{\oU}\right\|_{2,\infty} + \left\|\oU_1^*\bR_{\oU}-\oU^*_2\right\| \\
    &\lesssim \sqrt{\frac{K}{n}}+ \frac{K}{\beta_n n}\lesssim \sqrt{\frac{K}{n}}
\end{align*}
as long as $n\gtrsim K \beta_n^{-2}$, which is a relatively mild assumption and is assumed to be true from Theorem \ref{mainthmmatrixdenoising} to Theorem \ref{distributionthm} (Theorem \ref{mainthmu1} does not involve $\oU^*$).

\subsection{Proof of \eqref{exa2eq1}}\label{exa2eq1proof}
We formally state the theoretical guarantee of the Gaussian multiplier bootstrap method described in Example \ref{exa2} as below.
\begin{thm}
Assume the conditions in Theorem \ref{Piexpansion} hold. As long as 
\begin{align*}
    \max_{j:j\neq i}\frac{\|\boldsymbol{C}^{\bpi}_{j,k} - \boldsymbol{C}^{\bpi}_{i,k}\|_{\text{max}}\log^{5/2}n}{\sqrt{V_{\boldsymbol{C}^{\bpi}_{j,k} - \boldsymbol{C}^{\bpi}_{i,k}}}} = o(1) \quad \text{ and } \quad \max_{j:j\neq i}\frac{\varepsilon_3\sqrt{\log n}}{\sqrt{V_{\boldsymbol{C}^{\bpi}_{j,k} - \boldsymbol{C}^{\bpi}_{i,k}}}} = o(1),
\end{align*}
we have 
\begin{align*}
     \left|\mathbb{P}(\mathcal{T}>c_{1-\alpha})-\alpha\right|\to 0.
\end{align*}
\end{thm}
\begin{proof}
We define 
\begin{align*}
    &\mathcal{T}^{\sharp} = \max_{j:j\neq i}\left|\frac{\textbf{Tr}\left[\left(\boldsymbol{C}^{\bpi}_{j,k} - \boldsymbol{C}^{\bpi}_{i,k}\right)\bW\right]}{\sqrt{V_{\boldsymbol{C}^{\bpi}_{j,k} - \boldsymbol{C}^{\bpi}_{i,k}}}}\right|.
\end{align*}
Since $|W_{ij}|\leq 1$, we know that \cite[Condition E, M]{chernozhuokov2022improved} holds with $b_1=b_2=1$ and $B_n\asymp n \|\boldsymbol{C}^{\bpi}_{j,k} - \boldsymbol{C}^{\bpi}_{i,k}\|_{\text{max}} / \sqrt{V_{\boldsymbol{C}^{\bpi}_{j,k} - \boldsymbol{C}^{\bpi}_{i,k}}}$. As a result, by \cite[Theorem 2.2]{chernozhuokov2022improved} we have
\begin{align}
     \left|\mathbb{P}(\mathcal{T}^{\sharp}>c_{1-\alpha})-\alpha\right|&\lesssim \left(\frac{B_n^2 \log^5((n-1)(n^2+n)/2)}{(n^2+n)/2}\right)^{1/4} \nonumber \\
     &\lesssim \log^{5/4}n \sqrt{\frac{\|\boldsymbol{C}^{\bpi}_{j,k} - \boldsymbol{C}^{\bpi}_{i,k}\|_{\text{max}}}{\sqrt{V_{\boldsymbol{C}^{\bpi}_{j,k} - \boldsymbol{C}^{\bpi}_{i,k}}}}} \to 0. \label{exa2proofeq1}
\end{align}
Therefore, to prove \eqref{exa2eq1}, it is enough to show 
\begin{align}\label{eq:tsharp_close}
    \sup_{x\in \mathbb{R}}\left|\mathbb{P}(\mathcal{T}^{\sharp}\leq x) - \mathbb{P}(\mathcal{T}\leq x)\right|\to 0. 
\end{align}
One can see that
\begin{align}
     \sup_{x\in \mathbb{R}}\left|\mathbb{P}(\mathcal{T}^{\sharp}\leq x) - \mathbb{P}(\mathcal{T}\leq x)\right|\leq \mathbb{P}(|\mathcal{T}^{\sharp}-\mathcal{T}|>\delta)+\sup_{x\in \mathbb{R}}\mathbb{P}(x< \mathcal{T}^{\sharp} \leq x+\delta),\quad \forall \delta >0. \label{exa2proofeq2}
\end{align}
We take $\delta\asymp \max_{j:j\neq i}\varepsilon_3 / \sqrt{V_{\boldsymbol{C}^{\bpi}_{j,k} - \boldsymbol{C}^{\bpi}_{i,k}}}$. Then by Theorem \ref{Piexpansion} we know that 
\begin{align}
    \mathbb{P}(|\mathcal{T}^{\sharp}-\mathcal{T}|>\delta) = O(n^{-10}).\label{exa2proofeq3}
\end{align}
Next we show that $\sup_{x\in \mathbb{R}}\mathbb{P}(x< \mathcal{T}^{\sharp} \leq x+\delta)\to 0$. Let $(Z_1, Z_2,\dots, Z_{n-1})$ be a centered Gaussian random vector with the same covariance structure (thus we know that $Z_k\sim \mathcal{N}(0, 1),\; \forall k\in [n-1]$) as 
\begin{align*}
    \left(\frac{\textbf{Tr}\left[\left(\boldsymbol{C}^{\bpi}_{j,k} - \boldsymbol{C}^{\bpi}_{i,k}\right)\bW\right]}{\sqrt{V_{\boldsymbol{C}^{\bpi}_{j,k} - \boldsymbol{C}^{\bpi}_{i,k}}}}:j\in [n]\backslash i\right)\in \mathbb{R}^{n-1}.
\end{align*}
Then by \cite[Theorem 2.1]{chernozhuokov2022improved} we know that 
\begin{align*}
    \sup_{x\in\mathbb{R}}\left|\mathbb{P}(\mathcal{T}^{\sharp}\leq x) - \mathbb{P}\left(\max_{1\leq k\leq n-1}|Z_k|\leq x\right)\right|&\lesssim \left(\frac{B_n^2 \log^5((n-1)(n^2+n)/2)}{(n^2+n)/2}\right)^{1/4} \\
     &\lesssim \log^{5/4}n \sqrt{\frac{\|\boldsymbol{C}^{\bpi}_{j,k} - \boldsymbol{C}^{\bpi}_{i,k}\|_{\text{max}}}{\sqrt{V_{\boldsymbol{C}^{\bpi}_{j,k} - \boldsymbol{C}^{\bpi}_{i,k}}}}} \to 0.
\end{align*}
As a result, we have
\begin{align}
    &\left|\sup_{x\in\mathbb{R}}\mathbb{P}(x< \mathcal{T}^{\sharp} \leq x+\delta) - \sup_{x\in\mathbb{R}}\mathbb{P}\left(x< \max_{1\leq k\leq n-1}|Z_k| \leq x+\delta\right)\right| \nonumber \\
    \leq &\sup_{x\in\mathbb{R}}\left|\mathbb{P}(x< \mathcal{T}^{\sharp} \leq x+\delta) - \mathbb{P}\left(x< \max_{1\leq k\leq n-1}|Z_k| \leq x+\delta\right)\right|   \nonumber \\
    \leq & 2 \sup_{x\in\mathbb{R}}\left|\mathbb{P}(\mathcal{T}^{\sharp}\leq x) - \mathbb{P}\left(\max_{1\leq k\leq n-1}|Z_k|\leq x\right)\right|\to 0.\label{exa2proofeq4}
\end{align}
On the other hand, by \cite[Theorem 3]{chernozhukov2015comparison} we know that
\begin{align}
    \sup_{x\in\mathbb{R}}\mathbb{P}\left(x< \max_{1\leq k\leq n-1}|Z_k| \leq x+\delta\right)\lesssim \delta\sqrt{\log n} \to 0. \label{exa2proofeq5}
\end{align}
Combine \eqref{exa2proofeq4} with \eqref{exa2proofeq5} we know that
\begin{align}
    \sup_{x\in\mathbb{R}}\mathbb{P}(x< \mathcal{T}^{\sharp} \leq x+\delta)\to 0. \label{exa2proofeq6}
\end{align}
Plugging \eqref{exa2proofeq3} and \eqref{exa2proofeq6} in \eqref{exa2proofeq2} we prove \eqref{eq:tsharp_close}. This, along with 
\eqref{exa2proofeq1} proves that 
\begin{align*}
     \left|\mathbb{P}(\mathcal{T}>c_{1-\alpha})-\alpha\right|\to 0,
\end{align*}
concluding our proof.

\end{proof}

\end{document}